\documentclass[12pt]{amsart}
\usepackage[dvips]{graphicx,color}
\usepackage{amscd,amssymb,amsmath,amsthm,epsfig,xypic}
\usepackage{psfrag}
\usepackage[mathscr]{euscript}
\input xy
\xyoption{all}
\DeclareFontFamily{OT1}{rsfs}{}
\DeclareFontShape{OT1}{rsfs}{n}{it}{<-> rsfs10}{}
\DeclareMathAlphabet{\curly}{OT1}{rsfs}{n}{it}

\newtheorem{thm}{Theorem}
\newtheorem{theorem}[thm]{Theorem}
\newtheorem{cor}[thm]{Corollary}
\newtheorem{lem}[thm]{Lemma}
\newtheorem{lemma}[thm]{Lemma}
\newtheorem{prop}[thm]{Proposition}

\DeclareMathOperator{\red}{red}
\DeclareMathOperator{\ev}{ev}
\DeclareMathOperator{\vir}{vir}
\newcommand{\beq}{\begin{equation}}
\newcommand{\eeq}{\end{equation}}

\newcommand\B{\curly B}
\newcommand\tX{\,\widetilde{\!\mathcal X}}
\newcommand{\proj}{\mathbb{P}}
\newcommand{\Z}{\mathbb{Z}}

\newcommand{\rarr}{\rightarrow}
\newcommand{\oh}{{\mathcal{O}}}
\newcommand{\com}{\mathbb{C}}
\newcommand{\Q}{\mathbb{Q}}

\newcommand{\lan}{\langle}
\newcommand{\ran}{\rangle}

\newcommand\C{\mathbb C}

\newcommand{\bZ}{{\mathsf{Z}}}
\newcommand{\bF}{{\mathsf{F}}}
\newcommand{\E}{\mathcal E}

\newcommand\I{\curly I}

\renewcommand\O{\mathcal O}
\newcommand\PP{\mathbb P}
\newcommand\FF{\mathbb F}
\newcommand\II{\mathbb I\udot}
\newcommand\BI{\overline{\,\mathbb I\,}\!\udot}
\newcommand\bI{\overline{I\,}\!\udot}
\newcommand\LL{\mathbb L}
\renewcommand\P{\mathcal P}

\newcommand{\hodge}{\mathbb{E}}

\newcommand{\QMod}{\mathrm{QMod}}

\newcommand{\Mbar}{\overline{M}}
\newcommand{\QQ}{\mathbb{Q}}
\newcommand{\CC}{\mathbb{C}}
\newcommand{\ZZ}{\mathbb{Z}}
\newcommand{\EE}{E\udot}
\newcommand{\EER}{E^{\bullet\mathrm{red}}}
\newcommand{\LLL}{\mathbb{L}}
\newcommand{\calO}{\mathcal{O}}
\newcommand{\calS}{\mathcal{S}}

\newcommand{\rt}[1]{\stackrel{#1\,}{\rightarrow}}
\newcommand{\Rt}[1]{\stackrel{#1\,}{\longrightarrow}}
\newcommand\udot{^{\scriptscriptstyle\bullet}}
\newcommand\ldot{_{\scriptscriptstyle\bullet}}
\newcommand\To{\longrightarrow}
\newcommand\into{\hookrightarrow}
\newcommand\Into{\ar@{^{ (}->}[r]}

\newcommand\ip{\,\lrcorner\,}

\renewcommand\_{^{}_}
\newcommand\take{\backslash}

\newfont{\bigtimesfont}{cmsy10 scaled \magstep5}
\newcommand{\bigtimes}{\mathop{\lower0.9ex\hbox{\bigtimesfont\symbol2}}}

\newcommand\At{\operatorname{At}}

\newcommand\tr{\operatorname{tr}}

\newcommand\id{\operatorname{id}}

\newcommand\Hom{\operatorname{Hom}}
\renewcommand\hom{\curly H\!om}
\newcommand\Ext{\operatorname{Ext}}
\newcommand\ext{\curly Ext}

\newcommand\Spec{\operatorname{Spec}}

\newcommand\Crit{\operatorname{Crit}}

\newcommand\X{\mathcal X}
\newcommand\bX{\,\overline{\!\mathcal X}}
\newcommand\oX{\,\overline{\!X}}

\theoremstyle{remark}
\newtheorem*{rem}{Remark}
\newcommand{\md}{\overline{M}}

\begin{document}

\title{Curves on $K3$ surfaces and modular forms}
\author[Maulik, Pandharipande, Thomas, Pixton]{D. Maulik, R. Pandharipande, and
R. P. Thomas, \\ with an appendix by A. Pixton}
\date{January 2010}
\maketitle

\begin{abstract}
We study the virtual geometry of the moduli spaces
of curves and sheaves on $K3$ surfaces in primitive classes.
Equivalences relating the reduced Gromov-Witten invariants
of $K3$ surfaces 
to characteristic numbers of 
 stable pairs moduli spaces are proven. 
As a consequence, we prove
the Katz-Klemm-Vafa conjecture evaluating $\lambda_g$ integrals (in 
all genera)  
in terms of explicit modular forms. Indeed, all $K3$
invariants in primitive classes are shown to be
governed by modular forms.

The method of proof is by degeneration to elliptically
fibered rational surfaces. New formulas relating reduced
virtual classes on $K3$ surfaces to standard virtual
classes after degeneration are needed for both
maps and sheaves. We also prove a Gromov-Witten/Pairs
correspondence for toric 3-folds.

Our approach uses a result of Kiem and Li  to produce
reduced classes. In Appendix \ref{oldnew}, we answer a number of questions about the relationship
between the Kiem-Li approach, traditional virtual cycles, and symmetric
obstruction theories.

The interplay between the boundary geometry of the moduli spaces of
curves, $K3$ surfaces, and modular forms is
explored in Appendix \ref{Pixton} by A. Pixton.
\end{abstract}

\setcounter{tocdepth}{1}
\tableofcontents
\thispagestyle{empty}

\baselineskip=15pt

\section*{Introduction}

\subsection{Stable maps and reduced classes}
Let $S$ be a complex algebraic $K3$ surface, and let
$\beta \in H_2(S, \ZZ)$ be a nonzero effective curve class.
The moduli space $\Mbar_g(S,\beta)$ of stable maps from connected genus $g$
curves to $S$ representing $\beta$
has expected dimension
$$\dim_{\CC}^{\vir}(\Mbar_g(S,\beta)) 
= \int_\beta c_1(S) +(1-g)\big(\dim_{\CC}(S)-3\big) = g- 1.$$
However, via the holomorphic symplectic form
on $S$, the standard obstruction theory for $\overline{M}_g(S,\beta)$
admits a trivial quotient. 
As a result,
$$[\Mbar_g(S,\beta)]^{\vir}=0 \,. $$
The vanishing reflects the deformation invariance
of Gromov-Witten theory: $S$
admits deformations for which $\beta$ is not of type
$(1,1)$ and thus not represented by  holomorphic curves.

A {\em reduced}
obstruction theory, obtained by removing
 the trivial factor, yields a reduced virtual class\footnote{
See \cite{brl,gwnl} for foundational discussions.}
$$[\Mbar_{g}(S,\beta)]^{\red} \in A_{g}(\Mbar_g(S,\beta),\QQ)$$
of dimension $g$.
A rich Gromov-Witten theory is obtained by 
integrating codimension $g$ tautological 
classes on $\Mbar_{g}(S,\beta)$ 
against $[\Mbar_{g}(S,\beta)]^{\red}$. Such
integrals
are invariant with respect to deformations of $S$ for 
which the class $\beta$ remains of type $(1,1)$.

The class $\beta\in H_2(S,\mathbb{Z})$ is {\em primitive} if
$\beta$ is not divisible.{\footnote{Primitive implies nonzero.}}
While the reduced Gromov-Witten theory of $S$ is defined
for all $\beta$, here we primarily study the primitive
case.

\subsection{Hodge classes}
The rank $g$ Hodge bundle,
 $$\mathbb{E} \rightarrow \Mbar_{g}(S,\beta),$$
with fiber $H^0(C,\omega_C)$ over the point 
$[f: C\rightarrow S] \in \Mbar_{g}(S,\beta),$ 
is well defined for all $g$. The Hodge bundle   
is pulled back from the moduli space of
curves
$$\overline{M}_g(S,\beta) \rarr \overline{M}_g$$
when $g$ is at least $2$.
The top Chern class $\lambda_g$ of $\hodge$
is the
most beautiful and well-behaved integrand
in the reduced theory of $S$. Define
\begin{equation} \label{Rgbeta}
R_{g,\beta} = \int_{[\Mbar_g(S,\beta)]^{\red}}(-1)^g \lambda_g\,.
\end{equation}

Let $X$ be a polarized
 Calabi-Yau 3-fold which admits a
$K3$-fibration,
$$\pi: X \rightarrow \mathbb{P}^1\,. $$
Such a fibration determines a map of the base $\mathbb{P}^1$
to the moduli space of polarized $K3$ surfaces.
The integrals $R_{g,\beta}$ precisely 
relate the Gromov-Witten invariants of $X$ to the
intersection numbers of $\proj^1$ with 
 Noether-Lefschetz divisors in the moduli of $K3$ 
surfaces \cite{gwnl}.

\subsection{Katz-Klemm-Vafa conjecture}\label{kkkcon}
Let $\beta \in H_2(S,{\mathbb Z})$ be a primitive effective curve
class. 
 The Gromov-Witten partition function{\footnote{The partition function
only contains connected contributions. The reduced class suppresses
contributions from
stable maps with disconnected domains.}} for  $\beta$ is
$$\mathsf{Z}_{GW,\beta} = \sum_{g = 0}^{\infty}  R_{g,\beta}\ u^{2g-2}\,.$$
The BPS counts $r_{g,\beta}$ are uniquely defined by
$$\mathsf{Z}_{GW,\beta} = \sum_{g=0}^{\infty} r_{g,\beta}\ u^{2g-2} 
\left(\frac{\sin(u/2)}{u/2}\right)^{2g-2} \,.$$
By deformation invariance{\footnote{
The moduli space of quasi-polarized $K3$ surfaces with
$\lan \beta,\beta \ran =2h-2$ is connected
(and, in fact, is a ball quotient), see \cite{dolga}.}}, 
both $R_{g,\beta}$ and $r_{g,\beta}$ depend only
upon the norm 
$\langle\beta,\beta\rangle$. 
We write $R_{g,h}$ and 
 $r_{g,h}$ for $R_{g,\beta}$ and $r_{g,\beta}$ respectively  when
$\langle \beta,\beta\rangle=2h-2$.

The evaluation of $r_{g,h}$
in terms of modular forms  was conjectured  by S. Katz, A. Klemm, and
C. Vafa \cite{kkv}. The Fourier expansion of the
discriminant modular form is
$$\Delta(q)= q \prod_{n=1}^\infty (1-q^n)^{24}\,.$$
Define the series 
$$\Delta(y,q) = q \prod_{n=1}^{\infty}(1-q^n)^{20}(1-yq^n)^{2}
(1-y^{-1}q^n)^2
$$
where $\Delta(1,q)=\Delta(q)$. Our first result is
the proof of the Katz-Klemm-Vafa conjecture.

\begin{thm} \label{ooo}
The invariants $r_{g,h}$ for primitive curve classes are determined by
$$
\sum_{g=0}^{\infty}\sum_{h=0}^{\infty} (-1)^{g}r_{g,h} \left(
\sqrt{y} - \frac{1}{\sqrt{y}}\right)^{2g} q^{h-1}
= \frac{1}{\Delta(y,q)}\,.
$$
\end{thm}

By
Theorem \ref{ooo}, the invariants $r_{g,h}$ are integers.
The formula  may also be directly written 
for the integrals $R_{g,h}$.
For $n\geq 1$, let $E_{2n}$ be the Eisenstein series
$$E_{2n}(q) = 1 - \frac{4n}{B_{2n}}\sum_{k\geq 1} \frac{k^{2n-1}q^{k}}{1-q^{k}}\ ,$$
where $B_{2n}$ is the corresponding  Bernoulli number.

\begin{cor}\label{ttt}
For primitive curve classes,
$$\sum_{g = 0}^\infty \sum_{h = 0}^\infty R_{g,h}\ u^{2g-2}q^{h-1} 
= \frac{1}{u^{2} \Delta(q)}\cdot
\mathrm{exp}\left(\sum_{g=1}^\infty u^{2g} \frac{|B_{2g}|}{g\cdot(2g)!} 
E_{2g}(q)\right)\ 
.$$
\end{cor}

Theorem 1 specializes in genus 0 to the rational curve counts
on
$K3$ surfaces
predicted by S.-T. Yau and E. Zaslow \cite{yauz}.
The Yau-Zaslow formula was proven for primitive classes
in \cite{beu,brl}.

Of course, the integrals $R_{g,\beta}$ may also be considered
in the non-primitive case. A complete conjecture is explained
in \cite{gwnl,clay} based on \cite{kkv}. While the genus 0 integrals $R_{0,\beta}$ have
been calculated for
all classes $\beta$ in \cite{gwyz}, new methods appear to be
required in higher genus.

\subsection{Descendents}
Let $\beta\in H_2(S,\mathbb{Z})$ be a primitive
effective curve class.
The moduli space of stable maps from connected genus $g$ curves with $r$ ordered
marked points $\overline{M}_{g,r}(S,\beta)$ comes with $r$ evaluation maps 
$$\text{ev}_i: \overline{M}_{g,r}(S,\beta) \rightarrow S\,.$$
Pulling back cohomology classes on $S$ 
via $\text{ev}_i$
gives \emph{primary} classes on
$\overline{M}_{g,r}(S,\beta)$.
{\em Descendent} classes are obtained from the
Chern classes of the cotangent lines
$$ L_i \rarr \overline{M}_{g,r}(S,\beta)$$
at the marked points.

Let $\gamma_1, \ldots, \gamma_r \in H^*(S,{\mathbb{Z}})$, and
let $$\psi_i = c_1(L_i) \in H^2(\overline{M}_{g,r}(S,\beta),\mathbb{Q}).$$
The insertion $\tau_k(\gamma)$ corresponds 
to the class $\psi_i^k \cup \text{ev}_i^*(\gamma)$ on the moduli space
of maps. 
Let
\begin{equation} \Big\langle \tau_{k_1}(\gamma_{1}) \cdots \label{fvf12}
\tau_{k_r}(\gamma_{r})\Big\rangle^{S}_{g,\beta} 
= \int_{[\overline{M}_{g,r}(S,\beta)]^{\red}} 
\prod_{i=1}^r \psi_i^{k_i} \cup \text{ev}_i^*(\gamma_{i})
\end{equation}
denote the reduced descendent
Gromov-Witten invariants. 
By convention, the descendent
vanishes if the degree of the integrand does not
match the dimension $g+r$ of the reduced virtual class.

If only descendents of classes in $H^0(S,\mathbb{Z})$ and
$H^4(S,\mathbb{Z})$ appear in \eqref{fvf12},  the bracket
for primitive $\beta$ depends 
only upon the norm 
$$\langle \beta,\beta\rangle=2h-2\,$$
by deformation invariance.
Since the classes in $H^2(S,\mathbb{Z})$ are {\em not}
monodromy invariant, the bracket \eqref{fvf12} may
depend upon $\beta$ if descendents of $H^2(S,\mathbb{Z})$
are present.
When possible, we will replace the subscript $\beta$ of the descendent
bracket 
by $h$.

\subsection{Point insertions}

The evaluation of Theorem \ref{ooo} extends naturally to the
integrals 
$$\Big\langle (-1)^{g-k} \lambda_{g-k} \tau_0(\mathsf{p})^{k}\Big\rangle^{S}_{g, h}=
\int_{[\Mbar_{g,k}(S,\beta)]^{\red}}(-1)^{g-k} \lambda_{g-k}
\prod_{i=1}^k \text{ev}_i^*(\mathsf{p})
\  ,
$$
where $\lambda_i$ is the $i^{th}$ Chern class of the Hodge
bundle $\hodge$ and
$\mathsf{p} \in H^4(S,\mathbb{Z})$ is the point class.


\begin{thm}\label{kkvpoints} For primitive classes on $K3$ surfaces, we have
\begin{multline*}
 \sum_{g= 0}^\infty \sum_{h = 0}^\infty
\Big\langle (-1)^{g-k} \lambda_{g-k} \tau_0(\mathsf{p})^{k}\Big\rangle^{S}_{g, h} u^{2g-2}
q^{h-1}
=\\ 
\frac{1}{u^{2} \Delta(q)}\cdot
\mathrm{exp}\left(\sum_{g=1}^\infty u^{2g} \frac{|B_{2g}|}{g(2g)!} 
E_{2g}(q)\right)\\
\cdot
 \Big( \sum_{m= 1}^\infty q^{m} \sum_{d | m} \frac{m}{d} 
\big(2 \sin ({du}/{2})\big)^{2}\Big)^{k} \,.
\end{multline*}
\end{thm}
The last factor is related to the point insertions. In the $k=0$ case,
when no points are inserted, Theorem \ref{kkvpoints} specializes to
 Theorem \ref{ooo} by Corollary \ref{ttt}.

\subsection{Quasimodular forms} \label{qmforms}
The ring of {\em quasimodular forms} with possible poles at
$q=0$ is the algebra generated by the Eisenstein series
$E_2$ over the ring of $\mathbf{SL}(2,\mathbb{Z})$ 
modular forms with possible poles at $q=0$. 
The ring of quasimodular forms is closed under $q\frac{d}{dq}$.
See \cite{bghz} for a basic treatment.

By deformation invariance,
the full descendent theory of algebraic $K3$ surfaces is
captured by elliptically fibered $K3$ surfaces.
Let $S$ be an elliptically fibered $K3$ surface with 
section. Let
$$\mathsf{s}, \mathsf{f} \in H_2(S, \mathbb{Z})$$
denote the section and fiber classes.
A descendent potential function for the
reduced theory of $K3$ surfaces in primitive classes is defined
by
$$\bF^{S}_{g}\big(\tau_{k_1}(\gamma_{l_1}) \cdots
\tau_{k_r}(\gamma_{l_r})\big)=
\sum_{n=0}^\infty 
\Big\langle \tau_{k_1}(\gamma_{l_1}) \cdots
\tau_{k_r}(\gamma_{l_r})\Big\rangle^{S}_{g,\mathbf{s}+ h \mathbf{f} } 
\ q^{h-1}
$$
for $g\geq 0$.  
For arbitrary insertions, we prove the following result.

\begin{thm} \label{qqq}
$\bF_{g}^S\big(\tau_{k_1}(\gamma_{1})
\cdots
\tau_{k_r}(\gamma_{r})\big)
$
is the Fourier expansion in $q$
of a quasimodular form with pole at $q=0$ of
order at most $1$.
\end{thm}

The simplest of the $K3$ series is the count
of genus $g$ curves passing through $g$ points,
\begin{equation}
\label{htth}
\bF_{g}^S ( \tau_0(p)^g ) 
= 
\frac{1}{\Delta(q)}
\cdot
\left( -\frac{1}{24}\ 
q\frac{d}{dq} E_2\right)^g \ ,
\end{equation}
first calculated{\footnote{Our indexing conventions
differ slightly from those adopted in \cite{brl}.}} 
by J. Bryan and C. Leung \cite{brl}.
Formula \eqref{htth} is also a specialization of
Theorem 3.

In the non-primitive case, we conjecture the genus $g$ reduced descendent
potential
to be a quasimodular form of higher level.
A precise statement is made in Section \ref{ccllyy}.

\subsection{Stable pairs on $K3$ surfaces}
We will relate the reduced Gromov-Witten invariants
of $K3$ surfaces to integrals over the moduli spaces of
sheaves on $K3$ surfaces. 

Let $S$ be a $K3$ surface. 
A {\em pair} $(F,s)$ consists of a sheaf $F$ on $S$  supported
in dimension 1 together with a section $s\in H^0(S,F)$.  A pair $(F,s)$
 is {\em stable} if 
\begin{enumerate}
\item[(i)]
the sheaf $F$ is {pure}, 
\item[(ii)] the section $\oh_S \stackrel{s}{\rightarrow} F$ has 0-dimensional
cokernel.
\end{enumerate}
Purity here simply means
 every nonzero subsheaf of $F$ has support of dimension 1.
As a consequence,
 the scheme theoretic support $C\subset S$ of  $F$ is a 
curve.
The discrete invariants of a stable pair are the 
holomorphic Euler characteristic
$\chi(F)\in \mathbb{Z}$ and the class{\footnote{$[F]$ 
is the sum of the classes of the irreducible
1-dimensional curves on which $F$ is supported weighted
by the generic length of $F$ on the
curve. Equivalently,
$[F]=c_1(F)$.
}}
 $[F]\in H_2(S,\mathbb{Z})$.

Let $\beta\in H_2(S,\mathbb{Z})$ be a nonzero effective
curve class.
Let $P_n(S,\beta)$ be the moduli space of stable pairs satisfying
$$\chi(F)=n, \ \ [F]=\beta.$$
After appropriate choices \cite{pt1}, pair stability coincides
with stability arising from geometric
invariant theory in Le Potier's study \cite{LeP}. Hence, the moduli 
space $P_n(S,\beta)$ is a projective scheme.

The class $\beta$ is {\em irreducible} if $\beta$ is
not a sum of two nonzero effective curve classes.{\footnote{An irreducible class
is primitive. By deforming  $S$,  every primitive curve class $\beta\in H_2(S,\mathbb{Z})$
can be made irreducible.}
A basic result proven in \cite{ky,pt3} is the following.

\begin{prop} \label{hhtt} If $\beta$ is irreducible, $P_n(S,\beta)$ is
 nonsingular of dimension $n+\langle \beta,\beta \rangle +1$.
\end{prop}

When studying stable pairs, we will often assume
$\beta$ is irreducible.
In the irreducible case, $P_n(S,\beta)$ depends, up to
deformation equivalence, only  upon the norm of $\beta$.
We will use the notation $P_n(S,h)$
when $\langle \beta,\beta \rangle=2h-2$.

\subsection{Euler characteristic} \label{jjje}
Let $\beta\in H_2(S,\mathbb{Z})$ be an irreducible effective curve class
with norm $\langle \beta,\beta \rangle=2h-2$.

Let $\Omega_{P}$ be the cotangent bundle of
the moduli space $P_n(S,h)$.
Define the partition function
\begin{eqnarray*}
\bZ_{h}^{P}(y) & =  &
\sum_n \int_{P_n(S,h)} c_{n+2h-1}(\Omega_P)\ y^n \\
& = &
\sum_{n} (-1)^{n+2h-1} \mathsf{e}(P_n(S,h))\  y^n.
\end{eqnarray*}
Here, $\mathsf{e}$ denotes the topological Euler
characteristic.
We have written the stable pairs partition
function in the variable $y$ instead of the traditional $q$
since the latter will be reserved for the Fourier
expansions of modular forms.{\footnote{The conflicting uses of
$q$ seem impossible to avoid. The possibilities for confusion are
great.}} 
Since $P_n(S,h)$ is empty if $n<1-h$, we see $\mathsf{Z}^P_h$ is
a Laurent series in $y$.

The topological Euler characteristics of 
$P_n(S,h)$ have been calculated by T. Kawai and K. Yoshioka.
By Theorem 5.80 of \cite{ky},
\begin{equation*}
\sum_{h=0}^\infty \sum_{n=1-h}^\infty  
\mathsf{e}(P_n(S,h))\  y^n q^{h-1} = \\
\left(\sqrt{y}-\frac{1}{\sqrt{y}}\right)^{-2}
\frac{1}{\Delta(y,q)} \,.
\end{equation*}
We require the signed Euler characteristics,
$$
\sum_{h=0}^\infty \bZ^P_{h}(y) \ q^{h-1} = 
\sum_{h=0}^\infty \sum_{n=1-h}^\infty (-1)^{n+2h-1} 
\mathsf{e}(P_n(S,h))\  y^n q^{h-1}. 
$$
Therefore,
$\sum_{h=0}^\infty \bZ^P_{h}(y) \ q^{h-1}$ equals
\begin{equation}\label{jjy}
-\left(\sqrt{-y}-\frac{1}{\sqrt{-y}}\right)^{-2}
\frac{1}{\Delta(-y,q)}
\,.
\end{equation}

\subsection{Correspondence}\label{kkkcon2}
To prove Theorem \ref{ooo}, we formulate and prove
a Gromov-Witten/Pairs correspondence in the setting
of reduced classes.

Let $\beta\in H_2(S,\mathbb{Z})$ be an irreducible effective curve class.
We write the Gromov-Witten partition
function as
$$\bZ^{GW}_h(u)  =
\sum_{g= 0}^\infty  r_{g,h}\ u^{2g-2} 
\left( \frac{\sin
({u/2})}{u/2}\right)^{2g-2}\,.$$
Our Gromov-Witten/Pairs correspondence for the reduced theories of
the 3-fold $S\times\C$ implies\footnote{The standard Gromov-Witten/Pairs conjecture
of \cite{pt1} applies to virtual classes for 3-fold theories. Our 
 analogue is for reduced classes (in a 
$\com^*$-equivariant context).
The $\com^*$-equivariant theories of $S\times \com$ are
equivalent to the theories of $S$.
}
\begin{equation}\label{gwp}
\bZ^{GW}_h(u)  =  \bZ^{P}_h(y) 
\end{equation}
after the substitution $-e^{iu}=y$.
Together with the Euler characteristic
calculation \eqref{jjy},
the correspondence \eqref{gwp}
immediately yields Theorem \ref{ooo}.

To complete the proof of Theorem \ref{ooo}, we 
must establish the reduced Gromov-Witten/Pairs correspondence
for $S\times \com$.
There are two main ideas in the argument:
\begin{enumerate}
\item[(i)] Let $R$ be the rational elliptic surface
obtained by blowing-up the base locus of
a pencil of cubics in $\proj^2$. Let $E\subset R$
be a nonsingular member of the pencil.
Using special degenerations of elliptically fibered
$K3$ surfaces $S$ to unions of rational elliptic surfaces $R\cup_E R$,
we prove a new formula relating the reduced
virtual classes of $S\times \com$ to the standard 
virtual classes of $R\times \com$.
We prove the formula separately for stable maps and stable pairs.

\begin{figure}[h]
\begin{center}
\includegraphics{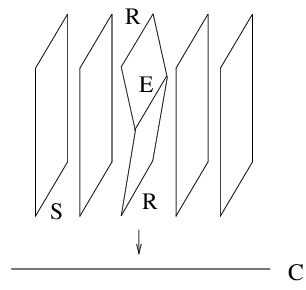} \vspace{-4mm}
\caption{A degeneration of $K3$ surfaces \label{fig}}
\end{center}
\end{figure}

\item[(ii)]
Since $R$ is isomorphic to $\proj^2$ blown-up at
$9$ points, $R\times \com$ is deformation equivalent to a toric 3-fold.
We prove a Gromov-Witten/Pairs correspondence for 
toric 3-folds following \cite{moop}.
\end{enumerate}
Together (i) and (ii) yield the correspondence \eqref{gwp}
and complete the proof of Theorem \ref{ooo}.

We have no direct approach to the $\lambda_g$ integrals
$R_{g,\beta}$ on $\overline{M}_{g}(S,\beta)$.
The moduli space of stable maps has contracted components
and subtle virtual contributions. The nonsingularity of
 the corresponding moduli spaces
of stable pairs is remarkable.  
Theorem \ref{ooo} provides a model use of the Gromov-Witten/Pairs
correspondence.

Part (i) constitutes the technical heart of the paper.
The primitivity of $\beta\in H_2(S,\mathbb{Z})$ is
crucial. 
In Section \ref{npd}, we state a
degeneration formula
in the non-primitive case which leads to much more
subtle invariants of $R$.
Unfortunately, 
the toric correspondence (ii)
is not sufficient to conclude a
Gromov-Witten/Pairs correspondence for non-primitive
classes $\beta\in H_2(S,\mathbb{Z})$.
The non-primitive degeneration formula will
be pursued in a sequel \cite{mpt2}.

\subsection{Point insertions for stable pairs}
Let $\beta\in H_2(S,\mathbb{Z})$ be an irreducible effective
curve class of norm $\langle \beta,\beta\rangle =2h-2$.

The linear system of curves of class $\beta$ is $h$-dimensional.
Let
\begin{equation} \label{rho}
\rho: P_n(S,h) \rightarrow \proj^h
\end{equation}
be the canonical morphism obtained by sending $(F,s)$
to the support of $F$. 
A point incidence condition for stable pairs corresponds
to the  $\rho$ pull-back of a hyperplane $H\subset \proj^h$.
The integral for stable pairs associated to
$k$ point conditions is defined by
$$ \mathsf{C}^k_{n,h}=
\int_{P_n(S,h)} c_{n+2h-1-k} (\Omega_P) \cup \rho^*(H^k)
\,.$$
By Bertini, the subvariety
 $$P^k_n(S,h) = \rho^{-1}(H_1) \cap \ldots \cap \rho^{-1}(H_k) \subset P_n(S,h)$$
is nonsingular of dimension $n+2h-1-k$ for generic hyperplanes.
Using Gauss-Bonnet,
the 
Euler characteristics of the spaces $P^k_n(S,h)$ are expressible in terms of
the integrals $\mathsf{C}^k_{n,h}$ by the formula
\begin{equation}\label{bttr}
\mathsf{e}\left(P^k_n(S,h)\right) = (-1)^{n+2h-1-k} \sum_{i=0}^{n+2h-1-k}  
(-1)^i \binom{i+k-1}{k-1} \mathsf{C}^{k+i}_{n,h}\,.
\end{equation}
In fact, equation \eqref{bttr} may be easily 
inverted to express $\mathsf{C}^k_{n,h}$
in terms of the Euler characteristics.

\begin{thm} \label{pst} The point conditions for irreducible classes
on $K3$ surfaces are evaluated by
\begin{multline*}
\sum_{n} \sum_{h=0}^\infty
\mathsf{C}^k_{n,h}  (-y)^n q^{h-1}
= \\
\frac{(-1)^{k+1}}{\Delta(y,q)} \cdot
\frac{ \Big( \sum_{m= 1}^\infty q^{m} \sum_{d | m} \frac{m}{d} 
\left( y^d -2+ y^{-d}\right)\Big)^{k}}
{y-2+y^{-1}} 
 \,.
\end{multline*}
\end{thm}

Point conditions in the reduced
Gromov-Witten theory of $S$ are evaluated by Theorem \ref{kkvpoints}.
We derive Theorem \ref{kkvpoints} from Theorem \ref{ooo} using
degeneration and exact Gromov-Witten calculations for Hodge integrals.
Theorem \ref{kkvpoints} then implies 
Theorem \ref{pst} by the equivariant Gromov-Witten/Pairs
correspondence for $S\times \com$.

We do not know a direct approach along the lines of \cite{ky}
for determining the integrals $\mathsf{C}^k_{n,h}$ or the
Euler characteristics of $P^k_{n}(S,h)$.

\subsection{Plan of the paper}
We start, in Section \ref{ggww}, 
with a precise statement of the
Gromov-Witten/Pairs correspondence for the
reduced theory of
 $S\times \com$
with primary insertions, leaving many of the proofs for later Sections.
Elliptically fibered $K3$ surfaces are reviewed in Section \ref{ellk3}.
 The
degeneration formulas
in terms of the standard virtual classes of the
rational elliptic surface
are proven in Section \ref{spd} for stable pairs and in Section \ref{gwd} for Gromov-Witten theory. We give full details for stable pairs and a briefer account for the more standard
Gromov-Witten theory.

The Gromov-Witten/Pairs
correspondence for toric 3-folds is established
 in Section \ref{gwptor},
completing the proof of Theorem \ref{ooo}. Theorems \ref{kkvpoints}
and \ref{pst} are proven in Section \ref{pint}.
The quasimodularity of Theorem \ref{qqq} is obtained in Section \ref{qmod}
from a boundary
induction in the tautological ring of the moduli space of
curves using the strong form of Getzler-Ionel vanishing proven in 
\cite{fpm}. 

Our approach uses a result of Kiem-Li \cite{KL} to construct
reduced classes\footnote{A more detailed account of the
deformation theory in \cite{KL} has appeared very recently \cite{KL2}.} . In Appendix \ref{oldnew}, we compare the
Kiem-Li method to standard virtual cycle techniques. 
The inquiry leads naturally to
a counterexample
to a question of 
Behrend and Fantechi concerning symmetric obstruction theories
that is explained in Section \ref{oldnewl}.

In Appendix \ref{Pixton}, by A. Pixton \cite{pix}, the interplay between Theorem 1
and boundary expressions for $\lambda_g$ in low genus are explored.

\subsection*{Acknowledgements}
Much of the work presented here was 
completed at MSRI in 2009 
during a program on modern moduli in algebraic geometry. 
We thank the organizers
for creating a stimulating environment.
We thank  J. Bryan,
B. Conrad, C. Faber, H. Flenner, D. Huybrechts,
B. Bakker, D. Joyce, A. Klemm, A. Marian, D. Oprea,  
E. Scheidegger and S. Yang for may related conversations.
Correspondence with A. Boocher and
D. van Straaten lead to the example in Appendix
 \ref{oldnew}. We are particularly grateful
to Jun Li for discussions and an advanced copy of \cite{LiWu}.

D.M. was partially supported by a Clay research fellowship. 
R.P. was partially supported by DMS-0500187
and the Clay Institute. R.T. was partially supported by an
EPSRC programme grant.

\section{Reduced Gromov-Witten/Pairs correspondence}
\label{ggww}
\subsection{Stable maps}
Let $S$ be a complex projective $K3$ surface, and let $\beta\in H_2(S,\mathbb{Z})$
be a primitive effective curve class. 
Consider the noncompact Calabi-Yau 3-fold
$$X =S \times \CC$$
equipped with the $\CC^*$-action defined by scaling the second factor.  
Let $$\iota: S  \rightarrow X$$ denote the inclusion given by the identification 
$S = S \times \{0\}$.

Let $\Mbar_{g}(X, \iota_{*}\beta)$ be the moduli space of connected genus $g$ stable
maps  to $X$ representing the
class $\iota_*\beta$.  
Since $X$ is a Calabi-Yau $3$-fold, 
the moduli space has expected dimension $0$ 
with respect to the standard obstruction theory.
Since $S$ has a holomorphic symplectic form, $\Mbar_{g}(X, \iota_{*}\beta)$
 admits a reduced obstruction theory and reduced virtual class, 
$$[\Mbar_{g}(X, \iota_{*}\beta)]^{\red} \in A_1(\Mbar_{g}(X, \iota_{*}\beta),
\mathbb{Q})\,. $$
The construction of the reduced theory exactly follows
Section 2.2 of \cite{gwnl}.
Although $\Mbar_g(X,\iota_*\beta)$ is not 
compact, the $\CC^*$-fixed locus 
$$\Mbar_{g}(X, \iota_{*}\beta)^{\com^*}\subset \Mbar_{g}(X, \iota_{*}\beta)$$
is compact, so we can consider the reduced residue invariants{\footnote{See
\cite{BryP} for a discussion of residue Gromov-Witten theories.}} 
$$N_{g,\beta} = \int_{[\Mbar_{g}(X,\iota_*\beta)]^{\mathrm{red}}} 1 \in \QQ(t)\,.$$
Here, $t$ is the first Chern class of the standard representation of
$\com^*$ and the 
 generator of $H^*_{\com^*}(\bullet)$,
the $\com^*$-equivariant cohomology of a point. The relationship between
 the residue invariants of $S\times\C$ and the invariants \eqref{Rgbeta} of $S$ is the following. 
\begin{lemma}\label{gth}
$N_{g,\beta} = \frac{1}{t} R_{g,\beta} .$
\end{lemma}
\begin{proof} The result is a direct consequence of the
virtual localization formula of \cite{GP},
\begin{eqnarray*}
N_{g,\beta} & = & 
\int_{[\Mbar_{g}(X,\iota_*\beta)^{\com^*}]^{\red}} 
\frac{1}{e(\text{Nor}^{vir})} \\
& = &  
\int_{[\Mbar_{g}(S,\beta)]^{\red}} \frac{t^g-\lambda_1t^{g-1}+ \lambda_2t^{g-2} -\ldots
+(-1)^g \lambda_g}{t} \\
& = &
\frac{1}{t}{R_{g,\beta}} \,.
\end{eqnarray*}
The first equality is by localization. The denominator
on the right is the equivariant Euler class of the
virtual normal bundle. Over a stable map $[f\colon C\to S]$,  
the virtual normal bundle has
fiber $$H^0(C,f^*N)-H^1(C,f^*N)\ ,$$
 where $N$ is the normal bundle to $S$ in $X$. Since
$N\cong t$, we have $$\text{Nor}^{vir}\cong t-\mathbb E^\vee\otimes t \ ,$$ 
from which the above formula follows.
\end{proof}

If $\beta$ is irreducible, then $\overline{M}_{g}(X,\iota_*(\beta)) =
\overline{M}_g(S,\beta) \times \com$
and the reduced virtual class is pulled back from the projection 
to the first factor. In the irreducible case, Lemma \ref{gth}
is immediate. An alternative
proof of Lemma \ref{gth} for primitive $\beta$ is obtained
by deforming to the irreducible case.

\subsection{Stable pairs}
Let $P_{n}(X, \iota_{*}\beta)$ the moduli space of stable
pairs $(F,s)$ on $X=S\times \C$ with
 $$\chi(F)=n, \ \ [F]=\beta.$$
We will construct in Section \ref{spred}  
a reduced virtual class in dimension $1$, 
$$[P_{n}(X, \iota_{*}\beta)]^{\red} \in A_1(P_{n}(X, \iota_{*}\beta),
\mathbb{Q})\,. $$
Again, we  consider the reduced residue invariants
$$P_{n,\beta} = \int_{[P_{n}(X,\iota_*\beta)]^{\red}} 1 \in \QQ(t)\,.$$

By deformation invariance of the reduced theory, the invariant
$P_{n,\beta}$ can be computed when $\beta$ is 
irreducible.{\footnote{A primitive $(1,1)$-class
$\beta$ on a $K3$ surface $S$ can always be deformed
through curve classes to an irreducible $(1,1)$-class
on another $K3$ surface $S'$.}}
By standard arguments{\footnote{The only subtlety
is to show the deformations of such pairs on
$X$ remain supported scheme-theoretically on the fibers of
the projection $X\rarr \com$. The result
follows from the tangent space analysis of 
Lemma C.7. of \cite{pt3}.}
$$P_{n}(X,\iota_*\beta)= P_n(S,\beta) \times \com$$
in the irreducible case.
By Proposition \ref{hhtt}, $P_n(X,\iota_*\beta)$
is nonsingular of dimension $n + \langle \beta, \beta \rangle+2$.
The obstruction bundle of the standard deformation
theory \cite{HT, pt1} of $P_{n}(X,\iota_*\beta)$ has fiber 
$$\text{Ext}^2(I^\bullet,I^\bullet)_0 \cong 
\text{Ext}^1(I^\bullet, I^\bullet\otimes K_X)^*_0 
$$
over the moduli point of the pair 
$$
I\udot=\{\O_X\Rt{s}F\}\,.
$$
Here,
$K_X$ is the canonical bundle and
the isomorphism is by Serre duality.
Since $\text{Ext}^1(I^\bullet, I^\bullet)_0$ is the tangent
space to $P_n(S,\beta) \times \com$,
the moduli of stable pairs on $X$, and
$K_X^*$ is trivial with the standard representation,
the obstruction bundle is
$$(\Omega_{P_n(S,\beta)} \oplus -t ) \otimes t\  \cong\, (\Omega_P \otimes t)
\oplus \com.$$
The reduced class is obtained by removing the trivial factor
$\com$, as we show in Section \ref{spsym}.

\begin{lemma}\label{gthr}
$P_{n,\beta} = \frac{1}{t} (-1)^{n+\langle \beta,\beta \rangle+1} 
\mathsf{e}(P_n(S,\beta))$.
\end{lemma}
\begin{proof} We calculate the residue of the
top Chern class of the reduced obstruction bundle,
\begin{eqnarray*}
P_{n,\beta} & = & 
\int_{P_{n}(X,\iota_*\beta)} 
e(\Omega_P \otimes t) \\
& = &  
\int_{P_{n}(S,\beta)} \frac{e(\Omega_P)}{t} \\
& = &
\frac{1}{t}(-1)^{n+\langle \beta,\beta \rangle+1} 
  \mathsf{e}(P_n(S,\beta))       \,.
\end{eqnarray*}
The second equality comes from localisation. We have omitted all of the
terms in $e(\Omega_P \otimes t)$ which do not contribute.
\end{proof}

\subsection{Point insertions}
For both theories of $X$,
we can define reduced residue invariants with point insertions.
For Gromov-Witten theory, define{\footnote{We have dropped $X$ from the bracket to simplify
the notation.}}
$$\Big\langle \tau_0(\mathsf{p})^k \Big\rangle^{GW}_{g,\beta}
=\int_{[\Mbar_{g,k}(X,\iota_*\beta)]^{\red}} \prod_{i=1}^k \text{ev}_i^*(\mathsf{p}) \ \in \QQ(t)\ $$
where the evaluation maps are taken to $S$
$$\text{ev}_i: \overline{M}_{g,k} (X, \iota_*(\beta)) \rarr S \ $$
and $\mathsf{p}\in H^4(S,\mathbb{Z})$ is the point class.

For stable pairs, the product $P_n(X,\iota_*\beta) \times X$ is equipped
with a universal sheaf $\mathbb{F}$.
Define operations
$$\tau_0(\mathsf{p}): A_*^{\com^*}( P_n(X,\iota_*\beta)) \rarr A_*^{\com^*}( P_n(X,\iota_*\beta))$$
by the slant product
$$\tau_0(\mathsf{p})( \bullet) = \pi_{P*}\big( \pi^*_S(\mathsf{p}) \cdot \text{ch}_2(\mathbb{F})
\cap \pi_P^*(\bullet)\big)\ ,$$
where 
$\pi_P$ and $\pi_S$ are
the projections of $P_n(X,\iota_*\beta) \times X$
 to the first factor and to
$S$ (via the second factor). 
Notice that $\text{ch}_2(\mathbb F)$ is the pull-back via the map $\rho$ 
of \eqref{rho} of the universal
curve in $S\times\PP^h$.
Define the residue invariants
$$\Big\langle \tau_0(\mathsf{p})^k  \Big\rangle^{P}_{n,\beta}
= \int_{P_n(X,\iota_*\beta)}  
\tau_0(\mathsf{p})^k \Big( [P_n(X,\iota_*\beta)]^{\red} \Big) \ 
\in \QQ(t)\ $$
following 
Section 6.1 of \cite{pt2}.

The calculations of Lemmas \ref{gth} and \ref{gthr} immediately extend to yield the 
following formulas,
\begin{eqnarray*}
\Big\langle \tau_0(\mathsf{p})^k  \Big\rangle^{GW}_{g,\beta} & =& t^{k-1}
\Big\langle (-1)^{g-k} \lambda_{g-k} \tau_0(\mathsf{p})^k\Big\rangle^{S}_{g, \beta}\ , \\ 
\Big\langle \tau_0(\mathsf{p})^k  \Big\rangle^{P}_{n,\beta} & =&
 t^{k-1}\int_{P_n(S,\beta)} c_{n+2h-1-k} (\Omega_P) \cup \rho^*(H^k)\,.
\end{eqnarray*}

\subsection{Correspondence}
The reduced Gromov-Witten/Pairs correspondence is stated in terms of the
generating series
\begin{eqnarray*}
\mathsf{Z}_{\beta}^{GW}\left( \tau_{0}(\mathsf{p})^k, u \right) & =& 
\sum_{g=0}^\infty
\Big\langle \tau_0(\mathsf{p})^k  \Big\rangle^{GW}_{g,\beta} u^{2g-2}\ , \\
\mathsf{Z}_{\beta}^{P}\left( \tau_{0}(\mathsf{p})^k, y \right) & =& 
\sum_{n}
\Big\langle \tau_0(\mathsf{p})^k  \Big\rangle^{P}_{g,\beta} \ y^n \,.
\end{eqnarray*}
The stable pairs series  is a Laurent function in $y$ since $P_n(X,\iota_*\beta)$
is empty for sufficiently negative $n$.
The above partition functions specialize to the 
partition function 
$\mathsf{Z}_{\beta}^{P}$ and 
$\mathsf{Z}_{\beta}^{GW}$ of Sections \ref{jjje} and \ref{kkkcon2}
when $k=0$ and $t=1$.

\begin{theorem}\label{redgwpt}
For primitive
$\beta\in H_2(S,\mathbb{Z})$,
\begin{enumerate}
\item[(i)] $\mathsf{Z}_{\beta}^{P}\left( \tau_{0}(\mathsf{p})^k\right)$
 is a rational function of $y$.
\item[(ii)]
After the variable change $-e^{iu}=y$,  
$$\mathsf{Z}_{\beta}^{GW}\left( \tau_{0}(\mathsf{p})^k\right)  = \mathsf{Z}_{\beta}^{P}
\left( \tau_{0}(\mathsf{p})^k\right)\,.$$
\end{enumerate}
\end{theorem}

Theorem \ref{redgwpt} is not a specialization of the
Gromov-Witten/Pairs correspondence for 3-folds conjectured in \cite{pt1}.
The main difference is the occurrence of the reduced class.
Since the reduced class suppresses contributions from stable
maps with disconnected domains,
the correspondence here may be viewed here
as concerning only {\em connected} curves.
Theorem \ref{redgwpt} will be proven in Sections \ref{ellk3}-\ref{gwptor}.

\section{Elliptically fibered $K3$ surfaces}
\label{ellk3}

\subsection{Elliptic fibrations}
We fix here some notation which will be used throughout the paper.
Let $S$ be an elliptically  fibered $K3$ surface 
\begin{equation}\label{mqq}
\pi: S \rarr \proj^1
\end{equation}
with a section. We assume $\pi$ is smooth 
except for 24 nodal rational fibers.
Let 
 $$\mathsf{s},\mathsf{f} \in H_2(S,\ZZ)$$
 denote the classes of the section and the elliptic fiber.
The intersection pairings are
$$\langle \mathsf{s}, \mathsf{s} \rangle =-2, \ \ \ \langle \mathsf{s}, \mathsf{f} \rangle = 1, \ \ \
\langle \mathsf{f}, \mathsf{f} \rangle = 0 \,.$$  

By deformation invariance, the reduced Gromov-Witten and stable pairs theories for primitive effective classes depend only on the norm
$\langle\beta,\beta\rangle = 2h-2$.
By deformation invariance, we can fully
capture the both theories 
for 
primitive classes on all algebraic $K3$ surfaces 
by studying
$$\beta = \mathsf{s} + h\mathsf{f}$$
on elliptically fibered $K3$ surfaces $S$.

\subsection{Rational elliptic surface} \label{rez}
A rational elliptic surface $R$ is obtained by blowing-up the
9 points of the base locus of a generic pencil of cubics. The pencil
determines a map
$$\pi: R \rarr \proj^1$$
with nonsingular elliptic fibers (except for 12 nodal rational fibers).
Let $D\subset R$ be one of the 9 sections of $\pi$, and let $E\subset R$
be 
a fixed elliptic fiber with distinguished point
\begin{equation} \label{pointp}
p = E \cap D.
\end{equation}

Let $R_1$ and $R_2$
be two copies of a rational elliptic surface $R$.
Let $D_1=D_2$, $E_1=E_2$, and $p_1=p_2$ be identical
choices of the auxiliary data.
A reducible surface
\begin{equation*}
R_{1} \cup_E R_{2}
\end{equation*}
is obtained by attaching $R_1$ and $R_2$
along the respective fibers $E_i$ 
(with the corresponding distinguished points $p_i$ 
identified).  
The singular surface 
is elliptically fibered over a broken rational curve,
\begin{equation}\label{nwq}
R_{1} \cup_E R_{2} \rightarrow
\proj^1 \cup \proj^1\,.
\end{equation}
The fibration \eqref{nwq} has
 a distinguished section $D_1 \cup D_2$.

\subsection{Degeneration} \label{jjyy}
The fibration \eqref{nwq} is a degeneration of 
\eqref{mqq}. More precisely,
there exists a family of fibrations
\begin{equation}
\calS \stackrel{\pi}{\rarr} \mathcal{C}  \rightarrow B \label{rrrt}
\end{equation}
 over a pointed curve $(B, 0)$ 
with the following properties:
\begin{enumerate}
\item[(i)]
$\calS$ is a nonsingular 3-fold, and $\mathcal{C}$ is
a nonsingular surface.
\item[(ii)] $\pi$ has a section.
\item[(iii)]
When specialized to  nonzero $\xi \in B$, we obtain a
nonsingular elliptically fibered $K3$ surface of
the form \eqref{mqq}. 
\item[(iv)]
When specialized to $0\in B$,
we obtain \eqref{nwq}.
\item[(v)] The relative canonical bundle $\omega_{\calS/B}$ is trivial.
\end{enumerate}

Denote by $\epsilon$ the  degenerating family of $K3$ surfaces obtained from
composing \eqref{rrrt},
$$\epsilon: \calS \rarr B\,. $$
Since the section and fiber classes are globally defined by (ii),
the sub-lattice of $H^2(\calS_\xi,\ZZ)$ spanned by 
 $\mathsf{s}$ and $\mathsf{f}$ is fixed by the monodromy of $\epsilon$  around
$0\in B$.

The degenerating family \eqref{rrrt} will play an essential role
in
our proof of Theorem \ref{redgwpt}.

\section{Reduced stable pairs} \label{spd}

\subsection{Definitions} \label{spdef}
Let $S$ be a complex algebraic $K3$ surface, and let
$\beta\in H_2(S,\mathbb{Z})$ be an effective curve class.
Let
$$X=S\times\C.$$ We include $S$ as
the fiber over $0\in \C$,
$$\iota\colon S\into X\,.$$

Let $P_n(X,\iota_*\beta)$ be the quasi-projective
moduli space
of stable pairs  $(F,s)$ on $X$ with 
holomorphic Euler characteristic and class
$$\chi(F)=n, \ \ \ \ \ [F]=\iota_*\beta\in H_2(X,\Z)\,.$$
Strictly speaking, to construct $P_n(X,\iota_*(\beta))$,
we apply 
Le Potier's results \cite{LeP} 
to the projective 3-fold $$\oX=S\times\PP^1$$ to obtain 
a projective moduli space containing $P_n(X,\iota_*\beta)$ as an open 
subscheme.


We can also consider stable pairs on families of $K3$ surfaces.
Let
$$ \epsilon\colon\calS \rarr B$$
be a smooth{\footnote{We will later consider
families degenerating to $R_1\cup_E R_2$ as in Section \ref{jjyy}.}}
family of $K3$ surfaces, and let
$$\X=\calS\times\C \rarr B$$
be the corresponding family of  3-folds. We consider $\calS$ as a subvariety via
the inclusion
$$
\iota\colon\calS\times\{0\}\into\calS\times\C=\X.
$$
Let $\beta$ be a section of the local system 
with fiber $H_2(\calS_\xi,\mathbb{Z})$ over $\xi \in B$.

By making $B$ smaller if necessary, we can choose a holomorphic 2-form
which is symplectic on every fiber of $\calS\to B$,
\begin{equation} \label{holo2form}
\sigma\in H^0(\epsilon_* \Omega^2_{\calS/B}).
\end{equation}
In particular, $\omega_{\X/B}$ is trivial.
By \cite{LeP}, there is a family of moduli spaces
$$\P\to B$$
 representing the functor 
 taking $B$-schemes $A\to B$ to the set of flat families of stable pairs in the class
 $\iota_*\beta$ on the fibers  of
 $$A\times_B\X\to A.$$ 
In addition, 
there is a universal sheaf $\FF$ on $\P\times_B\X$, flat over $\P$, 
with a global section $\mathbb S$, such that the restriction of
\begin{equation} \label{}
\O_{\P\times_B\X}\Rt{\mathbb S}\FF
\end{equation}
to the fiber over any closed point $(F,s)\in\P$ over $\xi\in B$ is the corresponding stable pair $\O_{\X_\xi}\Rt{s}F$.

\subsection{Standard obstruction theory}
\label{nnzz}
As in \cite{pt1}, given a stable pair $(F,s)$ on $X$, 
we let $I\udot\in D^b(X)$ denote the complex of sheaves
$$
I\udot=\{\O_X\Rt{s}F\}
$$
in degrees $0$ and $1$. When the section is onto, 
$I\udot$ is quasi-isomorphic to the
 kernel $\I_C$, the ideal sheaf of the Cohen-Macaulay curve $C$ which is 
the scheme theoretical support of $F$.
Similarly we let
$$
\II=\{\O_{\P\times_B\X}\Rt{\mathbb S}\FF\}
$$
denote the universal complex.

From the perspective of \cite{pt1}, the trace-free 
Ext groups,
$$\Ext^1(I\udot,I\udot)_0 \ \ \ \text{and} \ \ \ \Ext^2(I\udot,I\udot)_0$$
  provide deformation and obstruction spaces for the stable pair $(F,s)$. 
More generally, let $\LL_{\P/B}^\vee$ denote the derived dual 
of the truncated relative cotangent complex of $\P$, and consider the map
\begin{equation} \label{dualobsthy}
\LL_{\P/B}^\vee\To R\pi\_{\P*}(R\hom(\II,\II)_0)[1],
\end{equation}
given by the image of the relative Atiyah class of $\II$ under the projection
\begin{eqnarray*}
\Ext^1(\II,\II\otimes\LL_{(\P\times_B\X)/B}) &\!\!\!\To\!\!\!&
\Ext^1(\II,\II\otimes\pi_\P^*\LL_{\P/B})_0
\\ &=& \Hom(\pi_\P^*\LL_{\P/B}^\vee,R\hom(\II,\II)_0[1]) \\
&=& \Hom(\LL_{\P/B}^\vee,R\pi\_{\P*}R\hom(\II,\II)_0[1]).
\end{eqnarray*}
Here $\pi\_\P$ and $\pi\_\X$ are the projections 
from $\P\times_B\X$ to $\P$ and $\X$ respectively. 


\begin{prop} \label{perf}
The map \eqref{dualobsthy} is a perfect theory for the morphism
 $\P\rarr B$.
\end{prop}

\begin{proof}
The result is proved in \cite[Section 2.3]{pt1} and \cite[Theorem 4.1]{HT} for \emph{projective}
morphisms $\pi\_\P$. Since the fibers of our $\pi\_\P$ are noncompact, we need a small
modification to check that the complexes
\begin{equation} \label{Serre}
R\pi\_{\P*}(R\hom(\II,\II)_0\otimes\omega_{\pi\_\P})[2]\ ,  \quad
R\pi\_{\P*}(R\hom(\II,\II)_0)[1]
\end{equation}
are still naturally dual to each other.  The relative canonical bundle
$$\omega_{\pi\_\P}=\pi_\X^*\omega^{}_{\X/B}$$ 
is trivial in our situation. The rest of the proofs in \cite{HT, pt1} go through
as before.

The proper way to deal with the noncompactness is to work with local cohomology in place
of $R\pi\_{\P*}$. However, we follow
a simpler approach obtained by compactifying the
fibers of $\X \to B$ by
$$\bX=\calS\times\PP^1\to B\,.$$

Pairs extend trivially over $\bX\take\X$ by pushing forward the sheaf and section, allowing
us to view $\P$ as the moduli space
of stable pairs on the fibers of $\bX\to B$ whose underlying sheaf has support in
$\X\subset\bX$. Suppressing the pushforward maps, we get a universal pair $(\FF,\mathbb
S)$ on
$\P\times_B\bX$ and a universal complex
$$
\BI=\{\O_{\P\times_B\bX}\Rt{\mathbb S}\FF\}
$$
whose restriction to $\P\times_B\X$ is $\II$.

Let $\overline\pi\_\P$ denote the projection $\P\times_B\bX\to\P$. 
Since $R\hom(\BI,\BI)_0$ is supported on $\X\subset\bX$, the two complexes \eqref{Serre}
are
\begin{equation} \label{Serre2}
R\overline\pi\_{\P*}(R\hom(\BI,\BI)_0\otimes\omega_{\overline\pi\_\P})[2]\ ,
\quad
R\overline\pi\_{\P*}(R\hom(\BI,\BI)_0)[1].
\end{equation}
Therefore the usual relative Serre duality down the projective fibers of $\overline\pi\_\P$
applies to give the duality \eqref{Serre}.
In particular, by \cite[Lemma 2.10]{pt1}, the first complex is quasi-isomorphic to a complex
$$
E\udot=\{E^{-1}\to E^0\}\simeq R\pi\_{\P*}(R\hom(\II,\II)_0\otimes\pi_\X^*\omega^{}_{\X/B})[2]
$$
of locally free sheaves on $\P$ in degrees $-1$ and $0$. We denote the second, the dual
of the first, by 
$$E\ldot=\{E_0\to E_1\}$$ in degrees $0$ and $1$.
 Dualising \eqref{dualobsthy},
 we obtain the more familiar form,
\begin{equation} \label{obsthy}
E\udot\To\LL_{\P/B}.
\end{equation}
As in \cite{HT,pt1}, the morphism \eqref{obsthy} is surjective on $h^{-1}$ and an isomorphism on $h^0$,
verifying the axioms \cite{BF} of a perfect relative obstruction theory.
\end{proof}

\subsection{Trivial quotient}
\label{spred}
For the 3-fold $X=S\times\C$, the obstruction theory constructed in
Section \ref{nnzz}
has virtual cycle equal to 0 because of the
existence of 
a trivial factor $\com$ 
obstructing extensions of $I\udot$ along
deformations of $S$ which take $\beta$ out of the $(1,1)$ locus. 
To construct a nonzero virtual cycle,
we must remove
this trivial piece of the obstruction theory.

The obstruction sheaf of the deformation-obstruction theory \eqref{dualobsthy} is the
degree 1 cohomology sheaf\footnote{Here $\ext^i_{\pi\_\P}$ denotes the $i$th cohomology
sheaf of $R\pi_{\P*}R \hom$. We abbreviate the latter to  $R\hom_{\pi\_P}$,
the
derived functor of $\hom_{\pi_{\P}}=\pi_{\P*}\hom$.}
\begin{equation} \label{obs}
\text{Ob}=\ext^2_{\pi\_\P}(\II,\II)_0\,.
\end{equation}
As in \eqref{Serre2}, we also have 
$$
\text{Ob}=\ext^2_{\overline\pi\_\P}(\BI,\BI)_0\,.
$$
Consider the image of the relative Atiyah class of $\BI$ under the map
\begin{eqnarray*}
\Ext^1(\BI,\BI\otimes\LL_{(\P\times_B\bX)/B}) &\!\!\!\To\!\!\!&
\Ext^1(\BI,\BI\otimes\pi_{\bX}^*\LL_{\bX/B})\_0
\\ &\!\!\!\To\!\!\!& H^0(\ext^1_{\overline\pi\_\P}(\BI,\BI\otimes\pi_{\bX}^*\Omega_{\bX/B})).
\end{eqnarray*}
Cup product with its image $\At$ defines the map
$$
\ext^2_{\overline\pi\_\P}(\BI,\BI)_0\xymatrix{\ar[r]^{\cup\At\,}&}
\ext^3_{\overline\pi\_\P}(\BI,\BI\otimes\pi_{\bX}^*\Omega_{\bX/B})
\Rt{\tr}R^3\overline\pi\_{\P*}(\pi_{\bX}^*\Omega_{\bX/B}).
$$
 Pulling back 
the fiberwise symplectic form $\sigma$ of \eqref{holo2form}
to $\P\times_B\bX\to B$ gives a section $\bar\sigma$ of $\pi\_{\P_*}
(\pi_{\bX}^*\Omega^2_{\bX/B})$. Wedging with 
$\bar\sigma$,
the upshot is a map from \eqref{obs} to $\O_\P$:
\begin{equation} \label{map}
\text{Ob}\To R^3\overline\pi\_{\P*}(\pi_{\bX}^*\Omega^3_{\bX/B})=
R^3\overline\pi\_{\P*}\omega_{\overline\pi\_{\P}}=\O_\P.
\end{equation}

\begin{prop} \label{not0}
The map \eqref{map} is onto.
\end{prop}

\begin{proof}
Since the higher $\ext^i_{\overline\pi\_\P}(\BI,\BI)_0$ sheaves on $B$
vanish for $i\ge3$,
we can work at closed points.
For a stable pair $(F,s)$ on 
 $$\bX_\xi=\oX=S\times\PP^1\ ,$$
we must show the composition
\begin{multline} \label{cohobs}
\Ext^2(\bI,\bI)_0\xymatrix{\ar[r]^{\cup\At(\bI)\,}&}
\Ext^3(\bI,\bI\otimes\Omega_{\oX})\Rt{\tr}
\\ H^3(\Omega_{\oX})
\Rt{\cup\bar\sigma}H^{3,3}(\oX)\Rt{\sim}\C
\end{multline}
is onto. Here, $\bar\sigma$ is the pull-back 
of the holomorphic symplectic form $\sigma$ \eqref{holo2form} from
the $K3$ surface $S$ to $\oX$, and $\bI$ is the complex $\{\O_{\oX}\to
F\}$ on $\oX$.

To show the map \eqref{cohobs} is surjective, 
we exhibit a class in $\Ext^2(\bI,\bI)_0$ on which
the composition is  nonzero. 
Choose a first order deformation $\kappa_S\in H^1(T_S)$
of $S$ which, via the holomorphic 
symplectic form
$$\sigma\colon T_S\cong\Omega_S\ ,$$
corresponds to a class $\kappa_S\ip\sigma\in H^{1,1}(S)$ 
whose pairing with $\beta$ is nonzero,
\begin{equation} \label{nonzero}
\int_\beta \kappa_S\ip\sigma\ne0.
\end{equation}
Let
$\bar\kappa\in H^1(T_{\oX})$
denote the pull-back of the Kodaira-Spencer class $\kappa_S$ to $\oX$. 
Let
\begin{equation}\label{h555}
\bar\kappa\circ\At(\bI)\,\in\,\Ext^2(\bI,\bI)
\end{equation}
be the cup product of 
$\bar\kappa$ with $\At(\bI)\in\Ext^1(\bI,\bI\otimes\Omega_{\oX})$
followed by the contraction of $T_{\oX}$ with $\Omega_{\oX}$. By \cite{HT},
 the element \eqref{h555}
is the obstruction to deforming $\bI$ to first order 
with the deformation $\kappa$ of
$\oX$, and in fact lies in $\Ext^2(\bI,\bI)_0\subset\Ext^2(\bI,\bI)$ since the determinant
$\O_{\oX}$ of $\bI$ deforms trivially.
        
By \cite[Proposition 4.2]{BuF},
$
\tr\big(\bar\kappa\circ\At(\bI)\circ\At(\bI)\big)\in H^3(\Omega_{\oX})
$
equals
$
2\bar\kappa\ip \mathrm{ch}_2(\bI).
$
Therefore the image of $\bar\kappa\circ\At(\bI)$ under the map \eqref{cohobs} is 
\begin{equation}\label{bb5}
2\int_{\oX}(\bar\kappa\ip\mathrm{ch}_2(\bI))\wedge\bar\sigma\ =\ 
-2\int_{\oX}(\bar\kappa\ip\bar\sigma)\wedge\mathrm{ch}_2(\bI),
\end{equation}
by the homotopy formula 
\begin{equation}\label{htyy}
0\ =\ \bar\kappa\ip(\mathrm{ch}_2\wedge\bar\sigma)\ =\
(\bar\kappa\ip\mathrm{ch}_2)\wedge\bar\sigma+
(\bar\kappa\ip\bar\sigma)\wedge\mathrm{ch}_2.
\end{equation}
Since $\mathrm{ch}_2(\bI)$ is Poincar\'e dual to $-\iota_*\beta$, we conclude
\eqref{bb5} equals
$$
2\int_\beta\kappa_S\ip\sigma,
$$
which by construction \eqref{nonzero} is nonzero.
\end{proof}

\subsection{Symmetric obstruction theories}\label{spsym}
By Proposition \ref{perf}, the two term complex of locally free sheaves
$E\udot$ over $\P$, quasi-isomorphic to 
$$R\pi\_{\P*}(R\hom(\II,\II)_0
\otimes\pi_\X^*\omega^{}_{\X/B})[2]\ ,$$ provides a perfect obstruction theory for 
$\P \rarr B$.
Via the trivialisation of $\omega_{\X/B}$ and the equality between \eqref{Serre}
and \eqref{Serre2}, the Serre duality of \eqref{Serre2} shows that $E\udot$ is isomorphic
to its own derived dual shifted by $[1]$,
\begin{equation} \label{symm}
(E\udot)^\vee[1]\cong E\udot.
\end{equation}
Moreover the pairing between $E\udot$ and $E\udot[1]$ is given by trace, 
which satisfies
$$\tr(a\cup b)=\tr(b\cup a)\ .$$ Hence,
 the isomorphism \eqref{symm} is also equal
to its own dual \cite[Lemma 1.23]{BF2},
and the deformation-obstruction theory of Proposition \ref{perf} is \emph{symmetric}
in the sense of \cite{bdt, BF2}.

Since $E\udot$ is an obstruction theory,  $h^0(E\udot)=\Omega_{\P/B}$. By definition,
the obstruction sheaf $\text{Ob}$ is $h^1$ of the dual
complex $E_{\scriptscriptstyle\bullet}=(E\udot)^\vee$. Since
$E\udot$ is symmetric,
we have
$$
\text{Ob}=\Omega_{\P/B}.
$$
A map $\text{Ob}\to\O_\P$ is therefore equivalent to a section of the tangent sheaf
\begin{equation} \label{vfield}
T_{\P/B}=\hom(\Omega_{\P/B},\O_\P)=h^{-1}(E\udot).
\end{equation}

In our case, the product geometry provides 
a section of  $T_{\P/B}$ 
by moving all stable pairs by the vector field $\partial_t$ lifted from the second
factor $\C$ of
$$\X = \mathcal{S} \times \C\,.$$
Explicitly, the section is 
\begin{equation} \label{tangent}
\spreaddiagramcolumns{2pc}
\xymatrix{
\O_\P \rto^(.35){\partial_t\ip\At(\II)} & \ext^1_{\pi\_\P}(\II,\II)_0} \,
\end{equation}
where the sheaf on the right is 
the 0th cohomology of $E\udot$. 
By the duality \eqref{Serre}  and the vanishing of
higher $\ext$s,
$$
\ext^1_{\pi\_\P}(\II,\II)_0\ =\ \hom(\ext^2_{\pi\_\P}(\II,\II)_0,\O_\P),
$$
just as in \eqref{vfield}. Therefore \eqref{tangent} gives
\begin{equation} \label{map2}
\ext^2_{\pi\_\P}(\II,\II)_0\to\O_\P.
\end{equation}

\begin{lem} \label{same}
The map \eqref{map2} is the same as \eqref{map}.
\end{lem}

\begin{proof}
The perfect pairing between the two complexes \eqref{Serre}
is provided by composition of the derived Homs in \eqref{Serre2} followed
by the trace, 
\begin{equation} \label{int}
R\overline\pi\_{\P*}(R\hom(\BI,\BI)\otimes\omega_{\overline\pi\_\P})[3]\Rt{\tr}
R\overline\pi\_{\P*}\omega_{\overline\pi\_\P}[3]\To R^3\overline\pi\_{\P*}
\omega_{\overline\pi\_\P}.
\end{equation}
The final map takes the highest nonvanishing cohomology sheaf of the complex.
Since the fibers of $\overline\pi\_\P$ are projective, last sheaf is $\O_\P$, as required.

To prove the Lemma, we must show sections $f$ of $\ext^2_{\pi\_\P}(\II,\II)_0$
satisfy
$$
\tr(f\cup\At(\II))\wedge\sigma\ =\ \tr(f\cup(\partial_t\ip\At(\II)))\wedge(\sigma\wedge
dt),
$$
where $\sigma\wedge dt$ is the trivialisation of $\omega_\X$. The result
follows from the homotopy formula $a\ip(b\wedge c)=(a\ip b)\wedge c\pm b\wedge(a\ip c)$ used before.
\end{proof}

In particular, since the map \eqref{tangent} is clearly pointwise
injective for pairs with curve class $\beta$ (supported in the $K3$ fibers of our
threefolds), we recover Proposition \ref{not0}.

\subsection{Reduced obstruction theory} \label{rdot}
We now assume that our smooth family of $K3$ surfaces
$$ \epsilon\colon\calS \rarr B $$
has base a nonsingular curve $B$, and that $\beta$ is
of type $(1,1)$ on every fiber $\calS_\xi,\ \xi \in B$.

Following the notation of Section \ref{spdef}, 
let
\begin{equation}\label{nnn6}
\X=\calS\times\C \rarr B\  
\end{equation}
be a family of 3-folds, and let
 $$\P\to B$$
be the associated family of moduli spaces
of stable pairs in class $\iota_*\beta$ on
the fibers of \eqref{nnn6}.
We will construct and 
prove the deformation invariance of the
reduced virtual class on the family $\P\to B$.

Since $B$ is nonsingular, every perfect obstruction theory 
$$E\udot\to\LL_{\P/B}$$ 
for $\P\to B$
induces a  perfect absolute 
obstruction theory for $\P$ by virtue of the exact triangle
\begin{equation} \label{PKS}
\Omega_B\to\LL_{\P}\to\LL_{\P/B},
\end{equation}
where we have suppressed a pull-back map. 
Using the composition $E\udot\to\LL_{\P/B}
\to\Omega_B[1]$, we define
$$
\E\udot=\operatorname{Cone}\big(E\udot\to\Omega_B[1]\big)[-1]\,.
$$
We have diagram of exact triangles
\begin{equation} \label{ladder}
\spreaddiagramrows{-.8pc}
\xymatrix{
\E\udot \rto\dto & E\udot \rto\dto & \Omega_B[1] \ar@{=}[d] \\
\LL_{\P} \rto & \LL_{\P/B} \rto & \Omega_B[1]\,. \!\!}
\end{equation}
By  computation with the long exact sequence in cohomology sheaves of the
above diagram, $\E\udot\to\LL_\P$ is an isomorphism on $h^0$ 
and surjective on $h^{-1}$.
Since 
$$E\udot=\{E^{-1}\to E^0\}$$ is a complex of locally free 
sheaves, the induced map
$E\udot\to\Omega_B[1]$ can be represented \emph{locally} by a genuine map of complexes
$$
\spreaddiagramrows{-.8pc}
\xymatrix{\big\{E^{-1} \rto\dto & E^0\big\} \\
\Omega_B\,.\!\!\!}
$$
Hence, $\E\udot$ is locally
represented by the 2-term complex of locally free sheaves
\begin{equation}\label{gg23}
\E\udot\cong\{E^{-1}\to E^0\oplus\Omega_B\}.
\end{equation}
Since $\P$ is quasi-projective, 
$\E\udot$ is  \emph{globally} a 2-term complex of locally free sheaves.
So $\E\udot\to\LL_\P$ is indeed a perfect obstruction theory. \medskip

From the perfect relative obstruction theory of Proposition \ref{perf},
$$E\udot=\big(R\hom_{\pi_\P}(\II,\II)_0[1]\big)^\vee\,,$$
we obtain a perfect absolute obstruction theory $\E\udot$ for $\P$.
From the dual of the top row of \eqref{ladder}
we have the long exact sequence of cohomology sheaves 
$$
0\to\ext^1_{\pi\_\P}(\II,\II)_0\to T_\P\to T_B
\to\ext^2_{\pi\_\P}(\II,\II)_0\to\text{Ob}_\P \to 0.
$$
Here $\ext^1_{\pi\_\P}(\II,\II)_0=T_{\P/B}$ is the relative tangent sheaf of $P/B$, and
$T_\P$ is the absolute tangent sheaf. 
Similarly,
$\text{Ob}_\P$
 is the obstruction sheaf $h^1((\E\udot)^\vee)$ of the absolute obstruction theory
$\E\udot$, the quotient of the relative obstruction sheaf 
$$\text{Ob}=\ext^2_{\pi\_\P}(\II,\II)_0$$
by the image of $T_B$.

By Proposition \ref{not0}, the map \eqref{map} is a surjection Ob$\,\to\O_\P$.
To
apply the construction of
Kiem-Li to $\E\udot\to\LL_\P$, we must show the map \eqref{map}
 annihilates $T_B$
and so descends to a surjection $\text{Ob}_\P\to\O_\P$. 
To do so we need a description of the
composition $E\udot\to\LL_{\P/B}\to\Omega_B[1]$. \medskip

By the description \eqref{dualobsthy} of 
$(\LL_{\P/B})^\vee \to (E\udot)^\vee$,
the dual of $$E\udot\to\LL_{\P/B}\to\Omega_B[1]$$
 is the composition
\begin{equation} \label{compP}
\spreaddiagramcolumns{-.6pc}
\xymatrix{T_B[-1] \rto & \LL^\vee_{\P/B} \rrto^(.32){\operatorname{At}_{\P/B}} &&
R\hom_{\pi\_\P}(\II,\II)[1],}
\end{equation}
which actually factors through the trace-free 
part $R\hom_{\pi\_\P}(\II,\II)_0[1]$ as
in \cite[Theorem 4.1]{HT} since all of the complexes 
$\II$ have fixed trivial determinant.
Here, the first map is the Kodaira-Spencer
class of the fibers of $\P\to B$ obtained from the exact triangle
\eqref{PKS}. 
The second is cup product with the image $\operatorname{At}_{\P/B}$ of
the relative Atiyah class $\operatorname{At}_{\X\times_B\P}$ of $\II$ under the map
$\Ext^1(\II,\II\otimes\LL_{\X\times_B\P})\to
\Ext^1(\II,\II\otimes\pi_\P^*\LL_{\P/B})$.

We can construct a similar composition
\begin{equation} \label{compX}
\spreaddiagramcolumns{-.6pc}
\xymatrix{T_B[-1] \rto & R\pi\_{\P*}\LL^\vee_{\X/B} \rrto^(.4){\operatorname{At}_{\X/B}}
&& R\hom_{\pi\_\P}(\II,\II)[1]}
\end{equation}
from the projection $\operatorname{At}_{\X/B}$ of
$\operatorname{At}_{\X\times_B\P}$ under $\LL_{\X\times_B\P}\to\LL_{\X/B}$. 
The
first map of \eqref{compX} 
is the Kodaira-Spencer class of the fibers of $\X\to B$  obtained from
($\pi\_{\P*}\pi^*_{\X}$ applied to) the exact triangle
\begin{equation} \label{XKS}
\Omega_B\to\LL_\X\to\LL_{\X/B}\,.
\end{equation}

\begin{prop} \label{coincide} The two composition
 \emph{(\ref{compP})} and \emph{(\ref{compX})} coincide.
\end{prop}

\begin{proof} We relate both maps to a third, constructed in the same way using the full
Atiyah class
\begin{equation} \label{compXP}
\spreaddiagramcolumns{-.6pc}
\xymatrix{T_B[-1] \rto & R\pi\_{\P*}\LL^\vee_{\X\times_B\P}
\rrto^(.42){\operatorname{At}_{\X\times_B\P}} && 
R\hom_{\pi\_\P}(\II,\II)[1]}\,.
\end{equation}
Here, the first map is the Kodaira-Spencer class in 
$\Ext^1(\LL_{\X\times_B\P},\Omega_B)$
of the inclusion 
$$i\colon\X\times_B\P\subset\X\times\P $$ 
coming from the exact
triangle
\begin{equation} \label{XPKS}
\Omega_B\to Li^*\LL_{\X\times\P}\to\LL_{\X\times_B\P}\,.
\end{equation}
The composition \eqref{compXP} coincides with \eqref{compP} by
 the following
commutative diagram of exact triangles on $\X\times_B\P$ relating 
the Kodaira-Spencer
classes \eqref{XPKS} and \eqref{PKS},
$$
\spreaddiagramrows{-.8pc}
\spreaddiagramcolumns{-.5pc}
\xymatrix{
& \LL_\X \ar@{=}[r]\dto & \LL_\X \dto \\
\Omega_B \rto \ar@{=}[d] & Li^*\LL_{\X\times\P} \rto\dto & \LL_{\X\times_B\P} \dto\\
\Omega_B \rto & \LL_\P \rto & \LL_{\P/B}.\!\!}
$$
We have suppressed
several pull-back maps.
The central row gives rise to \eqref{compXP} while the 
bottom row induces \eqref{compP}.

Similarly we have the following
diagram relating the Kodaira-Spencer classes \eqref{XPKS} and \eqref{XKS},
$$
\spreaddiagramrows{-.8pc}
\spreaddiagramcolumns{-.5pc}
\xymatrix{
& \LL_\P \ar@{=}[r]\dto & \LL_\P \dto \\
\Omega_B \rto\ar@{=}[d] & Li^*\LL_{\X\times\P} \rto\dto & \LL_{\X\times_B\P} \dto\\
\Omega_B \rto & \LL_\X \rto & \LL_{\X/B}.\!\!}
$$
The central row gives rise to \eqref{compXP} while the bottom row induces \eqref{compX},
so the two compositions coincide.
\end{proof}

By Proposition \ref{coincide},
we may use the description \eqref{compX} of the map 
$$T_B\to\ext^2_{\pi\_\P}(\II,\II)_0$$
to compute the composition with 
$$\ext^2_{\pi\_\P}(\II,\II)_0\to\O_\P\,.$$ 
As
in Proposition \ref{not0}, 
we use the extension $\BI$ of $\II$ over $\bX\supset\X$. The
result is
\spreaddiagramcolumns{.5pc}
\begin{eqnarray} \nonumber
T_B && \hspace{-6mm} \To\!\!\xymatrix{R^1\overline\pi\_{\P*}(\LL^\vee_{\bX/B})
\rto^{\operatorname{At}_{\bX/B}} & \ext^2_{\overline\pi\_\P}(\BI,\BI)_0
\rto^(.7){\operatorname{At}_{\bX/B}} & } \\ \label{arrowss}
&& \ext^3_{\overline\pi\_\P}(\BI,\BI\otimes\pi_{\bX}^*\Omega_{\bX/B})
\Rt{\tr} R^3\overline\pi\_{\P*}(\pi_{\bX}^*\Omega_{\bX/B})
\Rt{\wedge\bar\sigma} R^3\overline\pi\_{\P*}(\omega\_{\overline\pi\_\P}).
\end{eqnarray}
Working locally over $\P$, we will show the composition vanishes
 when applied
to any section of $T_B$. 
Let 
$$\mathsf{KS}\in R^1\overline\pi\_{\P*}(T_{\bX/B})$$ denote the associated
Kodaira-Spencer class. By \cite[Proposition 4.2]{BuF} (applied to $\mathcal N=\O_\P$),
the image of our section under the first four maps above can be computed as
\begin{equation}\label{btt4}
\mathsf{KS}\ip \mathrm{ch}_2(\BI)\in R^3\overline\pi\_{\P*}(\pi_{\bX}^*\Omega_{\bX/B}).
\end{equation}
The class \eqref{btt4} is the (1,3)-part of the derivative down the vector field
along $B$ of $\iota_*(\beta)$. By assumption the class
$\iota_*\beta$ is of pure type (2,2) over all of $B$, so the class \eqref{btt4} is zero.
We have proven the following result.

\begin{prop} \label{kkll}
We have a surjection $\text{\em Ob}_\P \rightarrow \O_\P$
extending the surjection $\text{\em Ob}\rightarrow \O_\P$ of \eqref{map}.
\end{prop}

\subsection{Reduced classes} \label{rc}
From the perfect absolute obstruction theory  $\E\udot$,
the constructions of   \cite{BF,LiTian} produce 
a normal cone $C$ and an  embedding
$$C\subset \E_1$$ 
into the total space of the vector bundle $\E_1$ dual to $\E^{-1}$. Without loss of generality
we may assume that $\E_1 \cong E_1$ (as in \eqref{gg23}),
so $$C\subset E_1\,.$$ 
Restricting \eqref{obsthy} to a fiber $$\iota\_\xi\colon\P_\xi\into\P$$
over $\xi\in B$ 
yields the perfect obstruction theory
$$
E\udot|_{\P_\xi}\To\LL_{\P_\xi},
$$
a normal cone $C_\xi$, and an embedding $C_\xi\subset E_1|_{\P_\xi}$. 
By \cite{BF}, the cone $C$ specialises to the cone $C_\xi$,
\begin{equation} \label{special}
[C_\xi]=\iota_\xi^![C].
\end{equation}
The cone $C_\xi$ lies on $C$ and
relation \eqref{special} is valid on $C$.
Intersecting $C_\xi$ with the zero section of $E_1|_{\P_\xi}$ yields
 the usual virtual cycle $[\P_\xi]^{vir}$ employed in \cite{pt1}
(and which vanishes here).

A reduced virtual class is obtained by the following
construction.
Let
$$
F_1\subset \E_1
$$
on $\P$
denote the locally free kernel of the surjective composition \eqref{map}
$$
\E_1\to \text{Ob}_\P\to\O_\P.
$$
By results of 
Kiem and Li \cite{KL}, the normal cone $C\subset\E_1$ lies in 
$F_1\subset \E_1$ \emph{as a cycle} (rather than scheme theoretically\footnote{In Appendix
\ref{oldnew} we explain why, for our particular moduli space $\P$ and obstruction theory
$\E\udot$,
replacing $\E_1$ by $F_1$ gives a genuine perfect obstruction theory. Therefore
$C$ lies in $F_1$ scheme theoretically. This is not the case for general obstruction
theories, however, and is not necessary for what follows.}). 
Therefore we may view $C$ as a cycle in $F_1$ and
$C_\xi$ as a cycle in $F_1|_{\P_\xi}$.
We define
$$[\P]^{\mathrm{red}} = 0^![C] \in A_2(\P, \mathbb{Z})$$ and by
intersecting $C$ with the zero section $0$ of $F_1$ and
$$
[\P_\xi]^\mathrm{red}=0_\xi^![C_\xi] \ \in A_1(\P_\xi,\mathbb{Z})
$$
by intersecting $C_\xi$ with the zero section
$0_\xi$
of $F_1|_{\P_\xi} \subset E_1|_{\P_\xi}$.
The deformation invariance of $[\P_\xi]^\mathrm{red}$
is a consequence of the equation
$$
\iota_\xi^![\P]^\mathrm{red}=[\P_\xi]^\mathrm{red},
$$
obtained by the identity \eqref{special} on $C$.


\subsection{Reduced invariants}
Since $X=S\times\C$ is not compact, neither is the
 moduli space $P_n(X,\iota_*\beta)$.
However, the fixed point set 
$$P_n(X,\beta)^{\com^*}\subset P_n(X,\beta)$$  of
the $\com^*$-action induced by scaling the second factor of $X$ 
is compact.
We can therefore define invariants by residues.

The $\C^*$-action on $P_n(X,\beta)$ lifts to the perfect obstruction theory 
\eqref{obsthy}. 
Since the map \eqref{map} is easily seen to be $\C^*$-invariant, 
we obtain a $\com^*$-equivariant reduced virtual class.
We define
$$P_{n,\beta} = \int_{[P_n(X,\beta)]^\mathrm{red}} 1 \,,$$
 where the right side is the $\com^*$-equivariant
residue.

By Lemma \ref{gthr}, the integral
can be evaluated as $$P_{n,\beta}=\frac1t(-1)^{n+\langle \beta,\beta \rangle+1} 
\mathsf{e}(P_n(S,\beta))$$ 
when $\beta \in H_2(S,\mathbb{Z})$ is irreducible.

\subsection{Degenerating family of $K3$ surfaces}\label{dk3s}
We now consider the family of $K3$ surfaces 
$$\epsilon: \calS \rarr B$$
over a pointed curve $(B,0)$
defined in
Section \ref{jjyy}.
The family $\epsilon$ satisfies conditions (i-v) of Section
\ref{jjyy} and 
has special fiber 
$$\calS_0 =R_1\cup_E R_2\,.$$ As before, let
\begin{equation}\label{yayaya}
\X= \calS \times \C \to B\,. 
\end{equation}
Denote
the special fiber by $X[0]=\calS_0\times\C$,
and let 
$$X[0]=Y_1\cup_{E\times\C}Y_2$$ denote the 
decomposition where $Y_i= R_i \times \C$.

Following the notation of \cite{LiWu}, let $\B=\B(\beta,n)$
 denote the Artin stack of $(\beta,n)$-decorated semistable models 
of $\X/B$, with the associated
 universal family 
\begin{equation} \label{uu77}
\tX\to\B\,.
\end{equation}
The stack $\B$ has a (non-representable) morphism
to $B$ with the fiber over $0\in B$  denoted by
$\B_0$. Away from $\B_0$, the universal family \eqref{uu77}
is just the family
of quasi-projective schemes 
$$(\X\take X[0])\to B\take\{0\}\,.$$
 Replacing the special fiber $X[0]$ is the union over
all $k$ of the $k$-step semistable models
$$
X[k]=\Big(R_1\cup_E(E\times\PP^1)\cup_E\ldots\cup_E(E\times\PP^1)\cup_E R_2\Big)\times\C
$$
with automorphisms $(\C^*)^k$ covering the identity on $X[0]$
($k$ is the number of extra components 
$(E\times\PP^1)\times\C$ in the semistable model). 
The decoration is an assignment of $H_2$ classes and integers 
for each component of the fibers of $\widetilde\X/\B$, 
satisfying standard gluing and continuity conditions described 
in \cite{LiWu}. 
In particular, on the nonsingular fibers, the decoration is simply $(\beta,n)$.

A relative stable pair on the special fiber is
\begin{equation}\label{relpair}
\O_{X[k]} \stackrel{s}{\rightarrow} F
\end{equation}
where
$F$ is a sheaf on $X[k]$ with holomorphic
Euler characteristic $\chi(F)=n$ and 
class which pushes down to
$$\iota_*\beta\in H_2(X[0],\Z).$$
The stability conditions for the pair are
\begin{enumerate}
\item[(i)] $F$ is pure with finite locally free resolution,
\item[(ii)] $F$ is transverse to the singular loci $E_i$ of $X[k]$,
$$\text{Tor}_j(F,\O_{E_i})=0\ \ \ \text{for all $i$ and $j>1$} \ ,$$
\item[(iii)] the section $s$ has 0-dimensional cokernel supported
away from the singular loci $E_i$ of $X[k]$, and
\item[(iv)] the pair \eqref{relpair} has only finitely many automorphisms covering
the automorphisms $(\C^*)^k$ of $X[k]/X[0]$,
\end{enumerate}
see \cite{LiWu,pt1}.  

There is a Deligne-Mumford moduli stack 
$$\P\to\B$$ of 
stable pairs on the fibers of $\tX\to\B$
whose restriction to each component of $X[k]$ has support and holomorphic Euler characteristic equal to the decoration. 
There is  universal complex 
$\II$ over 
$$\tX\times_\B\P$$ by condition (iv). 
Composing with $\B\to B$ gives
$\P\to B$ which, away from the special
fiber, 
is the quasi-projective moduli space studied
in Sections \ref{spdef}-\ref{spsym}.
 
As before, $\P$ is an open subset of a \emph{proper} Deligne-Mumford stack formed by considering relative stable pairs
on the compactification  $\bX\to\B$ given
by replacing the $\C$ factor in \eqref{yayaya} by $\PP^1$ everywhere.

The universal complex $\II$ is \emph{perfect}
due to condition (i) above. The deformation theory 
and Serre duality of \cite{HT,pt1}
go through exactly as before. Let 
$\pi\_{\tX},\pi\_\P$ denote the projections from
$\tX\times_\B\P$ to $\tX$ and $\P$ respectively. 
Just as in \eqref{dualobsthy}, the Atiyah class of $\II$ 
gives a perfect obstruction theory
\begin{equation}\label{BrokenObsThy}
E\udot=R\pi\_{\P*}(R\hom(\II,\II)_0\otimes\pi_{\tX}^*\omega^{}_{\tX/\B})
[2]\to\LL_{\P/\B}\,.
\end{equation}
Moreover, Serre duality applies to give the duality \eqref{symm}:
$$
(E\udot)^\vee[1]\cong E\udot.
$$
Therefore the relative obstruction sheaf over $\P/\B$ is $\ext^2_{\pi\_\P}(\II,\II)_0$
with a map
\begin{equation} \label{mappp}
\ext^2_{\pi\_\P}(\II,\II)_0\to\O_\P
\end{equation}
defined exactly as in \eqref{map2}. 
The map coincides, as before,
with the map defined{\footnote{Since 
$\tX/\B$ is a reduced local complete
intersection morphism, $\LL_{\tX/\B}=\Omega_{\tX/\B}$.}}
in \eqref{map}. 
While the proof
of Proposition \ref{not0} is valid with the right notion of 
Chern classes for
perfect complexes on $X[k]$, the dual description \eqref{map2} is
technically easier. 
Since the vector field $\partial_t$
is nowhere zero on $\P$, the map \eqref{mappp} is again a surjection.

\subsection{Degeneration of the reduced class} \label{tinder}
Let
\begin{equation}\label{ke44}
\X \rarr B
\end{equation}
be the degenerating family of 3-folds considered in Section 
\ref{dk3s} above.
Let 
$\beta=\mathsf{s}+h\mathsf{f}$
be a vertical curve class. 
Let $$\P \rarr \B$$ be the moduli space of stable pairs on the fibers of 
\eqref{ke44} with holomorphic Euler characteristic $n$ and 
class $\beta$.
Let $\P_0$ be the special fiber of the composition 
$$\P\to\B\to B$$ parameterizing stable
pairs on semistable degenerations $X[k]$ of 
$$X[0]=Y_1\cup_{E\times\C}Y_2\,. $$

Given data
$\eta=(n_1,n_2,h_1,h_2)$
defining a splitting 
$$
h=h_1+h_2, \quad n+1=n_1+n_2,
$$
we can construct
the moduli spaces $\P_{\eta_1}$ and $\P_{\eta_2}$ of \emph{relative} stable pairs on $Y_1$ and $Y_2$ of classes 
$(\mathsf{s}+h_1\mathsf{f},n_1)$ and $(\mathsf{s}+h_2\mathsf{f},n_2)$ 
respectively. 
By restriction of relative stable pairs to the boundary divisor, 
$\P_{\eta_i}$ 
maps to $E\times\C$. We define
\beq \label{fibrep}
\P_\eta=\P_{\eta_1}\times\_{E\times\C} \P_{\eta_2}\ 
\eeq
which embeds into $\P_0$. In fact, $\P_0$ the union (not disjoint!) of 
the $\P_\eta$
over all possible splitting types $\eta$. 

The perfect obstruction theory $E\udot\to\LL_{\P/\B}$  
\eqref{BrokenObsThy} 
for $\P \rarr \B$
fits into the following commutative diagram of exact triangles:
\beq \label{dg1}
\spreaddiagramrows{-.8pc}
\xymatrix{
\E\udot \rto\dto & E\udot \rto\dto & \LL_\B[1] \ar@{=}[d] \\
\LL_{\P} \rto & \LL_{\P/\B} \rto & \LL_\B[1].\!\!}
\eeq
The bottom row induces
the map $E\udot\to\LL_\B[1]$ whose cone we define to be $\E\udot[1]$.

Since $\B$ is nonsingular, 
$h^{-1}(\LL_\B)=0$. But $\B$ has nontrivial isotropy groups,
so $h^1(\LL_\B)$ is nonzero.
However, stable pairs have no continuous automorphisms (since
$\P$ is a Deligne-Mumford stack with no continuous stabilizers
by condition (iv)) so
$$h^0(\LL_{\P/\B})\to h^1(\LL_\B)$$ is onto. 
From the long
exact sequences in cohomology of the above diagram,
$\E\udot$ has cohomology only
in degrees $-1$ and $0$. Just as in \eqref{ladder}, we conclude 
 $\E\udot\to\LL_{\P}$
is a perfect absolute obstruction theory for $\P$.

Restriction to $\P_0\subset\P$ yields $E\udot|_{\P_0}\to
\LL_{\P/\B}\big|_{\P_0}$ which we can 
compose with $\LL_{\P/\B}\big|_{\P_0}\to\LL_{\P_0/\B_0}$ to 
give a perfect obstruction theory
$$
E\udot|_{\P_0}\To\LL_{\P_0/\B_0}\,
$$
for $\P_0 \rarr \B_0$.
Just as in \eqref{dg1}, we can
construct a perfect absolute obstruction theory for $\P_0$ via the diagram
\beq \label{dg11}
\spreaddiagramrows{-.8pc}
\xymatrix{
\E_0\udot \rto\dto & E\udot|_{\P_0} \rto\dto & \LL_{\B_0}[1] \ar@{=}[d] \\
\LL_{\P_0} \rto & \LL_{\P_0/\B_0} \rto & \LL_{\B_0}[1],\!\!}
\eeq
which defines $\E_0\udot$ and the map to $\LL_{\P_0}$.

The top row of \eqref{dg11} and the pull-back to $\P_0$ of \eqref{dg1} 
give the diagram of exact triangles
$$
\spreaddiagramrows{-.8pc}
\xymatrix{
&& \LL_{\B_0/\B} \dto\ar@{=}[r] & \LL_{\B_0/\B} \dto \\
\E\udot|_{\P_0} \rto\dto & E\udot|_{\P_0} \ar@{=}[d]\rto & \LL_\B\big|_{\B_0}[1] \rto\dto & \E\udot|_{\P_0}[1] \dto \\
\E\udot_0 \rto & E\udot|_{\P_0} \rto & \LL_{\B_0}[1] \rto & \E_0\udot[1].}
$$
Since $\B_0\subset\B$ is the pull-back from $B$ of 
the divisor $\{0\}\subset B$ with associated line bundle $L_0$, we have 
 $\LL_{\B_0/\B}\cong L^\vee_0[1]$. Therefore,
 the rightmost column of the above diagram gives the exact triangle
\beq \label{dg3}
\E\udot|_{\P_0}\To\E_0\udot\To L^\vee_0[1],
\eeq
relating the obstruction theories of $\P$ (pulled back to $\P_0$) and $\P_0$.


\medskip
There is a divisor $\B_\eta\subset\B$ in the
stack of decorated semistable models whose pull-back to 
$\P$ is $\P_\eta$. 
The associated line bundles
$L_\eta$ satisfy
$$
\bigotimes_\eta L_\eta=L_0.
$$
We can replace $\P_0\subset\P$ over $\B_0\subset\B$ 
by $\P_\eta\subset\P$ over
$\B_\eta\subset\B$  and $L_0$ by $L_\eta$ in the above diagrams. 
The result is a perfect obstruction theory $\E_\eta\udot\to\LL_{\P_\eta}$ 
sitting in an exact triangle:
\beq\label{dg4}
\E\udot|_{\P_\eta}\To\E_\eta\udot\To L_\eta^\vee[1].
\eeq
\smallskip

The map  $\O_\P[1]\to E\udot$ of \eqref{map2}, extended to singular $K3$s in \eqref{mappp}, was shown in Section \ref{rdot} to lift to the absolute obstruction theory 
\begin{equation}\label{jjj333}
\O_\P[1]\to\E\udot
\end{equation}
over the nonsingular locus $B \setminus\{0\}$. Since $\II$ is a perfect complex the same
proof extends to the singular $K3$s, as far as \eqref{arrowss}. To finish off we must
show that the composition \eqref{arrowss} is zero. By the usual homotopy formula \eqref{htyy}
we compute the composition as
$$
-2(\mathsf{KS}\ip\bar\sigma)\cup\mathrm{ch}_2(\BI)\,\in\,
R^3\overline\pi\_{\P*}(\omega\_{\overline\pi\_\P})\,\cong\,\O_\P.
$$
Even in the singular geometry, the above cup product
equals the integral of $2\mathsf{KS}\ip\bar\sigma$ 
over the class $\iota_*\beta$. The integral
vanishes since  $\beta$ is always 
of type $(1,1)$ in the family $B$.
Hence, the lift \eqref{jjj333} extends over all of $\B$.

Let $\E^{\mathrm{red}}$ be the
cone of \eqref{jjj333}. 
Restricting \eqref{jjj333} to $\P_0$ and $\P_\eta$ and using (\ref{dg3}, \ref{dg4}) 
yields the compositions
\begin{eqnarray*}
\O_{\P_0}[1]\to\E\udot|_{\P_0}\to\E\udot_0, \\
\O_{\P_\eta}[1]\to\E\udot|_{\P_\eta}\to\E\udot_\eta.
\end{eqnarray*}
Taking the cones defines the respective
reduced theories $\E^{\mathrm{red}}_0$ and $\E_\eta^{\mathrm{red}}$. By \eqref{dg3} and \eqref{dg4}, we obtain the exact triangles
\beq\label{7}
\E^{\mathrm{red}}|_{\P_0}\To\E^{\mathrm{red}}_0\To L_0^\vee[1]\ ,
\eeq
\beq\label{8}
\E^{\mathrm{red}}|_{\P_\eta}\To\E^{\mathrm{red}}_\eta\To L_\eta^\vee[1]\,.
\eeq


\vspace{10pt}

We have now worked out the compatibilities of the reduced obstruction
 theories
for $\P$, $\P_\eta$, and  $\P_0$. 
We now turn to the compatibility
between the reduced obstruction theory of $\P_\eta$ and 
the usual obstruction theories of $\P_{\eta_1}$ and $\P_{\eta_2}$.

Consider a point
$[I\udot]\in \P_\eta$ of the moduli space corresponding to a 
stable pair on $X[k]$.
A decomposition of  $X[k]$ as $X_1\cup_{E\times\C}X_2$ yields
\beq\label {O12}
\xymatrix{
0\rto&\O_X\rto&\O_{X_1}\oplus\O_{X_2}\rto^(.55){(1,-1)}&\O_{E\times\C}\rto&0.}
\eeq
The stable pair
$I\udot$ restricts to stable pairs $I_1\udot$ and $I_2\udot$ 
over $X_1$ and $X_2$ respectively.
On $E\times\C$, $I\udot$ restricts to the ideal sheaf $\I_p$ of a point 
in the intersection $(\mathsf{s}\cap E)\times\C$. 
Tensoring \eqref{O12} with the 
perfect complex $R\hom(I\udot,I\udot)_0$ 
and taking sheaf cohomology gives 
the exact triangle on the bottom row of the following main diagram,
$$
 \spreaddiagramcolumns{-1pc} \xymatrix{
\LL^\vee_{\P_\eta/\B_\eta}[-1] \rto\dto & \LL^\vee_{\P_{\eta_1}/\B_{\eta_1}\,\times\,
\P_{\eta_2}/\B_{\eta_2}}[-1] \rto\dto & \LL^\vee_{\P_\eta/(\P_{\eta_1}\times\P_{\eta_2})}
\dto \\
\!R\!\Hom_{X[k]}(I\udot,I\udot)_0 \rto & \bigoplus_{i=1}^2R\!\Hom_{X_i}(I_i\udot,I_i\udot)_0
\rto & R\!\Hom_{E\times\C}(\I_p,\I_p)_0.}
$$
Here, $\B_{\eta_i}$ denotes the
stack of expanded degenerations of $(Y_i,E\times\C)$ decorated 
by $\eta_i$, so 
 $$\B_\eta
=\B_{\eta_1}\times\B_{\eta_2}\,.$$ 
The top row is the exact triangle of dual cotangent complexes 
for the fiber product structure \eqref{fibrep} relative to $\B_\eta$ (all restricted
to the point $I\udot\in\P$). 
The vertical arrows are the dual perfect obstruction theories 
provided by \eqref{dualobsthy}.

The last term of the top row has first cohomology
sheaf $(J/J^2)^\vee$ at $p$, where $J$ is the ideal of 
$$\P_\eta\subset\P_{\eta_1}\times\P_{\eta_2}\,.$$
Since $\P_\eta$ is the basechange of $\P_{\eta_1}\times\P_{\eta_2}$ to
 the diagonal in
$(E\times\C)^{2}$, the conormal bundle to the diagonal in
$(E\times\C)^{2}$ surjects onto $J/J^2$. 
The right hand vertical arrow is the dual
of this surjection (at $p$). The normal bundle to the diagonal,
$$T_p(E\times\C)\cong R\Hom_{E\times\C}(\I_p,\I_p)_0[1]\ ,$$
is identified with its
 image in $T_{(p,p)}\big((E\times\C)\times(E\times\C)\big)$ by the map 
$(1,-1)$. Comparing with the map $(1,-1)$ in \eqref{O12} shows 
 the diagram is commutative.

After composing the coboundary map of the exact triangle
occurring on  the bottom row
of the main diagram with the map \eqref{map2}
$$
\Ext^2(I\udot,I\udot)_0\To\C, 
$$
we obtain the following morphism
\begin{multline} \label{redcom}
\ T_pE\hookrightarrow T_p(E\times\C)\To R\Hom(I\udot,I\udot)_0[2]
\Rt{h^2}\\ \Ext^2(I\udot,I\udot)_0\To\C.
\end{multline}

\begin{prop} \label{propiso} The composition \eqref{redcom} is an isomorphism.
\end{prop}

\begin{proof}
The result is straightforward using the dual description \eqref{tangent}
\beq \label {dt}
\C\Rt{\partial_t}\Ext^1(I\udot,I\udot)_0
\eeq
of the map \eqref{map2}. 
The map is obtained from the deformation of $I\udot$ given by translation
in the $\C$-direction.

The dual of the sequence \eqref{O12} is
\beq\label{omega3}
0\to\omega_{X_1}\oplus\omega_{X_2}\to\omega_X\to\O_{E\times\C}\to0,
\eeq
where of course $\omega_X\cong\O_X$.

Tensoring with the perfect complex $R\hom(I\udot,I\udot)_0$ and taking sheaf cohomology gives the exact triangle
$$
\bigoplus_{i=1}^2R\Hom_{X_i}(I_i\udot,I_i\udot\otimes\omega_{X_i})_0\to
R\Hom_X(I\udot,I\udot)_0\to R\Hom_{E\times\C}(\I_p,\I_p)_0.
$$
Since the complex of sheaves $R\hom(I\udot,I\udot)_0$ is derived dual to itself (modulo
a shift), this exact triangle is the Serre dual of the bottom row of the main
diagram
 modulo a shift.

By the construction of the triangle using \eqref{omega3}, we see 
the deformation $\partial_t$ in ($H^1$ of) the second term maps 
to the corresponding
deformation $\partial_t$ in the third term. 
Hence, the composition of \eqref{dt} and
the second map in the exact triangle,
\beq\label{comred}
\C\Rt{\partial_t}\Ext^1_X(I\udot,I\udot)_0\to\Ext^1_{E\times\C}(\I_p,\I_p)_0=
T_p(E\times\C)\to T_p\C,
\eeq
is an isomorphism.

Under Serre duality, the splitting $T(E\times\C)\cong TE\oplus T\C$ is dual to the splitting
$T(E\times\C)\cong T\C\oplus TE$ in the opposite order, since the pairing
between the two spaces is by wedging (and the triviality of $\Lambda^2T(E\times\C))$.
Therefore, 
\eqref{comred} is precisely Serre dual to the composition \eqref{redcom}.
\end{proof}

After passing to absolute perfect obstruction theories,
the dual of the main diagram yields
\beq \label{dg7}
\xymatrix{
\E_1\udot\oplus\E_2\udot \rto\dto & \E_\eta\udot \rto\dto & \Omega_{E\times\C}[1] \dto \\
\LL_{\P_{\eta_1}\times\P_{\eta_2}} \rto & \LL_{\P_\eta} \rto &
\LL_{\P_\eta/(\P_{\eta_1}\times\P_{\eta_2})}\,,}
\eeq
at the point $[I\udot]\in\P_\eta$. Here $\E_i\udot$ is the absolute obstruction
theory for $\P_{\eta_i}$ derived in the usual way \eqref{dg1} from the relative obstruction
theory $\LL^\vee_{\P_{\eta_1}/\B_{\eta_1}}\to R\Hom_{X_i}(I_i\udot,I_i\udot)_0[1]$
of \eqref{dualobsthy}.

We have already shown the map \eqref{tangent} factors through the absolute
obstruction theory,
\beq \label{map3}
\C[1]\to\E_\eta\udot.
\eeq
By Proposition \ref{propiso}, the composition of \eqref{map3} and
$\E_\eta\udot\to\Omega_{E\times\C}[1]$ 
is isomorphic to the inclusion $\Omega_E\into\Omega_{E\times\C}$ (all shifted by $[1]$).
Since the latter is nontrivial, we can divide out the 
two top right hand terms of \eqref{dg7}
by  $\C\cong\Omega_E$ to give
\beq \label{fp}
\E_1\udot\oplus\E_2\udot\to\E_\eta^{\mathrm{red}}\to N^\vee[1],
\eeq
where $N^\vee$ is the conormal bundle of the diagonal $\C\into\C\times\C$.

\begin{thm}\label{rtrt}
The reduced virtual class of the moduli
space of pairs $\P$ of the degeneration $$\epsilon: \mathcal{X} \rarr B$$
with holomorphic Euler
characteristic $n$ and primitive class  
$$\beta = \mathsf{s} + h \mathsf{f}$$
satisfies three basic properties:
\begin{enumerate}
\item[(i)] For all $\xi \in B$,
 $$\iota_{\xi}^{!}[\P]^{\mathrm{red}} 
= [P_{n}(\mathcal{X}_\xi,\beta)]^{\mathrm{red}}\,.$$
\item[(ii)] For the special fiber,
 $$[\P_0]^{\mathrm{red}} 
= \sum_{\eta} \iota_{\eta*}[P_{\eta}]^{\mathrm{red}}\,.$$
\item[(iii)] The factorizations
$$[P_{n}(\mathcal{X}_0,\eta)]^{\mathrm{red}} =
[P_{n_1}((R_1/E)\times \com,\beta_1)]^{\mathrm{vir}} \times_\C
[P_{n_2}((R_2/E)\times \com,\beta_2)]^{\mathrm{vir}}$$
hold for $\beta_i= \mathsf{s}+h_i \mathsf{f}$.
\end{enumerate}
\end{thm}

\begin{rem}
By the fiber product
in (iii) we really mean $$\Delta^!\left([P_{n_1}((R_1/E)\times \com,\beta_1)]^{\mathrm{vir}}
\times
[P_{n_2}((R_2/E)\times \com,\beta_2)]^{\mathrm{vir}}\right),$$ where $\Delta\colon\C\to\C^2$
is the diagonal and both virtual cycles have obvious maps to $\C$.
However,
since all of the cycles involved can be taken to be linear combinations of products
of varieties with $\C$, this agrees with the resulting linear combination of fiber products
over $\C$.
\end{rem}

\begin{proof}
This follows the
proof of the parallel statements in \cite{LiWu} 
for the degeneration formula for the standard obstruction 
theory.
We use the compatibilities \eqref{7} and \eqref{8} of reduced obstruction theories 
in place of the standard compatibilities \eqref{dg3} and \eqref{dg4} of usual obstruction
theories. 
The statement (iii) is immediate from the exact triangle \eqref{fp}.
\end{proof}

Finally, we note all steps in the proof of Theorem \ref{rtrt}
respect the $\com^*$-action on
$$\mathcal{X} = \calS \times \com$$
given by scaling of the second factor. As a result, Theorem \ref{rtrt}
holds for the reduced and ordinary virtual classes in $\com^*$-equivariant
cycle theory.

\section{Reduced Gromov-Witten} \label{gwd}

\subsection{Stable maps to the fibers of $\epsilon$} We again work with the family
$\epsilon\colon\calS\to B$ of Section \ref{jjyy}. Denote the moduli
space of connected stable maps to the fibers of $\epsilon$,
as constructed in \cite{junli1, junli2}, by
\begin{equation}\label{ntt3}
\Mbar_g(\epsilon, \beta) \rightarrow B\,.
\end{equation}
Over nonzero $\xi \in B$,
the moduli space is simply $\Mbar_g(\calS_\xi,\beta)$.  
Over $0\in B$, the moduli space parameterizes
stable predeformable maps from a genus $g$ curve to an 
expanded target degeneration of $\calS_0=R_1\cup_E R_2$
of the form $$R_{1} \cup_{E} (E \times \mathbb{P}^1) \cup_{E}
\dots \cup_E (E \times \mathbb{P}^1)\cup_{E} R_{2}.$$  
Here, we have inserted a non-negative number of copies of 
$E \times \mathbb{P}^1$ at the singular locus of $\calS_0$, attached at $E \times \{0\}$ and 
$E \times \{\infty\}$.
We will denote the standard inclusion of the fiber over $\xi \in B$ by
$$\iota_\xi: \Mbar_g(\calS_\xi,\beta) \hookrightarrow \Mbar_g(\epsilon,\beta).$$

If  $\beta= \mathsf{s} + h \mathsf{f}$, then
the moduli space \eqref{ntt3}
has a simple structure. A stable map to any fiber $\calS_\xi$ in class $\beta$
has degree $1$ over the base of the elliptic fibration
$$\pi: \calS_{\xi} \rarr \proj^1\,.$$
Since the section $\mathsf{s}$ is rigid, the map consists of a fixed genus $0$ curve
mapping isomorphically to $\mathsf{s}$ attached to possibly higher genus curves mapping
to the fibers of $\pi$.
When $\xi=0$, the intersection of the genus 0 curve with the singular locus of the expanded
degeneration always 
has multiplicity $1$ at the distinguished point $p\in E$, see
 \eqref{pointp}.

\subsection{Relative maps} \label{relmapp}
Let $\beta = \mathsf{s} + h \mathsf{f}$, and let
 $\Mbar_{g}(R/E, \beta)$ denote the moduli space \cite{junli1} of 
stable relative maps to expanded degenerations of $(R,E)$ with multiplicity $1$ along the relative divisor $E$.
There is an evaluation map
\begin{equation}\label{L234}
\Mbar_{g}(R/E, \beta) \rarr E
\end{equation}
determined by the location of the relative point.
In fact, since $\mathsf{s}$ is rigid, the evaluation map {\em always} has value
 $p\in E$.

A stable map to $\calS_0$ can be split (non-uniquely) into relative stable maps to $(R_1,E)$ and $(R_2,E)$.  
Let 
$$\eta = (g_1,g_2,h_1,h_2)$$
denote a quadruple of non-negative integers satisfying 
$$g = g_1+g_2,\ \ \  h= h_1+h_2\,. $$  
Define $\Mbar_g(\calS_0, \eta)$ to be the fibered product over the evaluation maps \eqref{L234}
on both sides,
$$\Mbar_{g_1}(R_1/E_1, \mathsf{s}+ h_1 \mathsf{f}) \times_{E} \Mbar_{g_2}(R_2/E_2,
\mathsf{s}+ h_2 \mathsf{f}  )\,.$$
Since the evaluations maps both factor through the point $p\in E$, 
$$\Mbar_g(\calS_0, \eta) =
\Mbar_{g_1}(R_1/E_1, \mathsf{s}+ h_1 \mathsf{f}) 
\times \Mbar_{g_2}(R_2/E_2,\mathsf{s}+ h_2 \mathsf{f}  )\,. $$
By standard results \cite{junli1,junli2},
we have an embedding
$$\iota_{\eta}: \Mbar_g(\calS_0,\eta) \hookrightarrow \Mbar_g(\epsilon, \beta),$$
and the full moduli space to $\calS_0$ is the union
$$\Mbar_g(\calS_0,\beta) = \bigcup_{\eta} \Mbar_g(\calS_0,\eta).$$

In \cite{junli2}, the embedding $\iota_{\eta}$ is explicitly realized as given by a Cartier pseudo-divisor (in the sense of Fulton): there exists a line bundle $L_{\eta}$ on $\Mbar_g(\epsilon,\beta)$ 
with section $s_\eta \in \Gamma(L_\eta)$ whose zero locus is $\Mbar_g(\calS_0,\eta)$.  
If $(L_0,s_0)$ denotes the pseudo-divisor given by pulling back the Cartier divisor 
$0\in B$, we also have the identity
\begin{equation}\label{product}
(L_0, s_0) = \bigotimes_{\eta}\ (L_\eta,s_\eta).
\end{equation}

\subsection{Obstruction theories} \label{obth}

We follow the construction of the perfect obstruction theory 
on $\Mbar_g(\epsilon,\beta)$,
\begin{equation}
\EE_{\epsilon}\rightarrow \LLL_{\Mbar_g(\epsilon,\beta)}\ , \label{aaa4}
\end{equation}
presented in \cite{junli2}. 
We will always take $\beta = \mathsf{s} + h \mathsf{f}$.
Since the multiplicity of the stable map at the singular locus of $\calS_0$ 
is always $1$, the obstruction theory of \cite{junli2} simplifies substantially.

Let $\B$ denote the nonsingular Artin stack parameterizing expanded target degenerations of $\calS$ over $B$.
We first describe the relative obstruction theory for the morphism
$$\phi: \Mbar_g(\epsilon,\beta) \rightarrow \B \,.$$  

For convenience, we just describe the tangent and obstruction spaces at a closed point of the moduli space. For the general case see \cite{junli2}. 
 Fix $\xi \in B$, an expanded target degeneration{\footnote{Unless we
are over the special point $\xi=0$, we have 
$\widetilde{S}_\xi = \calS_\xi$. Nontrivial degenerations occur only over $0\in B$.}} 
$$\widetilde{S}_\xi \rightarrow \calS_\xi,$$
 and a stable map
$$f: C \rightarrow \widetilde{S}_\xi \,.$$  
The tangent and obstruction spaces relative to the morphism $\phi$ 
are given by
the cohomology groups in degrees $0$ and $1$ of the complex of vector spaces
\begin{equation} \label{RHomC}
R\Hom_{C}(f^*\Omega_{\widetilde{S}_\xi} \rightarrow \Omega_{C}, \calO_{C}).
\end{equation}
Here, the complex $f^*\Omega_{\widetilde{S}_\xi} \rightarrow \Omega_{C}$ 
is obtained from the map $f$ and is placed in degrees $-1$ and $0$.  
Following the method explained in Section \ref{tinder}, the absolute obstruction theory $\EE_{\epsilon}$ is then easily obtained from the relative
obstruction theory since $\B$ is nonsingular and $\Mbar_g(\epsilon,\beta)$ has no continuous
automorphisms.

We can similarly  construct perfect obstruction theories $\EE_0$ and $\EE_\eta$ for 
$\Mbar_g(\calS_0, \beta)$ and $\Mbar_{g}(\calS_0, \eta)$ respectively.  Just as in Section
\ref{tinder}, both  are related to 
$\EE_{\epsilon}$ via exact triangles (cf. \eqref{dg3} and \eqref{dg4})
\begin{equation}\label{compatibilityzero}
L_{0}^{\vee} \rightarrow  \iota_{0}^{*}\EE_{\epsilon}\rightarrow \EE_{0} \rightarrow
L_{0}^{\vee}[1]\ ,
\end{equation}
\begin{equation*} \label{compatibilityeta}
L_{\eta}^{\vee} \rightarrow  \iota_\eta^{*}\EE_{\epsilon}\rightarrow \EE_{\eta} \rightarrow
L_{\eta}^{\vee}[1]\ ,
\end{equation*}
where $L_0$ and $L_\eta$ are the line bundles in \eqref{product}.

In order to split the associated virtual class for $\EE_\eta$ into contributions from $R_1$ and $R_2$, 
we will require
 an exact triangle relating $\EE_\eta$ to the obstruction theories 
$\EE_{1}, \EE_{2}$ associated to each moduli space of relative stable maps
$\Mbar_{g_{i}}(R_i/E_i, \mathsf{s} + h_i \mathsf{f})$. This is the analogue of the exact
triangle \eqref{dg7} for stable pairs:
\begin{equation}\label{compatibilitysplitting}
N_{\Delta/E\times E}^{\vee} \rightarrow\EE_{1} \boxplus \EE_{2}\rightarrow \EE_{\eta} \rightarrow 
N_{\Delta/E\times E}^{\vee}[1].
\end{equation}
Here, $N_{\Delta/E\times E}^{\vee}$ denotes the conormal bundle to the diagonal of $E\times E$, 
pulled back to $\Mbar_g(\calS_0,\eta)$ via the evaluation maps.

\subsection{Reduced classes}\label{redcl}

The moduli space $\Mbar_g(\epsilon,\beta)$ carries
both an absolute obstruction theory \eqref{aaa4} and
an obstruction theory relative to $\B$ \eqref{RHomC}.
To define a reduced class as in Section \ref{rc}, we explain 
how to construct a 1-dimensional quotient of the 
relative obstruction space when $\beta = \mathsf{s} + h \mathsf{f}$.
We present a uniform treatment over all $\xi\in B$. However, unless $\xi=0$,
all structures involved with the singularities of $\calS_\xi$ are trivial.
 By an elementary analysis, the relative obstruction space over $\B$
equals the absolute
obstruction space since there is {\em no obstruction} to deforming
connected maps along
with deformations of the fibers of $\epsilon$.  
Hence, we obtain a 1-dimensional quotient of the
absolute obstruction theory of $\overline{M}_g(\epsilon,\beta)$ also.
\medskip

Let $\Omega_{\widetilde{S}_\xi}^{\mathrm{log}}$ denote the sheaf of differentials with logarithmic poles 
allowed along the singular locus
(the residues along each branch are required to add to zero).  
The sheaf $\Omega_{\widetilde{S}_\xi}^{\mathrm{log}}$ is
 locally free of rank $2$.  
Alternately, $\Omega_{\widetilde{S}_\xi}^{\mathrm{log}}$  is
 the sheaf of differentials for the log structure on $\widetilde{S}_\xi$ associated to the 
 smoothing of $\widetilde{S}_\xi$.  The surface $\widetilde{S}_\xi$ is log $K3$:
we have an isomorphism
$$\wedge^{2}\Omega_{\widetilde{S}_\xi}^{\mathrm{log}}\cong \calO_{\widetilde{S}_\xi}\,.
$$
Hence, there is a nondegenerate symplectic pairing on $\Omega_{\widetilde{S}_\xi}^{\mathrm{log}}$. 

Similarly, let $\Omega_{C}^{\mathrm{log}}$ denote the sheaf of differentials on $C$ 
with simple poles 
allowed only at the nodes mapping to the singular locus of $\widetilde{S}_\xi$
(again with the matching residue condition).  In a neighborhood of such nodes, 
$\Omega_{C}^{\mathrm{log}}$ is locally free.  
We may view $\Omega_{C}^{\mathrm{log}}$
as the sheaf of differentials associated to the log structure on $C$ pulled
back from 
$\widetilde{S}_\xi$ via $f$.  

\begin{lemma}\label{quasilog}
The natural map of complexes
$$[f^*\Omega_{\widetilde{S}_\xi}\rightarrow \Omega_C] \rightarrow 
[f^*\Omega_{\widetilde{S}_\xi}^{\mathrm{log}}\rightarrow \Omega_C^{\mathrm{log}}]$$
is a quasi-isomorphism.
\end{lemma}
\begin{proof}
The result follows from the diagram of exact sequences
$$
\spreaddiagramrows{-.8pc}
\xymatrix{
0 \rto & f^*\O_E \rto\ar@{=}[d] & f^*\Omega_{\widetilde S_\xi} \rto\dto &
f^*\Omega_{\widetilde S_\xi}^{\mathrm{log}} \rto\dto & f^*\O_E \ar@{=}[d] \rto & 0 \\
0 \rto & \O_p \rto & \Omega_C \rto & \Omega_C^{\mathrm{log}} \rto & \O_p \rto & 0.\!}
$$
The leftmost terms are the torsion subsheaves of the cotangent sheaves and the rightmost terms are the residues of the logarithmic forms.
\end{proof}

The reduced obstruction space is given by the kernel of the following composition of
 morphisms:
\begin{align*}
\psi: \text{Ext}^1([f^*\Omega_{\widetilde{S}_\xi} \rightarrow \Omega_{C}], \calO_{C})
&=
\text{Ext}^1([f^*\Omega_{\widetilde{S}_\xi}^{\mathrm{log}} 
\rightarrow \Omega_{C}^{\mathrm{log}}], \calO_{C})\\
&\cong \text{Ext}^1([(f^*\Omega_{\widetilde{S}_\xi}^{\mathrm{log}})^{\vee} 
\rightarrow \Omega_{C}^{\mathrm{log}}], \calO_{C})\\
&\rightarrow \text{Ext}^1([\omega_C^{\vee} \rightarrow  \Omega_{C}^{\mathrm{log}}], \calO_C)\\
&\rightarrow H^1(C, \omega_C) = \CC.
\end{align*}
Lemma \ref{quasilog} yields the first equality.
 The second equality is a consequence of the symplectic pairing.
 The third map is constructed via  the pull-back 
\begin{equation}\label{mmy}
  f^*(\Omega_{\widetilde{S}_\xi}^{\mathrm{log}}) \rarr \omega_C \,.
\end{equation}
The last map is obtained from the vanishing of the composition
$$\omega_{C}^{\vee} \rightarrow (f^*\Omega_{\widetilde{S}_\xi}^{\mathrm{log}})^{\vee}
\rightarrow f^*\Omega_{\widetilde{S}_\xi}^{\mathrm{log}} \rightarrow \Omega_{C}^{\mathrm{log}}\
.$$
Here we use the fact that the symplectic structure vanishes when pulled back to $C$.
\begin{lemma}\label{surjection}
The map $\psi$ is surjective.
\end{lemma}
\begin{proof}
For $\xi \ne 0$, the claim is part of the usual construction of the reduced class.
We discuss only the singular case  $\widetilde{S}_0$.
The map
 $\psi$ is induced by   
\begin{equation}\label{mmyy}
\text{Ext}^1\big((f^*\Omega_{\widetilde{S}_0}^{\mathrm{log}})^\vee , \calO_C\big) \cong
H^1\big(C,f^*\Omega_{\widetilde{S}_0}^{\mathrm{log}}\big) \rightarrow H^1(C,\omega_C),
\end{equation}
where the latter arrow is obtained from \eqref{mmy}.
We will prove \eqref{mmyy} is surjective.

Pick a connected component of the singular locus 
of $\widetilde{S}_0$ and take the corresponding separating 
node $p$ of $C$.  
Let $C_1$ and $C_2$ denote the connected components of the normalization of $C$ at $p$.  
Consider the sequence
$$0\rightarrow \calO_C \rightarrow \calO_{C_1} \oplus \calO_{C_2} \rightarrow 
\calO_p\rightarrow 0\,. $$
Tensoring with the bundles 
$f^*\Omega_{\widetilde{S}_0}^{\mathrm{log}}$
and $\omega_C$, taking cohomology and using the map \eqref{mmy} gives
the commutative diagram
$$
\spreaddiagramrows{-.5pc}
\xymatrix{
H^0\left(f^*\Omega_{\widetilde S_0}^{\mathrm{log}}\big|_{p}\right) \rto\dto &
H^1\big(C,f^*\Omega_{\widetilde S_0}^{\mathrm{log}}\big) \dto \\
H^0(\omega_C|_p) \rto & H^1(C,\omega_C)\,.\!}
$$
The left hand arrow is surjective since $C$ is an embedding in a neighborhood of $p$. The lower arrow is surjective by standard curve theory. Therefore the right hand arrow
\eqref{mmyy} is also surjective.
\end{proof}

By taking duals and applying the above construction in families, we obtain a morphism
$$\gamma: \calO_{\Mbar_g(\epsilon,\beta)}[1]\rightarrow \EE_{\epsilon}$$
for which the induced map on obstruction sheaves
$$h^1((\EE_\epsilon)^{\vee}) \rightarrow \calO_{\Mbar_g(\epsilon,\beta)}$$ is surjective.
We refer to $\EER_{\epsilon} = \mathrm{Cone}(\gamma)$ as the 
associated reduced obstruction theory, even though \emph{a priori} not all the conditions of an obstruction theory may be satisfied.\footnote{With
greater effort, standard reduced obstruction theories could
surely be constructed here. 
For families of nonsingular $K3$ surfaces, the most
difficult aspect of the construction of the reduced theory
is the obstruction study of Ran \cite{Ran} and Manetti \cite{Man}. In order
to avoid such a study for the broken $K3$ surface over
$0\in B$, we restrict ourselves to a weak reduced theory which
is technically simpler and sufficient for our
purposes.}
Just as in Section \ref{rc} the results of Kiem-Li allow us to construct a reduced virtual class from $\EER_{\epsilon}$.

\subsection{Degeneration of the reduced class}
We can define the reduced obstruction theory 
for $\Mbar_g(\calS_0,\beta)$ and $\Mbar_g(\calS_0,\eta)$ via the compositions
$$
\gamma_0: \calO[1]\rightarrow \iota^*_{0}\EE_{\epsilon}\rightarrow \EE_{0}\ , $$
$$
\gamma_{\eta}: \calO[1]\rightarrow \iota^*_{\eta}\EE_{\epsilon}\rightarrow \EE_{\eta}\,.
$$
From the exact triangles \eqref{compatibilityzero}, we deduce 
the induced co-section of the obstruction space is surjective in both cases.
We can take the cones of $\gamma_0$ and $\gamma_\eta$
to obtain the reduced obstruction theories
$\EER_{0}$ and $\EER_{\eta}$
respectively.
The compatibility statements:
\begin{equation}\label{reducedcompzero}
L_{0}^{\vee} \rightarrow  \iota_0^{*}\EER_{\epsilon}
\rightarrow \EER_{0} \rightarrow
L_{0}^{\vee}[1]\ ,
\end{equation}
\begin{equation*}
L_{\eta}^{\vee} \rightarrow  \iota_\eta^{*}\EER_{\epsilon}
\rightarrow \EER_{\eta} \rightarrow
L_{\eta}^{\vee}[1]\ ,
\end{equation*}
continue to hold.

\begin{lemma}\label{reducedcompsplit}
We have a quasi-isomorphism on $\Mbar_{g}(\calS_0,\eta)$,
$$\EE_{1} \boxplus \EE_{2}\cong \EER_{\eta},$$
compatible with the structure maps to the cotangent complex of $\Mbar_{g}(\calS_0,\eta)$.
\end{lemma}
\begin{proof}
By the second compatibility sequence \eqref{compatibilitysplitting} for 
$\EE_{\eta}$,
we have a natural map
\begin{equation}\label{jjjj}
\EE_1\boxplus\EE_2 \rightarrow \EER_{\eta}\,.
\end{equation}
If the induced map
\begin{equation}\label{nhh}
\calO[1] \rightarrow N_{\Delta/E\times E}^{\vee}[1]
\end{equation}
is a quasi-isomorphism, then \eqref{jjjj} is also a
quasi-isomorphism.

We prove the quasi-isomorphism statement
\eqref{nhh}
 after taking duals and passing to closed points.  
After fixing a curve, expanded target, and map 
$$f: C \rightarrow \widetilde{S}_0,$$ we
will prove the composition
\begin{equation}\label{ccb}
{T}_{p}E \rightarrow \text{Ext}^{1}
(f^*\Omega_{\widetilde{S}_0}^{\mathrm{log}},\calO_C)
\cong
 \text{Ext}^{1}
((f^*\Omega_{\widetilde{S}_0}^{\mathrm{log}})^\vee,\calO_C)
\rightarrow H^1(C,\omega_C)
\end{equation}
is an isomorphism.

As in the proof of Lemma \ref{surjection}, we will use the isomorphism
$$H^0(\omega_C|_p)\rightarrow H^1(C,\omega_C)$$
to factor the 
 composition in terms of the fibers of the vector bundles 
$f^*\Omega_{\widetilde{S}}^{\mathrm{log}}$ and $\omega_C$ over the point $p$.
In fact, the
composition \eqref{ccb}
 is the same as the composition 
\begin{equation}\label{vvp}
T_{p}E \rightarrow \left(\Omega_{\widetilde{S}_0}^{\mathrm{log}}\right)^{\vee}\Big|_{p}
\cong \Omega_{\widetilde{S}_0}^{\mathrm{log}}\big|_{p} \rightarrow \omega_{C}|_{p}\,.
\end{equation}
From an explicit local description of the symplectic pairing on 
$\Omega_{\widetilde{S}_0}^{\mathrm{log}}$,
the tangent direction along the elliptic fiber $E$ is
 identified with the dual of the tangent direction along the section $C$.  
As a consequence, composition \eqref{vvp} is an isomorphism.
\end{proof}

\begin{thm}\label{dmdm}
The reduced virtual class of maps to the degeneration $$\epsilon: \calS \rarr B$$
for primitive  $\beta = \mathsf{s} + h \mathsf{f}$
satisfies three basic properties:
\begin{enumerate}
\item[(i)] For all $\xi \in B$,
 $$\iota_{\xi}^{!}[\Mbar_{g}(\epsilon,\beta)]^{\mathrm{red}} 
= [\Mbar_{g}(\calS_\xi,\beta)]^{\mathrm{red}}\,.$$
\item[(ii)] For the special fiber,
 $$[\Mbar_g(\calS_0,\beta)]^{\mathrm{red}} 
= \sum_{\eta} \iota_{\eta*}[\Mbar_{g}(\calS_0,\eta)]^{\mathrm{red}}\,.$$
\item[(iii)] The factorizations
$$[\Mbar_{g}(\calS_0,\eta)]^{\mathrm{red}} = 
[\Mbar_{g_1}(R_1/E,\beta_1)]^{\mathrm{vir}} \times
[\Mbar_{g_2}(R_2/E,\beta_2)]^{\mathrm{vir}}$$
hold for $\beta_i= \mathsf{s}+h_i \mathsf{f}$.
\end{enumerate}
\end{thm}
\begin{proof}
The proof exactly follows the
proof of the parallel statements in J. Li's papers \cite{junli2} 
for the degeneration formula for the standard obstruction 
theory.
We use the compatibility of reduced obstruction theories \eqref{reducedcompzero}
 instead of \eqref{compatibilityzero}. 
 The statement (iii) is immediate from Lemma \ref{reducedcompsplit}.
\end{proof}

 \subsection{Non-primitive degeneration} \label{npd} 
We announce here a degeneration formula for arbitrary (not
necessarily primitive)
classes 
$$\beta  = m \mathsf{s} + h \mathsf{f}\ $$ 
with respect to the family 
$\epsilon$. Proofs and applications will appear
in \cite{mpt2}.

As before, 
a stable map to $\mathcal{S}_0$ can be split (non-uniquely) into
relative stable maps to $(R_1,E)$ and $(R_2,E)$.
The distribution of genera and fiber classes is given
by the data
$$\eta = (g_1,g_2,h_1,h_2)\,.$$
Let $\beta_i= m \mathsf{s} + h_i \mathsf{f}$ as before.
The additional data of an ordered partition $\gamma$ of $m$
$$\gamma= (\gamma_1, \ldots, \gamma_\ell), \ \ \sum_{i=1}^\ell \gamma_i =m$$
specifies the multiplicities with which $l$ distinct points on the two sides map to
the relative divisor.

The moduli spaces of relative maps{\footnote{The
superscript $\bullet$ denotes the moduli of maps
with possibly disconnected
domain curves. The map, however, is required to
be nonconstant on every connected component of the domain.
}}
$$\overline{M}^\bullet_{g_i}(R_i/E, \beta_i)_\gamma$$
are well-defined
and admit boundary evaluation maps
\begin{equation}\label{gtmm3}
\text{ev}_{\mathrm{rel}}:\overline{M}^\bullet_{g_i}(R_i/E, \beta_i)_\gamma \rarr
E^{\ell}\,.
\end{equation}
We define $\overline{M}^\bullet_{g}(\mathcal{S}_0,\eta)_\gamma$
by the fiber product of the boundary evaluations,
\begin{equation}\label{fred5}
\overline{M}^\bullet_{g}(\mathcal{S}_0,\eta)_\gamma
= 
\overline{M}^\bullet_{g_1}(R_1/E, \beta_1)_\gamma
\times_{E^{\ell}} \overline{M}^\bullet_{g_2}(R_2/E, \beta_2)_\gamma \,.
\end{equation}

After a construction of the reduced virtual class
for the family $\epsilon$,
Parts (i) and (ii) of Theorem \ref{dmdm} hold 
without change.
\begin{enumerate}
\item[(i)] For all $\xi \in B$,
 $$\iota_{\xi}^{!}[\Mbar_{g}^\bullet(\epsilon,\beta)]^{\mathrm{red}} 
= [\Mbar_{g}^\bullet(\calS_\xi,\beta)]^{\mathrm{red}}\,.$$
\item[(ii)] For the special fiber,
 $$[\Mbar^\bullet_g(\calS_0,\beta)]^{\mathrm{red}} 
= \sum_{\eta}\sum_{\gamma} \iota_{\eta,\gamma,*} 
\left(\frac{1}{|\text{Aut}(\gamma)|}
[\Mbar^\bullet_{g}(\calS_0,\eta)_\gamma]^{\mathrm{red}} \right)\, . $$
$\text{Aut}(\gamma)$ is the symmetry group
permuting equal parts of $\gamma$.
\end{enumerate}
 However,
the factorization rule (iii) is 
more interesting.

The class $\mathsf{s}$ determines, by intersection,
 a line bundle $\mathcal{L}$ of degree 1 on
the relative $E \subset R_1$.
Since the class $\mathsf{f}$ restricts to the trivial
line bundle on $E$,
the image of the boundary evaluation must lie in the
subvariety
$V \subset  E^{\ell}$ defined by
$$V= \left\{ \ (p_1, \ldots, p_\ell) \ | \ 
\mathcal{O}_E\Big(\sum_{i}^{\ell}\gamma_i p_i\Big) 
\cong {\mathcal{L}}^m \ \right\}\,.$$
The subvariety $V$ is nonsingular{\footnote{However,
$V$ need not be connected.}} of pure codimension 1.

The product of the boundary evaluation maps \eqref{gtmm3} yields
\begin{equation}\label{gtmm3f}
\text{ev}_{V\times V}:\overline{M}^\bullet_{g_1}(R_1/E, \beta_1)_\gamma 
\times 
\overline{M}^\bullet_{g_2}(R_2/E, \beta_2)_\gamma 
\rarr V \times V\,.
\end{equation}
Let $\Delta\subset V \times V$ be the diagonal. We also denote the inclusion map by $\Delta$.
By definition \eqref{fred5},
\begin{equation*}
\overline{M}_{g}^\bullet(\mathcal{S}_0,\eta)_\gamma
\cong \text{ev}^{-1}_{V\times V} (\Delta)\,.
\end{equation*}
The following rule replaces part (iii) of Theorem \ref{dmdm}.

\begin{enumerate}
\item[(iii)] The reduced virtual class of
$\Mbar^\bullet_{g}(\calS_0,\eta)_\gamma$ is given
by the Gysin pull-back of the standard virtual
classes of $\overline{M}_{g_i}(R_i/E, \beta_i)_\gamma$,
\begin{equation*}
[\Mbar_{g}^\bullet(\calS_0,\eta)_\gamma]^{\mathrm{red}} =
\Delta^! \Big(
[\overline{M}^\bullet_{g_1}(R_1/E, \beta_1)_\gamma]^{\mathrm{vir}} 
\times 
[\overline{M}^\bullet_{g_2}(R_2/E, \beta_2)_\gamma]^{\mathrm{vir}} 
\Big)\,.
\end{equation*}
\end{enumerate}
A detailed discussion will be given in \cite{mpt2}.

While the reduced Gromov-Witten invariants of $S$ with
disconnected domains vanish, the disconnected
theory is more natural for the factorization (iii).
Disconnected maps to $(R_1,E)$ and $(R_2,E)$
may glue to produce a connected map to $\calS_0$.

\section{Toric Gromov-Witten/Pairs correspondence}
\label{gwptor}

\subsection{Toric 3-folds}
Let $V$ be a nonsingular toric 3-fold acted upon by a
3-dimensional algebraic torus $\mathsf{T}$.
The conjectural Gromov-Witten/Pairs correspondence of 
\cite{pt1,pt2,pt3} for toric varieties equates the 
$\mathsf{T}$-equivariant Gromov-Witten theory of $V$ 
to
$\mathsf{T}$-equivariant stable pairs theory of $V$.
Since both sides can be defined by $\mathsf{T}$-equivariant
residues, $V$ is not required to be compact.
We refer the reader to Sections
3.1-3.2 of \cite{pt1} for background.

\begin{thm}
The $\mathsf{T}$-equivariant Gromov-Witten/Pairs correspondence
with primary field insertions 
holds for $V$.
\end{thm}

\begin{proof}
The $\mathsf{T}$-equivariant 
Gromov-Witten/Donaldson-Thomas correspondence with primary
field insertions
has been established for toric $V$ in \cite{moop}. The strategy of
\cite{moop} is to match the capped 
Gromov-Witten and Donaldson-Thomas vertices. The 1-leg case
was previously proven in \cite{GWDT} using \cite{BryP,OP6}.
The Gromov-Witten and Donaldson-Thomas theory
of $A_n$-resolutions  was studied in \cite{dmt,mo1,mo2}.
Once the local theories of $A_n$-resolutions are matched, the method of 
\cite{moop} is completely formal, relying only on properties of localization and degeneration
of the relevant moduli spaces.

To prove the Gromov-Witten/Pairs correspondence, we take 
precisely the same path. 
All the required results
for the Gromov-Witten capped vertex have already
been proven. Only the parallel results for the 
local stable pairs theory of $A_n$-resolutions
over curves are required.
For the latter, we follow, step by step, the
arguments from \cite{mo2,GWDT}.

Stable pairs have the same formal
properties (with respect to localization and
degeneration) as the Donaldson-Thomas theory
of ideal sheaves.  Moreover, the relative moduli spaces in each 
theory both involve the Hilbert scheme of points on a surface.  Thus, 
there is no difficulty in
following the arguments. Indeed, stable
pairs are simpler to treat since the support
is of pure dimension 1.  

Let us first review the basic steps in the proof of the Gromov-Witten/ Donaldson-Thomas correspondence, afterwards indicating which 
aspects
 require modification.

\begin{itemize}

\item \emph{Reduction to special geometries}

In \cite{moop}, a formal procedure
 is given to reduce 
invariants with primary insertions for arbitrary toric threefolds to relative invariants on the threefolds
\begin{equation*} 
\CC^2\times\proj^{1} \ \text{ and } \  A_n\times\proj^1 \text{ for
 $n=1,2$\ .}
\end{equation*}
The arguments in \cite{moop} rely on the degeneration formalism, localization for relative moduli spaces, and dimension counts - all of which behave the same for stable maps, ideal sheaves, and stable pairs.  In particular, the formal procedure
 can be applied to stable pairs to reduce the correspondence to the case of these specific 3-folds.

We view $\C^2$ as $A_0$. 
For the three special geometries, 
the proofs in \cite{GWDT,mo2} are based on the  following
principles.

\item \emph{Vanishing results}

Let $\sigma$ be a holomorphic symplectic
form on $A_n$ fixed by a 1-dimensional torus action on $A_n$.
Let $$\mathsf{T}_0\subset \mathsf{T}$$ 
be the 2-dimensional subtorus which
acts on the toric 3-fold
$A_n \times \proj^1$
and fixes the pull-back of $\sigma$.
\begin{enumerate}
\item[(i)] The $\mathsf{T}_0$-equivariant theory of
$A_n \times \proj^1$ 
vanishes unless the curve class $\beta$ is 
contracted over $A_n$.
\item[(ii)] In case the curve class $\beta$ is
contracted over $A_n$, the $\mathsf{T}_0$-equivariant theory of
$A_n \times \proj^1$ vanishes unless
the holomorphic Euler characteristic $\chi$ is minimal.
\end{enumerate}

\item \emph{Reduction to specific computations}

Using the vanishing results (i-ii), 
the full theory of
$A_n\times\proj^1$ is reduced to the calculation of specific 2-point 
invariants on the relative theory of $A_n\times\proj^1$
modulo $(s_1+s_2)^{2}$, where $s_1+s_2$ is the equivariant weight of the holomorphic symplectic form $\sigma$.  For $\CC^2$, the
argument is given in \cite{GWDT}. 
For $A_n$, the argument is given in \cite[Propositions 4.3 and 4.4]{mo2}.
In each case, the proof consists of comparing the Nakajima and fixed-point bases of $\text{Hilb}(A_n)$ and applying the vanishing statement along with certain dimension counts.
Again the argument is completely formal since relative insertions 
behave in the same manner for stable pairs and ideal sheaves.

\item \emph{Reduction to Hilbert scheme geometry}

The last step is to prove that the specific calculations
discussed above
can be identified with 2-point invariants
in the quantum cohomology of $\text{Hilb}(A_n,d)$ via a 
basic  matching result.

Consider the moduli space of genus 0, 2-pointed stable maps
to the Hilbert scheme of $d$ points of $A_n$ of curve class
$\beta$, 
$$\overline{M}_{0,\{0,\infty\}}
(\text{Hilb}(A_n,d), \beta)\ . $$
There is an open set
$$U_{n,d,\beta} \subset \overline{M}_{0,\{0,\infty\}}
(\text{Hilb}(A_n,d), \beta)$$
corresponding to the locus
where the domain is a simple chain of rational curves.
\begin{enumerate}
\item[(iii)] The canonical map from $U_{n,d,\beta}$
to the moduli space of ideal sheaves of   
$A_n \times \proj^1$- rubber relative  to
$0,\infty \in \proj^1$ is an isomorphism which
respects the obstruction theory.
\end{enumerate}

For any 2-point invariant, we study
 the contribution of virtual localization 
to any given fixed-point locus. 
If the locus does not lie in the open set of (iii), 
then the vanishing results imply the
contribution is divisible by $(s_1+s_2)^{2}$ and
 so is irrelevant for the computation.
 
\end{itemize}

To complete the proof of the Gromov-Witten/Pairs
correspondence, we follow the same 4 steps with stable pairs instead
of ideal sheaves.  Since the endpoint of the argument is 
 given by calculating
2-point invariants on $\text{Hilb}(A_n,d)$, the correspondence 
is independent of which sheaf theory we use.

All arguments in the above outline are formal except for 
the vanishing statements and the matching with the 
Hilbert scheme geometry.  In particular, it remains
to establish properties (i-iii) for
the stable pairs theory of the 3-fold $A_n\times \proj^1$.

For (i), a direct approach is obtained by constructing
a trivial quotient of the obstruction sheaf following Section \ref{spd}.
Property (ii) is equivalent to a vanishing for
the 1-leg stable pairs vertex which can be checked explicitly
from the localization formulas of \cite{pt2}. The proof is
presented in Section \ref{fff3}.
Property (iii) is immediate since  $U_{n,d,\beta}$
maps to the locus where the moduli of ideal sheaves
is isomorphic to the moduli of stable pairs.
\end{proof}

\subsection{1-leg stable pairs vertex}
\label{fff3}
We prove the necessary divisibility statement for stable pairs 
invariants for 
$$X = \CC^2\times\PP^1\ .$$  
The first two factors of the torus 
$\mathsf{T} = (\CC^*)^{3}$
act on $\CC^2$ by scaling the coordinates, and the last  factor
 acts on $\proj^1$ in the standard manner.
The
stable  invariants take values in
the equivariant cohomology ring $\QQ(s_1,s_2,s_3)$.  
Via localization \cite{pt2}, 
the contribution of the virtual class 
can be decomposed into vertex and edge terms associated to 
the moment polytope of $X$.
Given a partition $\mu$, the vertex contribution is 
given by a Laurent series
$$\mathsf{W}_{\mu}^{P}(q; s_1,s_2,s_3) \in \QQ(s_1,s_2,s_3)(\!(q)\!).$$

\begin{lemma}
The $q^{n}$ coefficient of $\mathsf{W}_{\mu}^{P}$ is divisible by
$s_1+s_2$ for $n > |\mu|$.
\end{lemma}
\begin{proof}
The proof here is nearly identical to 
the calculation in \cite{GWDT} and \cite{mnop2}
 with a slightly different combinatorial formula.
We again follow notation from \cite{pt2}.

By definition, $\mathsf{W}_{\mu}^{P}(q;s_1,s_2,s_3)$ is  a weighted sum 
$$\mathsf{W}_{\mu}^{P}(q;s_1,s_2,s_3) = \sum_{\mathcal{Q}}\mathsf{w}_{\mathcal{Q}}(s_1,s_2,s_3)q^{|\mathcal{Q}|}$$
over $\mathsf{T}$-invariant stable pairs $\mathcal{Q}$ on $\CC^{3}$ with limiting profile in the $x_{3}$- direction given by the subscheme of $\CC^{2}$ defined by monomial ideal $\mu[x_1,x_2]$ associated to $\mu$.  We prove the divisibility statement for each summand with $|\mathcal{Q}| > |\mu|$.

Such a $\mathsf{T}$-invariant stable pair $\mathcal{Q}$ has the following description.
Let 
$$M = \CC[x_3,x_{3}^{-1}]\otimes\frac{\CC[x_1,x_2]}{\mu[x_1,x_2]}\ .$$
We have a finitely generated $\mathsf{T}$-invariant $\CC[x_1,x_2,x_3]$-module 
$$\CC[x_3]\otimes\frac{\CC[x_1,x_2]}{\mu[x_1,x_2]} \subset 
{Q} \subset M.$$
Let $F(t_1,t_2,t_3)$ denote the Laurent series 
indexing the torus weights of ${Q}$,  and let 
$G(t_1,t_2)$ denote the same for $\frac{\CC[x_1,x_2]}{\mu[x_1,x_2]}$.   
The equivariant vertex contribution for 
$\mathcal{Q}$ is obtained from the Laurent polynomial
\begin{align*}
H(t_1,t_2,t_3)= F - \frac{\overline{F}}{t_{1}t_{2}t_{3}} &+ F\overline{F}\prod_{i=1}^{3}\frac{(1-t_{i})}{t_i} \\ 
&- \frac{1}{1-t_{3}}\left(G + \frac{\overline{G}}{t_{1}t_{2}} - G\overline{G}\frac{(1-t_1)(1-t_2)}{t_{1}t_{2}}\right)
\end{align*}
where $\overline{F} = F(1/t_{1},1/t_{2},1/t_{3})$ and similarly for $G$.  More precisely, $\mathsf{w}_{\mathcal{Q}}(s_1,s_2,s_3)$ is a product 
of linear factors associated to monomials in the above
 expression.
The divisibility of the
 vertex contribution by $s_1+s_2$ is equivalent 
to the negativity of 
the constant term in $H(t, t^{-1},u)$.

The $\mathsf{T}$-weights of $Q$ define a labelled box configuration, which can be described by a sequence of skew Young diagrams of the form
$$\rho_{k} = \mu\backslash \nu_{k}$$
for Young diagrams $\nu_{k} \subset \mu$ satisfying inclusions
$$\emptyset = \rho_{-m} \subset \rho_{-m+1} \cdots \subset \rho_{-1} \subset \rho_{0} = \rho_{1} = \cdots = \mu\ .$$
Here, $\rho_{k}$ contains a box $(a,b) \in \ZZ_{\geq 0}^{2}$ if $(a,b, k)$ is a weight of $Q$.
We have 
$$|\mathcal{Q}| = \sum_{k=-m}^{-1} |\rho_{k}|\ .$$ 
If $|\mathcal{Q}| > |\mu|$, we must have $\rho_{-1}\neq 0$.

Let $c_{r}(\rho_k)$ denote the number of boxes $(a,b) \in \rho_{k}$ for which $a-b = r$, and let
$d_r(\rho_k) = c_r(\rho_k) - c_{r+1}(\rho_k)$.
The constant term of $H(t,t^{-1},u)$ is given by
\begin{equation}\label{h677}
-c_{0}(\rho_{-1}) + \frac{1}{2}\sum_{r}\left(d_{r}(\rho_{0})^{2} - \sum_{k \leq 0}
(d_{r}(\rho_k)-d_r(\rho_{k-1}))^{2}\right).
\end{equation}
Since for each $r$, we have
$$\sum_{k\leq 0} d_{r}(\rho_{k})-d_{r}(\rho_{k-1}) = d_{r}(\rho_0) \ \in \{-1,
0,1\}\ , $$
the second term of \eqref{h677} is non-positive, 
and the entire expression is bounded above by 
$-c_{0}(\rho_{-1}) \leq 0$.

There are two cases.
If $c_{0}(\rho_{-1}) > 0$, then we are done.
If not,  since $\rho_{-1} \neq \emptyset$, 
there must be a box of $\rho_{-1}$ which minimizes
$|r|$. If $a>b$ for such a box, then
$$c_{r-1}(\rho_0)= c_r(\rho_0)$$
and $d_{r-1}(\rho_0)=0$. But,
$$c_{r-1}(\rho_{-1})< c_r(\rho_{-1})$$
so $d_{r-1}(\rho_{-1})<0$
and the second term of \eqref{h677} is strictly
negative. A similar logic applies if the 
$|r|$-minimal box satisfies $a<b$.
\end{proof}

\subsection{Proof of Theorems \ref{ooo} and \ref{redgwpt}}
Let $S$ be a $K3$ surface with primitive effective curve class 
$\beta \in H_2(S,\mathbb{Z})$ satisfying
$$\langle \beta, \beta \rangle = 2h-2\,.$$
Following the definitions of Section \ref{ggww}, 
to establish Theorem \ref{redgwpt}, we must prove:
\begin{enumerate}
\item[(i)] $\mathsf{Z}_{\beta}^{P}\left( \tau_{0}(\mathsf{p})^k\right)$
 is a rational function of $y$.
\item[(ii)]
After the variable change $-e^{iu}=y$,  
$$\mathsf{Z}_{\beta}^{GW}\left( \tau_{0}(\mathsf{p})^k\right)  = \mathsf{Z}_{\beta}^{P}
\left( \tau_{0}(\mathsf{p})^k\right)\,.$$
\end{enumerate}

We consider both the reduced Gromov-Witten and the reduced
stable pairs theory
of the 3-fold 
$$X = S \times \com\ $$
equivariant with respect to the scaling of $\com$.
Properties (i) and (ii) above are
reduced Gromov-Witten/Pairs correspondence for
$X$ with point insertions.

Let $R$ be the rational elliptic surface considered in
Section \ref{rez}, and let
$$Y = R \times \com\,.$$
By Theorems \ref{rtrt} and \ref{dmdm}, the reduced theory of $X$
may be calculated on the broken 3-fold
$$Y \cup_{E\times \com} Y \,.$$
The original $k$ point conditions are distributed
to two sides (the precise distribution will not matter). 
In order to prove the reduced Gromov-Witten/Pairs correspondence
for $X$, we need only establish the
 standard Gromov-Witten/Pairs correspondence
for the relative geometry $(Y,E\times \com)$ equivariant
with respect to the scaling of $\com$.

Since $R$ is deformation equivalent to a toric surface,
the toric Gromov-Witten/Pairs correspondence
implies the Gromov-Witten/Pairs correspondence
for $Y$ with primary field insertions
 equivariant with respect to the scaling of $\com$.

The final step is to deduce the Gromov-Witten/Pairs
correspondence for $(Y,E\times \com)$ from the
established correspondence for $Y$.
Let 
$$Z = E \times \proj^1 \times \com$$
and let $E\times \com \subset Z$ be a fiber over $\proj^1$.
The degeneration
\begin{equation}\label{ht334}
Y\,\rightsquigarrow\ Y\cup_{E\times \com} Z
\end{equation}
is obtained from the deformation to the normal cone of $E\times \com \subset Y$.
We will prove the correspondence for $(Y,E\times \com)$
from the correspondences for $Y$ and $(Z, E\times \com)$
via the degeneration formula.
In the degeneration \eqref{ht334}, all point
insertions are kept on $Y$.

The relative geometry $(Z,E\times \com)$ with no point insertions
is very simple.
We are only interested in curve classes on $Z$ which are degree 1
over the $\proj^1$ factor. The dimension of the moduli spaces on both the 
Gromov-Witten and stable pairs sides is 2. Hence, a point condition
in the relative divisor $E\times \com$ {\em must} be imposed
to produce nonvanishing invariants (which are then nonequivariant
constants). The theories of $(Z,E\times \com)$ equivariant
with respect to the scaling of $\com$ are equal to the
corresponding nonequivariant theories
of 
\begin{equation}\label{kkzzw}
(E \times \proj^1 \times E, E \times E)\,.
\end{equation}
The curve classes we are considering are degree 0 over the last
$E$ factor. Since $E \times E$ is holomorphic symplectic, the
reduced class constructions on both the Gromov-Witten
and stable pairs sides lead to vanishing unless the curve class is
also degree 0 over the first $E$ factor.
Finally, the match between the Gromov-Witten and stable
pairs theory for the relative geometry \eqref{kkzzw}
is a consequence of the Gromov-Witten/Pairs correspondence
for local curves.

We have now proven the Gromov-Witten correspondence
in the relevant curve classes for two out of the three
geometries in the degeneration \eqref{ht334}. Moreover,
the calculation for $(Z, E\times \com)$ indicated
above yields invertibility of the corresponding factor
in the degeneration formula in both Gromov-Witten and
stable pairs theory. We conclude the required Gromov-Witten/Pairs
correspondence for $(Y, E\times \com)$ holds.
We have completed the proof of Theorem \ref{redgwpt}.

By Lemma \ref{gth}, Lemma \ref{gthr},
and the Euler characteristic calculations of Kawai-Yoshioka,
Theorem \ref{ooo}
 is a consequence of Theorem \ref{redgwpt}
with no point insertions,
see Section \ref{kkkcon2}. \qed

\subsection{Proof of Corollary \ref{ttt}}
Let $\calS(u)= \frac{\sin(u/2)}{u/2}$.
We have
\begin{align*}
\sum_{g,h \geq 0} 
R_{g,h} u^{2g-2}q^{h-1} &= \sum_{g,h\geq 0}
r_{g,h} (u\calS(u))^{2g-2}q^{h-1}\\
&= (u\calS(u))^{-2}\sum_{g,h\geq 0} r_{g,h} (-1)^{g} ( e^{iu/2} - e^{-iu/2} )^{2g}q^{h-1}\\
&=(u\calS(u))^{-2}\frac{1}{\Delta(e^{iu},q)}\,.
\end{align*}
The first two equalities are by definition. The third is consequence of
Theorem \ref{ooo}.

Corollary \ref{ttt} is then a consequence of the following identity,
\begin{align*}
\log\left(\frac{1}{\calS(u)^{2}}\cdot\frac{\Delta(q)}{\Delta(e^{iu},q)}\right) &=
-2\log\calS(u) +  \sum_{n\geq 1}4\log(1-q^n)
\\ &  \ \ \ \ \  - 2 \sum_{n\geq 1}\left(
\log(1-e^{iu}q^n) +\log(1-e^{-iu}q^n)\right)\\
&=-2\log\calS(u) + 4\sum_{n,g,k \geq 1} \frac{(-1)^{g}k^{2g-1}}{(2g)!}u^{2g}q^{kn}\\
&=\sum_{g\geq1} u^{2g} \frac{(-1)^{g+1}B_{2g}}{g\cdot(2g)!} E_{2g}(q)\,.
\end{align*}
For the last equality, we have used the expansion
$$\log\calS(u) = \sum_{g\geq 1} \frac{(-1)^g B_{2g}}{2g\cdot(2g)!} u^{2g}\,.$$
Note also $(-1)^{g+1}B_{2g} = |B_{2g}|$ for $g\geq 1$. \qed



\section{Point insertions}
\label{pint}

\subsection{Overview}

Our strategy for proving Theorem \ref{kkvpoints} governing
point insertions is by degeneration.
We take the $K3$ surface, as before,
to be fibered
$$\pi: S \rarr \proj^1$$
with a section, and we take the primitive curve class to be of the form
$$\beta = \mathsf{s} + h\cdot \mathsf{f}.$$
Fix a nonsingular elliptic fiber $E$ of $\pi$ and consider the degeneration 
to the normal cone of $E$,
$$S_{0} = S \cup_{E} (E \times \PP^1).$$

We will use the degeneration to the normal cone
to reduce the integrals
of Theorem \ref{kkvpoints} to Theorem \ref{ooo} and 
calculations in the relative Gromov-Witten theory of $E \times \PP^1$.

\subsection{Degeneration formula}
We view $E\times \PP^1$ as elliptically fibered
$$\pi: E\times \PP^1 \rarr \PP^1\ $$
with a section (contracted over $E$).
Let $\mathsf{s},\mathsf{f} \in H^2(E\times \PP^1,\mathbb{Z})$ be the
classes of the section and fiber respectively.
Let 
$$\overline{M}_{g,r}((E\times \PP^1)/0, \mathsf{s}+h \mathsf{f}), \ \ \
\overline{M}_{g,r}((E\times \PP^1)/\{0, \infty\},\  \mathsf{s}+h \mathsf{f})$$
denote the moduli spaces of stable relative maps to the 
targets
$$(E\times \PP^1)/ E_0, \ \ \  (E\times \PP^1)/(E_0 \cup E_\infty)\,.$$
 Here, $E_0, E_\infty \subset E\times \PP^1$
are the fibers over $0,\infty \in \PP^1$.

Since the class $\mathsf{s}+ h \mathsf{f}$ has intersection number 1
with the fibers $E_0$ and $E_\infty$, the relative conditions
are given by cohomology classes on $E$. We will only
require the identity and point classes
$$\mathsf{1}, \omega \in H^*(E,\mathbb{Z})\,.$$
We denote the relative conditions by subscripts after the
relative moduli space. For example,
$$\overline{M}_{g,r}((E\times \PP^1)/\{0,\infty\},\ \mathsf{s}+h \mathsf{f})
_{\omega,\mathsf{1}}$$
denotes the moduli space with the relative conditions $\omega$
and $\mathsf{1}$ 
imposed
over $0$ and $\infty$ respectively.

We will use bracket notation for the relative invariants of
$E\times \PP^1$. The insertion $\hodge^\vee(1)$
stands for 
the Chern polynomial of the dual of the Hodge bundle,
\begin{multline*}
\Big\langle \omega \ \Big|\ \hodge^{\vee}(1)\ \tau_0(\mathsf{p})\ \Big|\ \mathsf{1}
 \ \Big\rangle^{(E \times \PP^{1})/\{0, \infty\}}_{g,\mathsf{s}+h\mathsf{f}} = \\
\int_{[\overline{M}_{g,r}((E\times \PP^1)/\{0,\infty\},\  \mathsf{s}+h \mathsf{f})
_{\omega,\mathsf{1}}]^{vir}}
\Big(1-\lambda_1 + \lambda_2 - \ldots +(-1)^g \lambda_g\Big) \cdot
\text{ev}_1^*(\mathsf{p}) \,.
\end{multline*}
Denote the generating function of relative invariants by
\begin{multline*}
\Big\langle \omega\  \Big|\  \hodge^{\vee}(1)\ \tau_0(\mathsf{p})\ \Big|
\ \mathsf{1}\ \Big \rangle^{(E \times \PP^{1})/\{0,\infty\}} =\\ 
\sum_{g=0}^\infty \sum_{h=0}^\infty 
\Big\langle \omega\ \Big|\ \hodge^{\vee}(1)\ 
\tau_0(\mathsf{p})\ \Big|\ \mathsf{1}
\Big \rangle^{(E \times \PP^{1})/\{0,\infty\}}_{g, \mathsf{s} + 
h\cdot \mathsf{f}} u^{2g-2} q^{h-1}\,.
\end{multline*}
The integrals of Theorem \ref{kkvpoints} may be written as
\begin{eqnarray*}
\Big\langle \hodge^\vee(1) \ \tau_0(\mathsf{p})^{k}\Big\rangle^{S} & = &
\sum_{g=0}^\infty \sum_{h=0}^\infty  
\Big\langle \hodge^\vee(1) \ \tau_0(\mathsf{p})^{k}\Big\rangle^{S}_{g, h} u^{2g-2}
q^{h-1} \\
& = & 
\sum_{g= 0}^\infty \sum_{h = 0}^\infty
\Big\langle (-1)^{g-k} \lambda_{g-k} \ \tau_0(\mathsf{p})^{k}\Big\rangle^{S}_{g, h} u^{2g-2}
q^{h-1}\,. 
\end{eqnarray*}

\begin{prop}\label{bubbleoffpoints} We have
\begin{multline*}
 \Big\langle \hodge^\vee(1) \ \tau_0(\mathsf{p})^{k}\Big\rangle^{S}
= \\
\Big\langle \hodge^\vee(1) \Big\rangle^{S} 
(u^2q)^k
\left( 
\Big\langle \omega\  \Big|\  \hodge^{\vee}(1)\ \tau_0(\mathsf{p})\ \Big|
\ \mathsf{1}\ \Big \rangle^{(E \times \PP^{1})/\{0,\infty\}}
\right)^{k}\,.
\end{multline*}
\end{prop}

\begin{proof}
For the degeneration of $S$ to $S_0$, 
there already exists a degeneration formula for the reduced Gromov-Witten theory of $S$ 
proven by Lee and Leung \cite{jleec1,jleec2}.{\footnote{An
algebraic proof can be obtained by the methods of Section \ref{gwd}.}}
The formula expresses the reduced invariants of $S$ in terms of 
the reduced relative Gromov-Witten theory of the pair $(S,E)$ and 
the standard relative Gromov-Witten theory of the pair
$(E \times \PP^1, E)$.  

We denote the generating series of reduced relative invariants of $S/E$
by
$$\Big\langle \hodge^{\vee}(1)\ \Big|\  \mathbf{1}\ \Big\rangle^{S/E} 
= \sum_{g=0}^\infty \sum_{h=0}^\infty 
\Big\langle \hodge_{g}^{\vee}(1)\ \Big|
\ \mathsf{1} \ \Big\rangle^{S/E}_{g,h}  u^{2g-2}q^{h-1}\,.$$
The degeneration formula then gives the identity
\begin{multline}\label{hhtt4}
 \Big\langle \hodge^\vee(1) \ \tau_0(\mathsf{p})^{k}\Big\rangle^{S} = \\
\Big\langle \hodge^{\vee}(1)\ \Big|\ \mathsf{1}\ \Big\rangle^{S/E}  u^2q\ 
\Big\langle \omega\  \Big|\  \hodge^{\vee}(1)\ \tau_0(\mathsf{p})^{k} \Big
\rangle^{(E \times \PP^{1})/0}\,.
\end{multline}
For the first term on the right, we easily see
$$\Big\langle \hodge^{\vee}(1)\ \Big|\ \mathsf{1}\ \Big\rangle^{S/E}=
\Big\langle \hodge^\vee(1) \Big\rangle^{S}.$$
Using the degeneration of $E\times \PP^1$ to 
$$(E \times \PP^1) \cup_{E} (E\times\PP^1) \cup \dots \cup_{E} (E \times \PP^1)\ ,$$
the second invariant on the right side
of \eqref{hhtt4} can be calculated,
\begin{multline} \label{hhtt55}
\Big\langle \omega\  \Big|\  \hodge^{\vee}(1)\ \tau_0(\mathsf{p})^{k}
\Big\rangle^{(E \times \PP^{1})/0} = \\
\left( \Big\langle \omega \  |\ 
 \hodge^{\vee}(1)\ \tau_0(\mathsf{p})\ \Big|\ \mathsf{1}\ 
 \Big\rangle^{(E \times \PP^{1})/\{0,\infty\}} 
\right)^{k} (u^2q)^k \Big\langle \omega\ \Big|\ 
\hodge^{\vee}(1)\ \Big\rangle^{(E\times\PP^{1})/0}\,.
\end{multline}
The Proposition then follows from  Lemma \ref{vanishing}
below applied to equations \eqref{hhtt4}-\eqref{hhtt55} for $k \geq 0$. 
\end{proof}

\begin{lemma}\label{vanishing} The following two series are trivial:
$$\Big\langle \omega \ \Big|\ \hodge^{\vee}(1) \Big
\rangle^{(E\times\PP^{1})/0} = \frac{1}{u^2q}\ , \ \ \ \ \
\Big\langle \ \mathsf{1}\ \Big|\
 \hodge^{\vee}(1)\ \tau_0(\mathsf{p})\ \Big \rangle^{(E\times\PP^{1})/0} = 
\frac{1}{u^2q}\,.$$
\end{lemma}
\begin{proof}
We only calculate the first. The argument for the second is
the same.
For dimension reasons, the Hodge class insertion is $(-1)^{g}\lambda_{g}$.
Consider the curve class $\mathsf{s} + h  \mathsf{f}$.  

\begin{enumerate}
\item[$\bullet$]
If $h = 0$, the invariant 
can be expressed in terms of the relative
 Gromov-Witten theory of $\PP^{1}$
 with an obstruction bundle term
 which is given by another $(-1)^g\lambda_{g}$ insertion.  
Since $\lambda_{g}^{2}=0$ for $g > 0$, only the $g=0$ term survives.
\item[$\bullet$]
If $h >0$, we can express the invariant in terms of 
the relative Gromov-Witten theory of  
$E \times E\times \PP^1$, with degree $(0,h,1)$, 
where again the obstruction bundle associated to the degree $0$ 
direction contributes the original $(-1)^g\lambda_g$ insertion.  
However, since the Gromov-Witten theory of abelian surfaces is 
trivial for noncontracted curves, all such terms vanish.
\end{enumerate}

Only the $g=0$ and $h=0$ terms of the series
$\Big\langle \omega \ \Big|\ \hodge^{\vee}(1) \Big
\rangle^{(E\times\PP^{1})/0}$ are nonvanishing. Direct calculation
shows the nonvanishing term is 1.
\end{proof}

The next result replaces the relative invariants of $E\times\PP^1$ 
in Proposition \ref{bubbleoffpoints}
with absolute invariants.

\begin{lemma}\label{absoluterelative} We have
$$\Big\langle \hodge^{\vee}(1)\ \tau_0(\mathsf{p})^{2} \Big
\rangle^{E\times\PP^1} = 
2 \Big\langle \omega \ \Big|\  \hodge^{\vee}(1)\ \tau_0(\mathsf{p})\ 
\Big|\ \mathsf{1} \ \Big\rangle^{(E \times \PP^{1})/\{0,\infty\}}\,.$$
\end{lemma}

\begin{proof}
We degenerate $E \times \PP^{1}$ to two copies 
of $E \times \PP^{1}$ attached along a fiber $E$ with
the two insertions $\tau_0(p)$ sent to different sides.  
Since we only consider curve classes intersecting the relative fiber once, 
the configurations in the degeneration formula
 are associated to the diagonal 
splittings of cohomology classes along the relative point.  
For parity reasons, we only need to consider even cohomology, 
so splittings must be of type $(\mathsf{1},\omega)$ or 
$(\omega,\mathsf{1})$.  The two 
configurations are clearly symmetric, 
and the second claim of Lemma \ref{vanishing} yields
$$\Big\langle \hodge^{\vee}(1)\ \tau_0(\mathsf{p})^{2} \Big
\rangle^{E\times\PP^1} = 
2 \Big\langle \omega\  \Big|\  \hodge^{\vee}(1)\ \tau_0(\mathsf{p})
\Big\rangle^{(E \times \PP^{1})/0}\,.$$
The Lemma then follows from applying the first claim of Lemma \ref{vanishing} to the degeneration of $E \times \PP^1$
along the divisor $E \times \infty$.
\end{proof}

\subsection{Proof of Theorem \ref{kkvpoints}}

Using Corollary \ref{ttt}, 
Proposition \ref{bubbleoffpoints}, and Lemma \ref{absoluterelative}, 
the proof of Theorem \ref{kkvpoints} is reduced to the following
two results.
\begin{lemma} We have
$$\Big\langle 
\hodge^{\vee}(1)\ \tau_0(\mathsf{p})^{2}\Big \rangle^{E\times\PP^1}
= \frac{2}{u^2q}\
 \sum_{m=1}^\infty  q^{m} \sum_{d | m} \frac{m}{d} 
\Big(2 \sin({du}/{2})\Big)^{2}\ 
.$$
\end{lemma}
\begin{proof}
Fix the genus $g$ and the curve class $\mathsf{s} + 
h\cdot\mathsf{f}$. 
 The Hodge class is determined 
by dimension constraints to be $(-1)^{g-1}\lambda_{g-1}$.

We degenerate $E$ to two rational curves $P_1$ and $P_2$ 
intersecting each other at $0$ and $\infty$ with the
two $\tau_0(\mathsf{p})$ insertions sent to different components. 
 The distribution of curve degree to $P_1 \times \PP^{1}$ and 
$P_2\times \PP^{1}$ is either of the form 
$$(h, 1) + (h, 0) \ \ \ \text{or} \ \ \ (h,0) + (h,1)\,.$$
We consider the first degree distribution, since the second case
 will yield the same answer.

In the degeneration formula, components of the domain 
map to $P_1\times\PP^1$ or $P_2\times\PP^1$.  
While the curve has a connected domain, the preimages $C_1$ 
and $C_2$ of the irreducible components of 
the target can each be disconnected. 
Let $\Gamma$ denote the dual graph with vertices given by 
the connected components of $C_1$ and $C_2$ and
edges determined by the intersections.  
Each vertex of $\Gamma$ is forced to have valence at least two since, 
for degree reasons, the corresponding subcurve
must intersect both relative divisors.  
The Hodge class $\lambda_{g-1}$ vanishes on the boundary cycles 
of $\Mbar_g$ for which the dual graph $\Gamma$ has Betti number
at least $2$, see \cite{FP}.  
As a result, the graph $\Gamma$ must be a single cycle with 
vertices alternating between $P_1\times \PP^1$ and $P_2\times \PP^1$.

There must exist a divisor $d$ of $h$ with the
following property:
each
connected component of each $C_i$ intersects the relative 
divisors at $0$ and $\infty$ at single points of multiplicity $d$.
Each $C_i$ then consists of $\frac{h}{d}$ connected components.
Let $C_0$ denote the connected component of $C_1$ which has degree 
$(d,1)$. Every other component has degree $(d,0)$.  
Dimension constraints force $C_0$ to contain one of the point insertions.  
There are $h/d$ choices for which connected component of $C_2$ contains the 
other point insertion.

We first consider the connected components $C'$ {\em other}
than $C_0$.  
If such a component has genus $g'$, the Hodge class $\lambda_{g-1}$ will restrict to $\lambda_{g'}$ on $C'$. 
 The contributions are given by the following evaluations:
\begin{eqnarray*}
\Big\langle (d, \omega) \ \Big|\ \lambda_{g'}\ \Big|\  (d,\mathsf{1})\ 
\Big\rangle^{(\PP^1 \times \PP^{1})/\{0,\infty\}}_{g', (d,0)}  = 
\frac{1}{d}\ \delta_{g',0}\ , \\
 \Big\langle (d, \mathsf{1}) \ \Big|\ \lambda_{g'}\ \tau_{0}(\mathsf{p})\ \Big|\
 (d,\mathsf{1})\ \Big\rangle^{(\PP^1 \times \PP^{1})/\{0,\infty\}}_{g', (d,0)}
  = \delta_{g',0}\,.
 \end{eqnarray*}
Here, the relative divisors have coordinates $0$ and $\infty$
on the first $\PP^1$ factor.
The  classes $\mathsf{1},\omega\in H^2(\proj^1,\mathbb{Z})$  
correspond to the identity and the point. 
The higher genus vanishings are obtained,
as before, from $\lambda_{g'}^2=0$.
In order to get a nontrivial invariant, the components $C'$
must have genus $0$. Hence,  $C_0$ must have genus $g-1.$

We are left with the evaluation of the connected component $C_0$ mapping to 
$\PP^1\times\PP^1$ with degree $(d,1)$.
The result then follows from Lemma \ref{mm3} below.
\end{proof}

\begin{lemma}\label{mm3} We have
$$\sum_{g=0}^\infty \Big\langle\  (d,\omega) \ \Big|\ (-1)^{g}\lambda_{g}
\tau_{0}(\mathsf{p})\ \Big|\  (d,\omega)\ \Big 
\rangle^{(\PP^1 \times \PP^{1})/\{0,\infty\}}_{g, (d,1)} u^{2g-2} = \frac{1}{u^2}
\calS(du)^{2},
$$
where
$$\calS(u) = \frac{\sin(u/2)}{u/2}\,.$$
\end{lemma}
\begin{proof}
We replace $(-1)^{g}\lambda_{g}$ with $\hodge_{g}^{\vee}(s)$ with
 a formal variable $s$.  
We can remove the $\tau_0(\mathsf{p})$ insertion by the identity
$$\Big\langle (d,\omega)\ \Big|\  \hodge^{\vee}(s)\ 
\tau_0(\mathsf{p})\ \Big|\ (d,\omega)\ \Big\rangle_{(d,1)} = 
\Big\langle (d,\omega)\ \Big|\ \hodge^{\vee}(s)\ \Big|\ (d,\omega)\ 
\Big \rangle^{\sim}_{(d,1)}.$$
The tilde on the right side denotes 
the rubber moduli space of maps to $\PP^1\times\PP^1$ relative to $0$ and 
$\infty$ up to $\CC^*$-scaling of the first $\PP^1$
factor of the target.  
The identity follows by adding an insertion with the divisor equation 
and then using the marking to rigidify the $\CC^*$-scaling.

Let us orient $\PP^1 \times \PP^1$ so the first factor is a horizontal coordinate
and the second factor is vertical. Then the relative divisors
are on the left and right and the $\CC^*$-scaling acts horizontally.  
While there is no nontrivial horizontal torus action, there is a nontrivial 
$\CC^*$-action in the vertical direction.  
We fix the vertical $\com^*$-action to have weight $t$ for the normal direction 
along the upper edge and weight $-t$ for the normal direction along the lower edge.  
We choose equivariant lifts of the relative point insertions on each side 
corresponding to the upper fixed point on each vertical edge.

\begin{figure*}[htp]
\centering
\psfrag{m}{$D_1$}
\psfrag{a}{$D_2$}
\includegraphics[scale=0.50]{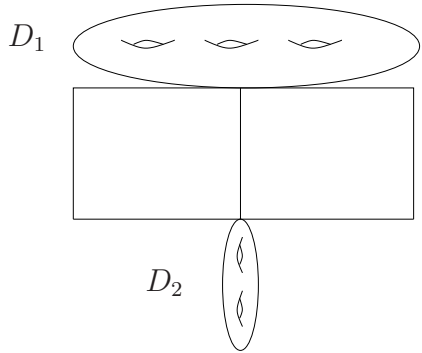}
\caption{Rubber localization}
\end{figure*}


If we localize, the $\com^*$-fixed stable maps have the following 
structure.{\footnote{The analysis follows the 
localization calculation of \cite{dmt}.}}
 First, there is an arbitrary curve $D_1$ mapping with degree $d$ to the upper 
horizontal $\PP^1$, relative to both $0$ and $\infty$, 
with a single relative point of multiplicity $d$.  
There is an arbitrary curve $D_2$ mapping with degree $0$ to the lower horizontal 
$\PP^1$ with no relative points.  
There is a single nonsingular rational curve mapping to 
$\PP^1\times\PP^1$ with degree $(0,1)$, connecting $D_1$ and $D_2$.  
We may use the latter rational curve to rigidify the $\CC^*$-scaling. 
 The fixed-point contribution is
\begin{multline*}
\sum_{g_1+g_2 = g}
t^{2}\cdot \Big\langle (d,\mathsf{1})\  \Big|\  
\hodge_{g_1}^{\vee}(s)\ \hodge_{g_1}^{\vee}(t)\
\frac{\text{ev}_1^*(\mathsf{p})}{t-\psi_{1}}\ 
\Big|\ (d,\mathsf{1})\Big\rangle^{\PP^{1}/\{0,\infty\}}_{g_1,d}\\
\cdot\frac{1}{t(-t)}\cdot
 \Big\langle \emptyset\  \Big| \ \hodge_{g_2}^{\vee}(s)\ 
\hodge_{g_2}^{\vee}(-t)\ \frac{\text{ev}^*_1(\mathsf{p})}{-t-\psi_{1}}\ \Big| \ 
\emptyset \Big\rangle^{\PP^{1}/\{0,\infty\}}_{g_2,0}.
\end{multline*}

Here, the first factor of $t^2$ comes from the relative insertions
and the intermediate term $\frac{1}{t(-t)}$
comes from the $\PP^1$ connecting $D_1$ and $D_2$.  
The point insertion in the first relative invariant comes from 
rigidifying the point where $D_1$ meets the connecting $\PP^1$.

To simplify this expression, we choose the substitution $s=1$ and $t=-1$
so we can apply the Mumford relation
$$\hodge_{g_1}^{\vee}(1)\hodge_{g_1}^{\vee}(-1) = (-1)^{g_{1}} \,.$$
The result is
\begin{align*}
\sum_{g_1+g_2=g}
(-1)^{g} \Big\langle (d,\mathsf{1})\ \Big| \ \tau_{2g_{1}}(p)\  \Big| \ 
(d,\mathsf{1}) \Big\rangle^{\PP^{1}/\{0,\infty\}}_{g_1,d} \cdot
\int_{\Mbar_{g_{2},1}} \frac{\hodge^{\vee}(1)\hodge^{\vee}(1)\hodge^{\vee}(0)}
{1 - \psi_{1}}\,.
\end{align*}
The 1-pointed relative invariants of $\PP^{1}$ are given in \cite{okpanp1}
$$\sum_{g} (-1)^{g}u^{2g-2}
\Big\langle (d,\mathsf{1})\ \Big|\  \tau_{2g}(\mathsf{p})\ \Big| 
\ (d,\mathsf{1})\ \Big\rangle^{\PP^{1}/\{0,\infty\}}_{g,d} = \frac{1}{u^2}
\frac{\calS(du)^{2}}{\calS(u)}\,,$$
and the triple Hodge integrals for $\Mbar_{g,1}$ are given in \cite{FP} as
$$\sum_{g}(-1)^{g}u^{2g-2} 
\int_{\Mbar_{g,1}} \frac{\hodge^{\vee}(1)\hodge^{\vee}(1)\hodge^{\vee}(0)}{1 - \psi_{1}}
= \frac{1}{u^2}\calS(u)\,.$$
Putting everything together completes the proof.
\end{proof}

\section{Quasimodular forms}
\label{qmod}

\subsection{Overview}
We follow the notation of Section \ref{qmforms}.
Let $S$ be an elliptically fibered $K3$ surface, and
let $$\gamma_{{1}},\dots,\gamma_{{r}}\in H^*(S,\mathbb{Z}) \ $$
be cohomology classes.
Our main result here, Theorem \ref{qqq}, states 
the
descendent series
$$\mathsf{F}^S_{g}(\tau_{k_{1}}(\gamma_{{1}})\cdots\tau_{k_{r}}
(\gamma_{{r}}))$$
is a quasimodular form with simple pole at $q=0$. 
The ring 
$$\QMod = \QQ[E_{2}(q), E_4(q), E_6(q)]$$
of holomorphic quasimodular forms (of level $1$) is the 
$\QQ$-algebra generated by Eisenstein series $E_{2k}$, see \cite{bghz}.  
The ring $\QMod$
is naturally graded by weight (where $E_{2k}$ has weight $2k$) 
and inherits an increasing filtration 
$$\QMod_{\leq 2k}\subset \QMod$$
given by forms of weight $\leq 2k$.
More precisely, we will prove the
descendent series are of the form
$$\mathsf{F}^S_{g}(\tau_{k_{1}}(\gamma_{{1}})\cdots\tau_{k_{r}}
(\gamma_{{r}}))
\in \frac{1}{\Delta(q)}\QMod_{\leq 2g+2r}\,.$$

Our results follow from 
an explicit (though difficult) algorithm for calculating 
descendent series which reduces the claims to the case of 
$g=0$ and the Gromov-Witten theory of elliptic curves.

\subsection{Elliptic curves}
We begin by explaining the quasimodularity statement for elliptic curves $E$. 
Given cohomology classes
$$\gamma_{1}, \dots, \gamma_{r} \in H^*(E,\mathbb{Z}),$$ 
consider the following descendent
series for connected Gromov-Witten invariants of $E$,
$$\mathsf{F}^E_{g}(\tau_{k_{1}}(\gamma_{1})\cdots\tau_{k_{r}}(\gamma_{r}))
= \sum_{d \geq 0} \Big\langle 
 \tau_{k_{1}}(\gamma_{1})\cdots\tau_{k_{r}}(\gamma_{r})
\Big\rangle^{E}_{g,d} \ q^{d}.$$

\begin{prop}\label{ellipticmodular} For $r>0$,
$\mathsf{F}^E_{g}(\tau_{k_{1}}(\gamma_{1})\cdots\tau_{k_{r}}(\gamma_{r}))$
is the Fourier expansion in $q$ of 
a holomorphic quasimodular form of weight 
$$2g-2+ 2\sum_{i=1}^{r} \deg(\gamma_{i}).$$
\end{prop}

We use $\deg(\gamma_i)$ here to denote half the cohomological degree. For 
instance, point classes in $H^2(E,\mathbb{Z})$ are degree $1$.
The $r=0$ case concerns the well-known counting of 
genus 1 covers of $E$. The resulting series is not in $\QMod$.

\begin{proof}
When all the $\gamma_i$ are
point classes in $H^2(E,\mathbb{Z})$, 
the result is stated as a corollary of Theorem 5 in \cite{gwhurwitz}.  
For the general case, we need only check that quasimodularity 
is preserved by the extended Virasoro relations for curves proved 
in \cite{okpanvir}  which provide rules for removing insertions 
$\tau_{k}(\gamma)$ for $\gamma \in H^{\leq 1}(E,\mathbb{Z})$.  
\end{proof}

\subsection{Tautological classes}
We will rephrase Theorem \ref{qqq} in terms of tautological classes.
For $2g-2+r >0$, let
$$R(\overline{M}_{g,r}) \subset H^*(\Mbar_{g,r},\QQ)$$
denote 
the subring of \emph{tautological} cohomology classes.  
The tautological ring $R(\overline{M}_{g,r})$ is
 spanned by the push-forwards of products of $\psi$-classes on 
boundary strata,  see \cite{icm} for an introduction.

Given $\alpha \in R(\overline{M}_{g,r})$ and 
$\gamma_{{1}},\dots,\gamma_{{r}}\in 
H^*(S,\mathbb{Z})$,
we can use the forgetful morphism
$$\pi: \Mbar_{g,r}(S,\beta) \rightarrow \Mbar_{g,r}$$
to define the reduced invariant
$$\Big \langle \pi^*(\alpha), 
\gamma_{{1}},\cdots,\gamma_{{r}}\Big\rangle^{S}_{g,\beta}
= \int_{[\Mbar_{g,r}(S,\beta)]^{\mathrm{red}}} \pi^*(\alpha)\cup \prod_{i=1}^{r} \ev_{i}^{*}(\gamma_{{i}})$$
when $2g-2+r >0$. Let
\begin{equation}\label{lkd9}
\mathsf{F}^S_{g}(\alpha; \gamma_{{1}},\dots,\gamma_{{r}}) = 
\sum_{h=0}^{\infty} \Big\langle 
\pi^*(\alpha), \gamma_{l_{1}},\cdots,\gamma_{l_{r}}\Big\rangle^{S}_{g,h}
q^{h-1}
\end{equation}
be the associated generating function.
The index $h$ of the sum 
 stands for a primitive effective
curve class $\beta\in H^2(S,\mathbb{Z})$ satisfying
$$\langle \beta , \beta \rangle = 2h-2\,.$$

\begin{prop} \label{pqqp}  We have
$$\mathsf{F}^S_{g}(\alpha; \gamma_{{1}},\dots,\gamma_{{r}}) 
\in \frac{1}{\Delta(q)}\QMod_{\leq 2g+2r}\,.$$
\end{prop}

The splitting formula for standard Gromov-Witten theory
takes a slightly modified form for the reduced Gromov-Witten
theory of $K3$ surfaces.
Given a reducible boundary divisor 
$$\iota:\Mbar_{g_{1},r_{1}+1} \times \Mbar_{g_{2},r_{2}+1} \rightarrow \Mbar_{g,r},$$
let $\Delta$ denote the pushforward of the 
fundamental class of the left side.  The splitting
formula for reduced classes is
\begin{multline*}
\pi^*(\Delta)\cap[\Mbar_{g,r}(S,\beta)]^{\mathrm{red}} = \\  
\iota_*([\Mbar_{g_{1},r_{1}+1}(S,\beta)]^{\mathrm{red}}
\times_{S}[\Mbar_{g_{2},r_{2}+1}(S,0)]^{\mathrm{vir}} \\ 
+ [\Mbar_{g_{1},r_{1}+1}(S,0)]^{\mathrm{vir}}\times_{S}[\Mbar_{g_{2},r_{2}+1}(S,\beta)]^{\mathrm{red}})\ .
\end{multline*}
Similarly, for the irreducible boundary divisor given by
$$\iota:\Mbar_{g-1, r+2} \rightarrow\Mbar_{g,r}\,,$$
we have
$$\pi^*(\Delta)\cap[\Mbar_{g,r}(S,\beta)]^{\mathrm{red}} =
[\Mbar_{g-1,r+2}(S,\beta)]^{\mathrm{red}}.$$
See \cite{jleec1,jleec2} for proofs (again the methods
of Section \ref{gwd} could also be used).

Using the usual trading of cotangent line classes on
$\overline{M}_{g,r}(S,\beta)$  and
the above splitting formulas for the reduced class,
we can easily express the
descendent series
$$\mathsf{F}^S_{g}(\tau_{k_{1}}(\gamma_{{1}})\cdots\tau_{{r}}
(\gamma_{l_{r}}))$$
in terms of linear combinations of the series
$$\mathsf{F}^S_{g'}(\alpha'; \gamma'_{{1}},\dots,\gamma'_{{r'}}) 
$$
where $g'\leq g$ or
$g'=g$ and $r'\leq r$.
Hence, Proposition \ref{pqqp} implies Theorem \ref{qqq}.

\subsection{Proof of Proposition \ref{pqqp}}

We proceed by induction on the pair $(g,r)$ where 
$g$ is the genus of the domain curve and $r$ is the number of insertions.  
We order such pairs $(g,r)$ so $(g',r') < (g,r)$ if either
\begin{enumerate}
\item[$\bullet$]
$g'<g$ or
\item[$\bullet$]
$g'=g$ and $r'<r$.
\end{enumerate} 
We assume the Proposition is known for all series
 with $(g',r') < (g,r)$. We will
give a procedure for reducing all invariants of type $(g,r)$ 
to invariants of lower type in a manner preserving quasimodularity.

\vspace{9pt}
\noindent{\bf{Base case:}} $(g,r)=(0,r\leq 2)$.
\vspace{9pt}

If $(g,r)=(0,0)$, then 
the Proposition follows from the Yau-Zaslow formula. 
 If $g=0$ and $r\leq 2$, we are in the unstable range
$2g-2+r\leq 0$, so no $\alpha$ insertion appears. Then,  
for dimension reasons, there must be an 
insertion of the form $$\tau_0(1), \tau_1(1)\ \text{or}\ \tau_0(\gamma),$$
where $\gamma \in H^2(S,\mathbb{Z})$.  
Since the claim of the Proposition
is preserved by applying either the string, dilaton or divisor equation, 
we can remove one insertion and  strictly reduce $r$.

\vspace{9pt}
\noindent \textbf{Case (i):}  $\deg \gamma_{l_{i}} \leq 1$ for all $i$.
\vspace{9pt}

We again use $\deg(\gamma)$ to denote half the 
cohomological degree of $\gamma$. 
In case (i), there are no point insertions.
Since the reduced virtual dimension is $g+r$, the
dimension constraint then implies
$$\deg \alpha \geq g.$$

By a strong form of Getzler-Ionel  vanishing proved in 
Proposition 2 of \cite{fpm},
there exists a tautological class 
$\alpha' \in R(\partial\Mbar_{g,r})$ satisfying
$$\iota_*\alpha' = \alpha,$$ 
where $\iota$ denote the inclusion of the boundary,
$$\iota: \partial\Mbar_{g,r} \hookrightarrow \Mbar_{g,r}\,.$$ 
Using $\alpha'$ and the
splitting formula for the virtual class, we can express 
the series
$$\mathsf{F}_{g}(\alpha; \gamma_{{1}},\dots,\gamma_{{r}})\ $$
as a linear combination of series of type $(g',r') < (g,r)$.

\vspace{9pt}
\noindent \textbf{Case (ii):}  $\gamma_{{1}}=\mathsf{p} \in H^4(S,
\mathbb{Z})$.
\vspace{9pt}

We will degenerate $S$ to the normal cone of an elliptic fiber $E$,
\begin{equation}\label{nttt3}
S \ \rightsquigarrow \ S \times _E (E\times\PP^1)\ ,
\end{equation}
and use the degeneration formula for the
reduced virtual class proven by  Lee and Leung \cite{jleec1,jleec2}.
Following the notation of Section \ref{pint},
the required generating series are:
\begin{align*}
\mathsf{F}^{S/E}_{g}(\alpha; \gamma_{{1}},\dots,\gamma_{{r}}) &= 
\sum_{h=0}^{\infty} 
\Big\langle \pi^*\alpha \cup \prod \tau_0(\gamma_{{i}}) \ \Big|\ 
\mathsf{1}\Big\rangle^{S/E}_{g, h}q^{h-1}\ ,\\
\mathsf{F}^{(E\times\PP^1)/E}_g(\alpha; \gamma_{{1}},\dots,\gamma_{{r}}) &= 
\sum_{h=0}^{\infty} 
\Big\langle \pi^*\alpha \cup \prod \tau_0(\gamma_{{i}}) \ \Big|\ 
\omega \Big\rangle^{(E\times\PP^{1})/E}_{g, \mathsf{s}+ h\cdot \mathsf{f}} 
q^{h}\,.
\end{align*}
The Gromov-Witten invariants for $S/E$ are reduced.
The first step is to prove a quasimodularity result for the latter 
series.
\begin{lemma}\label{EP1}
$\mathsf{F}^{(E\times\PP^1)/E}_g(\alpha; \gamma_{{1}},\dots,\gamma_{{r}}) 
\in \QMod_{\leq 2g+2r}.$
\end{lemma}
\begin{proof}

The relative invariants of
$(E\times\PP^1)/E$
 are algorithmically determined 
from the Gromov-Witten theory of $E$ using localization (via
the $\com^*$-action on $\PP^1$) and
the absolute/relative 
correspondence of \cite{mptop}.  The result is a combinatorial 
expression for the series
$$\mathsf{F}^{(E\times\PP^1)/E}_g(\alpha; \gamma_{{1}},\dots,\gamma_{{r}})$$
in terms of descendent series for $E$.  

Each insertion in a 
descendent series for $E$
contributes to the weight of the final answer
by Proposition \ref{ellipticmodular}.
The bound of $2g+2r$ is obtained by an elementary
analysis of the possible $\com^*$-fixed point loci of
the moduli space
$\overline{M}_{g,r}(E \times \PP^1, \mathsf{s} + h \mathsf{f})$.
\end{proof}

\begin{lemma}\label{Srel}
For $(g',r') < (g,r)$ and insertions 
$$\alpha' \in R(\Mbar_{g',r'}), \  \ \
\gamma'_{{1}},\dots,\gamma'_{{r'}} \in H^*(S,\mathbb{Z})\,, $$
we have  
$$\mathsf{F}^{S/E}_{g'}(\alpha'; \gamma'_{{1}},\dots,\gamma'_{{r'}}) 
\in \frac{1}{\Delta(q)}\QMod_{\leq 2g'+2r'}.$$
\end{lemma}
\begin{proof}
The result follow from the degeneration formula and the inductive hypothesis.  
The degeneration formula of Lee and Leung applied to
\eqref{nttt3} yields
\begin{align*}
\mathsf{F}^S_{g'}(\alpha'; \gamma'_{{1}},\dots,\gamma'_{{r'}})\, =\ &
\mathsf{F}^{S/E}_{g'}(\alpha'; \gamma'_{{1}},\dots,\gamma_{{r'}})\\
&+ \sum_{(g'',r'')< (g',r')} \mathsf{F}^{S/E}_{g''}(\cdots) \cdot 
\mathsf{F}^{(E\times \PP^1)/E}_{g'-g''}(\cdots).
\end{align*}
The second sum is also over all distribution of insertions (both $\alpha'$ and
the $\gamma'_i$).
We conclude the relative series $\mathsf{F}^{S/E}$ can be
expressed in terms of the absolute series $\mathsf{F}^{S}$
by a change of basis which is
upper-triangular with respect to the ordering on pairs bounded by 
$(g',r')$.
Moreover, 
the  coefficients 
of the change of basis are 
given by the series $\mathsf{F}^{(E\times \PP^1)/E}$.
The Lemma then follows from Lemma \ref{EP1} and the inductive 
hypothesis on $\mathsf{F}^{S}$.
\end{proof}

We now complete
the analysis of Case (ii) by
explaining a genus reduction procedure
dependent upon the insertion $\tau_0(\mathsf{p})$. 
 We use the formula of Lee and Leung applied
to the degeneration \eqref{nttt3}.
The  insertion $\tau_0(\mathsf{p})$
is specialized to lie on the bubble $E\times \PP^1$, and the 
remaining insertions are specialized 
arbitrarily. We obtain
\begin{equation}\label{hmqq}
\mathsf{F}^S_{g}(\alpha; \mathsf{p}, \gamma_{{2}},\dots,\gamma_{{r}}) =
\sum_{(g',r') \leq (g,r)} \mathsf{F}^{S/E}_{g'}(\alpha',\cdots)\cdot 
\mathsf{F}_{g''}^{(E\times\PP^1)/E}( \alpha''; \mathsf{p}, \cdots).
\end{equation}
Again, we have suppressed the summation over  the splitting of
the
tautological class $\alpha$ and
distribution of the insertions.  

Since the relative invariants on $E\times\PP^1$ occurring
in \eqref{hmqq} have both 
relative and absolute point insertions, 
the domain genus must be positive. So 
 $g''>0$ and $g'<g$.  
By Lemmas \ref{EP1} and \ref{Srel}, every nonzero term 
on the right side of \eqref{hmqq}
is quasimodular.  
Therefore, the left side is quasimodular of the desired
form.
\qed

\subsection{Non-primitive classes} \label{ccllyy}

Let $S$ be an elliptically fibered $K3$ surface with 
section. Let
$$\mathsf{s}, \mathsf{f} \in H_2(S, \mathbb{Z})$$
denote the section and fiber classes as before.
A natural descendent potential function for the
reduced theory of $K3$ surfaces is defined
by
$$\bF^{S}_{g,m}\big(\tau_{k_1}(\gamma_{l_1}) \cdots
\tau_{k_r}(\gamma_{l_r})\big)=
\sum_{n=0}^\infty 
\Big\langle \tau_{k_1}(\gamma_{l_1}) \cdots
\tau_{k_r}(\gamma_{l_r})\Big\rangle^{\mathrm{red}}_{g,m\mathbf{s}+ n \mathbf{f} } 
\ q^{m(n-m)}
$$
for $g\geq 0$ and $m\geq 1$.  
The following conjecture{\footnote{The conjecture
was made earlier by two of us in \cite{clay}.} specializes to Theorem 4
in the primitive $(m=1)$ case.

\vspace{10pt}
\noindent{\bf Conjecture.}
{\em 
$\bF^{S}_{g,m}\big(\tau_{k_1}(\gamma_{l_1})
\cdots
\tau_{k_r}(\gamma_{l_r})\big)
$
is the Fourier expansion in $q$
of a quasimodular form of level $m^2$ with pole at $q=0$ of
order at most $m^2$.}
\vspace{10pt}

By the ring of quasimodular forms 
of level $m^2$
with possible poles at $q=0$, we mean the algebra generated by 
the Eisenstein series
$E_2$ over the ring of modular forms of level $m^2$.

\appendix
\section{Reduced theories revisited}
\label{oldnew}

\subsection{Reduced obstruction theories} \label{eg2}
Given a perfect obstruction theory for a scheme $P$, together with a surjection 
\begin{equation}\label{surj}
\text{Ob}\to\O_P\ ,
\end{equation}
we will attempt now to produce a reduced obstruction theory.

For our discussion, we assume the obstruction theory of $P$
arises from the following standard model (which always holds locally).
Let
$$P\subset A$$
be an embedding of $P$ in a nonsingular
ambient space $A$, cut out as the zero locus of a
section $s$ of a vector bundle $E\to A$. 
In other words, $$s\colon E^\vee\to\O_A$$
 generates the
ideal $I\subset\O_A$ defining $P\subset A$. We obtain
$$
\spreaddiagramrows{-.6pc}
\xymatrix{
E^\vee|_P \rto\dto^s & \Omega_A|_P \ar@{=}[d]\rto & \Omega_P\! \ar@{=}[d] \\
I/I^2 \rto^d & \Omega_A|_P\rto & \Omega_P,\!\!}
$$
where the lower row is the exact sequence of K\"ahler differentials for $P\subset A$.
The first horizontal arrow is the composition $ds$ of the other arrows in the first square.
We denote the resulting 2-term complex of locally free sheaves on $P$ by
$$E\udot=\{E^{-1}\Rt{ds}E^0\}\,.$$

Since $\{I/I^2\rt{d}\Omega_A|_P\}$ is 
quasi-isomorphic to the truncated cotangent
complex $\LL_P$ of $P$, we obtain a morphism of complexes
$$
\spreaddiagramrows{-.6pc}
\xymatrix{
E\udot\! \dto &\ar@{}^{=}[d]&& E^{-1} \rto^{ds}\dto^s & E^0\!\! \ar@{=}[d] \\
\LL_P\!\! &&& I/I^2 \rto^d & \Omega_A|_P,\!\!\!}
$$
which is surjective on $h^{-1}$ and an isomorphism on $h^0$. We have
constructed a perfect 
obstruction{\footnote{Of course, 
a different choice of $(E,s)$ generating the same $I\subset\O_A$ gives 
a \emph{different} obstruction theory for $P$, with a potentially different obstruction sheaf.}}
theory for $P$ with obstruction sheaf $\text{Ob}$ defined to be the cokernel of the \emph{dual}
map $$E_0\Rt{ds}E_1\ , $$ 
where $E_i$ denotes the dual of $E^{-i}$.

The surjection \eqref{surj}
yields a surjection $E_1\to\O_\P$ whose locally free kernel
we denote by $F_1$,
\begin{equation} \label{FEO}
0\to F_1\to E_1\to\O_\P\to0.
\end{equation}
The map $E_0\to E_1$ factors through $F_1$. After dualizing, we obtain the complex
\begin{equation} \label{Fdot}
F\udot=\{F^{-1}\to E^0\}.
\end{equation}
We would like to know when $F\udot$ gives a (smaller) 
obstruction theory for $P$.

A simple example is given by considering the scheme 
$$
P=P_n=\Spec\C[x]/(x^n), \quad n\ge2,
$$
cut out of the 
ambient space $A=\C$ by the section $s=(x^n,0)$ of the 
trivial rank 2 bundle $E$. Since
$$
\spreaddiagramcolumns{2pc}
E_0\Rt{ds}E_1 \quad\mathrm{is}\quad \xymatrix{\O_P \rto^{(nx^{n-1},0)\ } & \O_P^{\oplus2}}\ ,
$$
we find
\begin{equation} \label{obP}
\text{Ob}=\O_{P_{n-1}}\oplus\O_P
\end{equation}
is the direct sum of the structure sheaf of
$P_{n-1}\subset P$ and the structure sheaf of $P$. 
The second summand is locally free
so we have a surjection
\begin{equation} \label{mapP}
\xymatrix{\text{Ob}\rto^{(0,1)\ }& \O_\P\,.}
\end{equation}
Here, replacing $E_1$ by $F_1$ amounts to removing the 
second summand of the original
trivial rank 2 bundle $E$ and working in the first summand. 
Since the section
$s$ lies in the first summand, the result is another 
obstruction theory for $P$.


\label{simp}
\subsection{Questions}
The Kiem-Li construction of the reduced class, the powerful
$T^1$-lifting result of \cite[Proposition 6.13]{BuF}\footnote{Given a thickening $S_0\subset S$ with square-zero ideal $I$, Buchweitz-Flenner
show the obstructions to extending a map 
$$f_0\colon S_0\to P\ \ \ \text{to}\ \ \ f\colon S\to P$$
lie in the kernel of any surjection
$\text{Ob}\to\O_P$ so long as the Kodaira-Spencer class of the thickening lies in
$\Ext^1(\Omega_{S_0},I)$.
By \cite[Theorem 4.5]{BF}, Question 1 is equivalent to asking whether their result extends
to thickenings in the larger group $\Ext^1(\LL_{S_0},I)$.}
and the simple example discussed in
Section \ref{simp} suggest the following very natural question.

\vspace{7pt}
\begin{itemize}
\item[(Q1)] Given a perfect obstruction theory whose 
obstruction sheaf admits a locally
free quotient, can the quotient be removed 
to leave another perfect obstruction
theory ?
\end{itemize}
\vspace{7pt}

Since the problem is local we restrict ourselves to the standard model of
Section \ref{simp} for the obstruction theory of $P$,
restricted to the case where the 
locally free quotient is
trivial of rank one \eqref{surj}. 
We then ask if the obstruction theory $E\udot\to\LL_P$
factors through the complex $F\udot$ of \eqref{Fdot}. 

Dualizing the surjection $E_1 \rarr \O_P$ obtained
from \eqref{surj}, yields 
a subsheaf
$$0 \rarr \O_P \rarr E^{-1}\,.$$
Then,
question (Q1)  specializes as follows.
 
\vspace{7pt}
\begin{itemize}
\item[(Q$1'$)] Does the composition $\O_P \rarr E^{-1} \rarr I/I^2$
always vanish?
\end{itemize}
\vspace{7pt}

In general, the answer to (Q$1'$) is no. 
In the example of Section \ref{simp},
if we replace the surjection \eqref{mapP} by $(x,1)$, 
then the resulting
composition 
$$\O_P \rarr E^{-1} \rarr I/I^2$$
sends $1$ to
$x^{n+1}$, nonzero in $I/I^2$  even for $n=2$.

So we consider a weaker question.
Given a splitting of the sequence \eqref{FEO}, we write
$$E^{-1} = F^{-1} \oplus \O_P \ $$
after dualizing.
We consider the diagram
\begin{equation} 
\label{square}
\xymatrix{F^{-1}
\oplus\O_P \rto\dto_{s=(s_1,s_2)\!}^{\!} & E^0\!\! \ar@{=}[d] \\
I/I^2 \rto^d & \Omega_A|_P\ \!\!} 
\end{equation}
and ask instead whether $s_1$ is surjective.

\vspace{7pt}
\begin{itemize}
\item[(Q2)] Is the first component 
$$s_1\colon F^{-1}\to
I/I^2$$ of the arrow $s$ in \eqref{square} surjective ?
\end{itemize}
\vspace{7pt}

An affirmative answer to (Q1) implies an
affirmative answer to (Q2), so (Q2) is weaker.
However,
an affirmative answer to (Q2) leads to a reduced
obstruction theory (at least locally). 
We can also further weaken the questions.

\vspace{7pt}
\begin{itemize}
\item[(Q3)] Is the answer to Questions $1'$ 
or 2 positive if the obstruction theory $E\udot$
is symmetric \cite{BF2, bdt} ?
\end{itemize}
\vspace{7pt}

\noindent While (Q2) does hold for the example of Section \ref{simp},
the following example shows the answer to both (Q2) and (Q3)
is negative in general. 

\subsection{Counterexample I} \label{kkxx33}
Let $f$ be a polynomial in $x_1$ and $x_2$ which is {\em not}
homogeneous with respect to any grading of $x_1$ and $x_2$.
For example,
$$f= x_1^2 + x_1x_2 + x_2^3\,.$$
We view $f$ as polynomial in the ring
$\C[x_1,x_2,t]$
which does not depend on $t$. 
Define $P\subset A= \text{Spec}(\C[x_1,x_2,t])$
by the vanishing of the 1-form
$$
\sigma=
df+fdt=\frac{\partial f}{\partial x_1}dx_1+\frac{\partial f}{\partial x_2}dx_2
+fdt\,.
$$
In other words, $P$ is defined by the ideal
\begin{equation} \label{PidealP}
I=\left(\frac{\partial f}{\partial x_1},\frac{\partial f}{\partial x_2},f\right).
\end{equation}
Since $f$ is has no $t$ dependence, 
$d\sigma=df\wedge dt$ vanishes on $P$ (as
$df$ lies in the ideal $I$). 
Hence, $\sigma$ is \emph{almost closed} in the terminology of \cite{bdt},
and the resulting obstruction theory 
$$E\udot=\{TA|_P\Rt{D\sigma}\Omega_A|_P\}$$ 
is symmetric.

The scheme
$P$ is cut out by the section{\footnote{We use the notation $f_x$ for $\partial f/\partial x$.}}
$s=(f_{x_1},f_{x_2},f)$ of the 
trivial rank 3 vector bundle on $A$. 
The induced map $E_0\to E_1$ is
$$
\left(\!\! \begin{array}{ccc}
f_{x_1x_1} & f_{x_2x_1} & f_{tx_1} \\
f_{x_1x_2} & f_{x_2x_2} & f_{tx_2} \\
f_{x_1} & f_{x_2} & f_t
\end{array} \!\!\right)=
\left(\!\! \begin{array}{ccc}
f_{x_1x_1} & f_{x_2x_1} & 0 \\
f_{x_1x_2} & f_{x_2x_2} & 0 \\
0 & 0 & 0
\end{array} \!\right)
\colon\ \O_P^{\oplus3}\Rt{}\O_P^{\oplus3},
$$
since all functions are
 $t$-independent and the partial derivatives of $f$ vanish on $P$
by definition. The symmetry of the matrix reflects the
fact that the  obstruction theory is symmetric. 

The third summand of $E_1=\O_P^{\oplus3}$ is also a summand
of the obstruction sheaf $\text{Ob}$, 
so Questions 2 and 3 ask whether 
$s_1$ is surjective. But
$$
s_1\colon\O_P^{\oplus2}\to I/I^2 \quad\mathrm{is}\quad s_1=(f_{x_1},f_{x_2}),
$$
and the element $[f]\in I/I^2$ is not in the image of this map precisely because $f$
is not quasihomogeneous \cite{Sa}. 

\subsection{Critical locus condition}
An \emph{a priori} stronger condition than having a symmetric obstruction theory is that
$P\subset A$
should be the critical locus $\Crit(\phi)$ of 
a holomorphic function $\phi$ on $A$, with the
induced obstruction theory
$$
E\udot\ = \xymatrix{ \{TA|_P\rto^{D(d\phi)\ } & \Omega_A|_P\}.}
$$
The symmetry of the Hessian of $\phi$ implies that 
$E\udot$ is indeed symmetric. 
\begin{itemize}
\item[(Q4)] Is the answer to Questions $1'$ or 2 positive 
if there exists a holomorphic
function $\phi$ on $A$ for which $$P=\Crit(\phi)\subset A$$ with 
the induced symmetric obstruction
theory ?
\end{itemize}

In fact, the previous example \eqref{PidealP} provides a counterexample also to (Q4) as it is easily seen\footnote{We thank Dominic Joyce for this observation. The published version of this paper erroneously claimed to prove that the scheme defined by the ideal \eqref{PidealP} is not the critical locus of any holomorphic function on any smooth variety. A correct example of a zero scheme of an almost closed 1-form with this property is now constructed in \cite{1form}.} to be the critical locus of the holomorphic function
$$
\Phi=e^tf
$$
on $\C^3$. However, we will prove the answer to (Q4)
is \emph{yes} when the surjection $\text{Ob}\to\O_P$ (coming from a vector field on $P$ as in
\eqref{vfield}) is the restriction of a vector field on $A$ 
\emph{along which the holomorphic function is constant}. 
In our application to stable pairs, $A$ and $P$ are products with
$\C$, the vector field is $\partial_t$ pulled back from $\C$, and the potential $\phi$ is $\C$-invariant.

\begin{prop} \label{zero}
Let $\phi$ be a holomorphic function on $A$, and let
$v$ be a vector field on $A$ satisfying $v(\phi)=0$. 
Suppose $v$ does not vanish on 
$$P=\Crit(\phi)
\subset A\ ,$$ 
and let $\text{\em Ob}\to\O_P$ be the induced surjection \eqref{vfield}. 
Then, the composition
$$\O_P \rarr E^{-1} \rarr I/I^2$$
is zero. Therefore
$\{F^{-1}\to E^0\}$ defines a reduced perfect obstruction theory for $P$.
\end{prop}

\begin{proof}
The map $\O_P\rarr E^{-1}=T_A|_P$ is defined
by the section $v$
of the tangent bundle.
The map $E^{-1} \rarr I/I^2$ takes $v$ to
 $$[d\phi(v)]\in I/I^2 \,.$$
 But, $d\phi(v)=v(\phi)=0$.
\end{proof}

The moduli space $\P/B$, obtained from families of nonsingular $K3$ surfaces, {\em is} the (relative) critical locus of a $\C$-invariant holomorphic
function, so we can use Proposition \ref{zero} to produce a reduced obstruction theory.

\begin{thm} 
\label{jj99}
The perfect obstruction theory $E\udot\to\LL_{\P/B}$ of \eqref{obsthy} factors
through $F\udot$. The reduced class of Section \ref{rc} is really
a virtual cycle for the obstruction theory $F\udot$.
\end{thm}

The only missing step in the proof of Theorem \ref{jj99} is the expression of
the moduli space $\P \rarr B$ locally as $\Crit(\phi)$ for
a $\C$-invariant holomorphic function $\phi$. 
There are two approaches to the question: 
via the methods of Joyce-Song \cite{joys} or via
the announced work of  Behrend-Getzler. Since we do not need 
Theorem \ref{jj99},
we omit the discussion here.

\section{Boundary expressions {\em by A. Pixton}}
\label{Pixton}

Let $\beta$ be a primitive effective class with self-intersection $2h-2$ on a $K3$ surface
$S$. We will compute here the Hodge integrals
\begin{equation*}
R_{g,h} = R_{g,\beta} = \int_{[\md_{g}(S,\beta)]^{\red}}(-1)^g\lambda_g
\end{equation*}
for $g\le 3$
via the boundary geometry of the moduli space of curves.
The results provide an alternate verification of the Katz-Klemm-Vafa conjecture in low genus.

For convenience, we  scale the Eisenstein series to
\begin{equation*}
C_{2g} = -\frac{B_{2g}}{2g\cdot(2g)!}E_{2g}.
\end{equation*}
By the proof of Corollary 2, we can rewrite
the KKV conjecture as 
\begin{equation*}
\sum_{g\ge 0}\sum_{h\ge 0}R_{g,h}t^{g}q^{h-1} = 
\frac{1}{\Delta(q)}\exp\left(\sum_{g\ge 1}(-1)^g2C_{2g}t^g\right).
\end{equation*}

Our approach has two steps. First, we write $\lambda_g$ as a linear combination of boundary strata of $\md_g$. This allows us to replace the Hodge integral with a linear combination of descendent integrals, which can be reduced for $g \leq 3$
to the purely stationary descendent integrals
\begin{equation*}
\int_{[\md_{g,n}(S,\beta)]^{\red}}\tau_{k_1}(\mathsf{p})\cdots\tau_{k_n}(\mathsf{p})\,.
\end{equation*}
Second, these integrals can be evaluated in terms of products of Gromov-Witten invariants of elliptic curves and the Yau-Zaslow formula. 

Define quasimodular forms $T_k$ in terms of the Gromov-Witten invariants of an elliptic curve $E$ by
\begin{equation*}
T_k = \sum_{\substack{i,j \ge 0 \\ 2i+j\le k}}(-1)^{i+j}\frac{C_2^i}{i!}\sum_{d\ge 0}q^d\int_{[\md_{g,2}(E,d[E])]^{\vir}}\lambda_j\tau_{k}(\mathsf{p})\tau_{k-2i-j}(\mathsf{p})\,.
\end{equation*}
By the degeneration formula of \cite{jleec1}, standard
localization arguments and the Yau-Zaslow formula, we find
\begin{multline}\label{stationary formula}
\sum_{h\ge 0}\left(\int_{[\md_{\sum_{i=1}^n k_i+n,n}(S,\beta)]^{\red}}
\tau_{k_1}(\mathsf{p})\cdots\tau_{k_n}(\mathsf{p})\right)q^{h-1} 
= \\ \frac{1}{\Delta(q)}T_{k_1}\cdots T_{k_n}\,.
\end{multline} 
Formula
(\ref{stationary formula}) in the case $k_1=\cdots=k_n=0$ 
was proven by Bryan and Leung \cite{brl}.


Faber and Pandharipande \cite{FP} describe how to replace Hodge insertions $\lambda_j$ with descendent insertions, and the work of Okounkov and Pandharipande in \cite{gwhurwitz}, \cite{okpanp1}, \cite{okpanvir} completely determines the descendent Gromov-Witten theory of target curves. Thus,
the quasimodular forms $T_k$ can be explicitly computed. 
We list the first two, expressed in terms of our scaled Eisenstein series $C_{2m}$:
\begin{align*}
T_0 &= q\frac{d}{dq} C_2 = -2C_2^2+10C_4\,, \\
T_1 &= q\frac{d}{dq}\left(\frac{2}{3}C_2^2-\frac{1}{3}C_4\right) = 
-\frac{8}{3}C_2^3+16C_2C_4-7C_6\,.
\end{align*}

We will also need the following formulas for the action of the differential operator $q\frac{d}{dq}$ on quasimodular forms:
\begin{equation*}
\begin{split}
q\frac{d}{dq}C_2 &= -2C_2^2+10C_4\,, \\
q\frac{d}{dq}C_4 &= -8C_2C_4+21C_6\,, \\
q\frac{d}{dq}C_6 &= -12C_2C_6+\frac{160}{7}C_4^2\,, \\
q\frac{d}{dq}\left(\frac{1}{\Delta(q)}\right) &= \frac{24C_2}{\Delta(q)}\,.
\end{split}
\end{equation*}


Since the $g = 0$ and $n=0$ case of (\ref{stationary formula}) is the
Yau-Zaslow formula, we have
\begin{equation*}
R_{0,h} = [q^{h-1}]\frac{1}{\Delta(q)} = [q^{h-1}][t^0]\frac{1}{\Delta(q)}\exp\left(\sum_{k\ge 0}(-1)^k2C_{2k}t^k\right).
\end{equation*}


The case $g = 1$ is more involved. We want to rewrite $\lambda_1$ in terms of 
$\mathbb{Q}$-classes of boundary strata on a moduli space of curves. 
However, $\md_1$ is not stable, so we add a marked point. By the divisor equation,
\begin{equation*}
R_{1,h} = \int_{[\md_{1,0}(S,\beta)]^{\red}}(-1)^1\lambda_1 = \int_{[\md_{1,1}(S,\beta)]^{\red}}(-1)^1\lambda_1\ev_1^*(\beta^{\vee}),
\end{equation*}
where $\beta^{\vee}\cdot\beta = 1$.

Now, let $\delta_0\in H^2(\md_{1,1})$ denote the $\mathbb{Q}$-class 
$[\Delta_0]$, where $\Delta_0$ is the boundary locus of genus $0$ curves with one node (and one marked point), which is just a single point. Since
 $$\lambda_1 = \frac{1}{12}\delta_0\ ,$$
 we can remove the $\lambda_1$ insertion, restrict to maps from $\Delta_0$, and resolve the node to obtain
\begin{multline*}
\int_{[\md_{1,1}(S,\beta)]^{\red}}(-1)^1\lambda_1\ev_1^*(\beta^{\vee}) = \\
-\frac{1}{12}\cdot\frac{1}{2}\int_{[\md_{0,3}(S,\beta)]^{\red}}\ev_1^*(\beta^{\vee})(\ev_2\times\ev_3)^*(\mathsf{D}),
\end{multline*}
where $\mathsf{D}\in H^4(S\times S)$ is the Poincar\'{e} dual of 
the diagonal embedding of $S$ in $S\times S$. The extra factor of $\frac{1}{2}$ 
appears because there are two different ways of labeling the two new marked points.
 If we choose a basis 
$$\gamma_0=1, \gamma_1,\ldots,\gamma_{22}\in H^2(S), \gamma_{23}=\mathsf{p}$$ 
for the cohomology of the $K3$ surface $S$, 
with dual basis $\{\gamma^{\vee}_i\}$, then
 $$\mathsf{D} = \sum_{i = 0}^{23}\gamma_i\times\gamma^{\vee}_i\ .$$
 Now, the genus zero invariants involving pull-backs of $\gamma_{23} = \mathsf{p}$ 
all vanish for reasons of dimension.
Therefore, we only have the terms
\begin{equation*}
-\frac{1}{24}\sum_{i=1}^{22}\int_{[\md_{0,3}(S,\beta)]^{\red}}\ev_1^*(\beta^{\vee})\ev_2^*(\gamma_i)\ev_3^*(\gamma_i^{\vee})\,.
\end{equation*}
Applying the divisor equation again yields
\begin{equation*}
-\frac{1}{24}\sum_{i=1}^{22}(\beta\cdot\gamma_i)(\beta\cdot\gamma_i^{\vee})\int_{[\md_{0,0}(S,\beta)]^{\red}}1\,.
\end{equation*}
The integral on the right is simply $R_{0,h}$, so
\begin{align*}
R_{1,h} &= -\frac{1}{24}(\beta\cdot\beta)R_{0,h} \\ &= -\frac{h-1}{12}[q^{h-1}]\frac{1}{\Delta(q)} \\ &= [q^{h-1}]\left(-\frac{1}{12}q\frac{d}{dq}\left(\frac{1}{\Delta(q)}\right)\right) \\ &= [q^{h-1}]\left(-\frac{2C_2}{\Delta(q)}\right),
\end{align*}
as predicted by the KKV conjecture.


In genus $2$,  we can write $\lambda_2$ in terms of boundary classes on $\md_2$. 
The relevant boundary strata are $\Delta_{00}$, 
the generic element of which is a genus $0$ curve with $2$ nodes, and $\Delta_{01}$, where the generic element is a genus $0$ curve with $1$ node intersecting a smooth genus $1$ curve 
in a single point. The corresponding $\mathbb{Q}$-classes 
are $\delta_{00},\delta_{01}\in H^4(\md_2)$. The relation
\begin{equation*}
\lambda_2 = \frac{1}{120}(\delta_{00}+\delta_{01})\,
\end{equation*} 
is well-known,
see \cite[Section 8]{mumford}.

Again, we can replace $\lambda_2$ by the classes $\delta_{00}$ and $\delta_{01}$ and then remove these classes by restricting to maps from curves in the corresponding boundary loci. 
After resolving the singularities of the source curves, we find
\begin{equation*}
\begin{split}
R_{2,h} &= \int_{[\md_{2}(S,\beta)]^{\red}}(-1)^2\lambda_2 \\ 
&= \frac{1}{120}\cdot\frac{1}{8}\int_{[\md_{0,4}(S,\beta)]^{\red}}
(\ev_1\times\ev_2)^*(\mathsf{D})(\ev_3\times\ev_4)^*(\mathsf{D}) 
\\ &\thinspace\thinspace
+ \frac{1}{120}\cdot\frac{1}{2}\int_{[\md_{1,1}(S,\beta)]^{\red}\times[\md_{0,3}(S,0)]^{\vir}}(\ev_1\times\ev_2)^*(\mathsf{D})(\ev_3\times\ev_4)^*(\mathsf{D}).
\end{split} 
\end{equation*}

In the second term, the curve class $\beta$ cannot split nontrivially between 
the two irreducible components because $\beta$ is primitive.
One of the two components must be contracted to a point. In fact,
only the rational component may be contracted because the
 moduli space $\md_{0,3}$ has dimension $0$.

We now compute the two terms of $R_{2,h}$. 
The first term is completely analogous to the calculation in genus $1$. We obtain
\begin{align*}
\frac{1}{120}\cdot\frac{1}{8}&\int_{[\md_{0,4}(S,\beta)]^{\red}}
(\ev_1\times\ev_2)^*(\mathsf{D})(\ev_3\times\ev_4)^*(\mathsf{D}) \\ &
= \frac{1}{960}(2h-2)^2[q^{h-1}]\frac{1}{\Delta(q)} \\
&= [q^{h-1}]\frac{1}{240}\left(q\frac{d}{dq}\right)^2\left(\frac{1}{\Delta(q)}\right) \\
&= [q^{h-1}]\frac{\frac{11}{5}C_2^2+C_4}{\Delta(q)}\,.
\end{align*}

For the second term we have $\md_{0,3}(S,0) = S$ and 
$$(\ev_3\times\ev_4)^*(\mathsf{D}) = 24\mathsf{p}\ ,$$ so the integral reduces to
\begin{equation*}
\frac{1}{10}\int_{[\md_{1,1}(S,\beta)]^{\red}}\ev^*(\mathsf{p})\,.
\end{equation*} 
Using (\ref{stationary formula}), we find
\begin{equation*}
\frac{1}{10}[q^{h-1}]\frac{T_0}{\Delta} = [q^{h-1}]\frac{-\frac{1}{5}C_2^2+C_4}{\Delta(q)}\,.
\end{equation*}
Adding the two terms of $R_{2,h}$ yields
\begin{equation*}
R_{2,h} = [q^{h-1}]\frac{2C_2^2+2C_4}{\Delta(q)}\,,
\end{equation*}
which agrees with the KKV conjecture.


The genus $3$ case is significantly more complicated. To start with, 
the tautological cohomology space containing $\lambda_3$ on $\md_3$ has rank $10$.
A basis has been found by 
 Faber in \cite{faber1}:
 $9$ of the $10$ generators can be chosen to be $\mathbb{Q}$-classes 
corresponding to boundary strata $(a),(b),\ldots,(i)$ depicted in Figure 6 of \cite{faber1}. 
For the last generator, we let $[(j)]_\Q$ be the $\mathbb{Q}$-class 
corresponding to a genus $0$ curve with $1$ node intersecting a smooth genus $2$ curve 
at a point, with a cotangent line class at the intersection point. 
Then, we can write 
\begin{equation*}
\begin{split}
(-1)^3\lambda_3 = -\frac{1}{504}\Big(&\frac{1}{2}[(a)]_\Q+[(b)]_\Q+[(c)]_\Q
+\frac{3}{10}[(d)]_\Q-\frac{2}{5}[(f)]_\Q \\
&+2[(g)]_\Q+2[(j)]_\Q\Big),
\end{split}
\end{equation*}
see \cite{faber2}.

Through arguments similar to those used in the genus 2 calculation, 
we can show that the integrals of all of these classes vanish except for those 
of $[(a)]_\Q,[(d)]_\Q,[(e)]_\Q$, and $[(j)]_\Q$. Since $[(e)]_\Q$ does not appear 
in the above formula for $\lambda_3$,  we need to calculate only three integrals.

First, the class $[(a)]_\Q$ can be handled analogously to $\delta_{00}$ in the genus $2$ case, since $(a)$ is just the locus of genus $0$ curves with $3$ nodes. We calculate:
\begin{align*}
\int_{[\md_{3,0}(S,\beta)]^{\red}}[(a)]_{\mathbb{Q}} &= \frac{1}{48}(\beta^2)^3[q^{h-1}]\frac{1}{\Delta(q)} \\
&= [q^{h-1}]\frac{1}{6}\left(q\frac{d}{dq}\right)^3\left(\frac{1}{\Delta(q)}\right) \\
&= [q^{h-1}]\frac{1760C_2^3+2400C_2C_4+840C_6}{\Delta(q)}\,.
\end{align*}

The class $[(d)]_\Q$ is similarly obtained by adding a node to the genus $2$ case $\delta_{01}$, so we can compute
\begin{align*}
&\int_{[\md_{3,0}(S,\beta)]^{\red}}[(d)]_\Q \\ &= \frac{1}{4}\int_{[\md_{1,3}(S,\beta)]^{\red}\times S}(\ev_1\times\ev_2)^*(\mathsf{D})(\ev_3\times\id)^*(\mathsf{D})(\id\times\id)^*(\mathsf{D}) \\
&= \frac{1}{4}\cdot 24(2h-2)\int_{[\md_{1,1}(S,\beta)]^{\red}}\tau_0(\mathsf{p}) \\
&= [q^{h-1}]12q\frac{d}{dq}\left(\frac{q\frac{d}{dq} C_2}{\Delta(q)}\right) \\
&= [q^{h-1}]\frac{-480C_2^3+1440C_2C_4+2520C_6}{\Delta(q)}\,.
\end{align*}
Finally, we calculate the integral of the $\psi$-class $[(j)]_\Q$:
\begin{align*}
\int_{[\md_{3,0}(S,\beta)]^{\red}}[(j)]_\Q &= \frac{1}{2}\int_{[\md_{2,1}(S,\beta)]^{\red}\times S}\psi_1(\ev_1\times\id)^*(\mathsf{D})(\id\times\id)^*(\mathsf{D}) \\
&= \frac{1}{2}\cdot 24\int_{[\md_{2,1}(S,\beta)]^{\red}}\tau_1(\mathsf{p}) \\
&= [q^{h-1}]\frac{12T_1}{\Delta(q)} \\
&= [q^{h-1}]\frac{-32C_2^3+192C_2C_4-84C_6}{\Delta(q)}\,.
\end{align*}

We now use the boundary formula for $\lambda_3$
 and the above three calculations to obtain
\begin{align*}
R_{3,h} &= \int_{[\md_{3,0}(S,\beta)]^{\red}}(-1)^3\lambda_3 \\ &= -\frac{1}{504}[q^{h-1}]\frac{1}{\Delta(q)}\Big(\frac{1}{2}(1760C_2^3+2400C_2C_4+840C_6) \\ &\thinspace\thinspace
\hspace{+80pt}
+\frac{3}{10}(-480C_2^3+1440C_2C_4+2520C_6)
\\ &\thinspace\thinspace \hspace{+80pt}
+2(-32C_2^3+192C_2C_4-84C_6)\Big) \\
&= [q^{h-1}]\left(-\frac{\frac{4}{3}C_2^3+4C_2C_4+2C_6}{\Delta(q)}\right),
\end{align*}
as predicted by the KKV conjecture.

In higher genus, the boundary expressions for $\lambda_g$ will likely
lead to nonstationary descendent invariants of $K3$ surfaces.
Using Theorem 4, the calculations can be, in principle, continued.
An interesting question is how the KKV conjecture, proven in Theorem 1,
constrains the possible boundary expressions for $\lambda_g$.

\vspace{+8 pt}
\noindent
Department of Mathematics \\
Massachusetts Institute of Technology \\
dmaulik@math.mit.edu

\vspace{+8 pt}
\noindent
Department of Mathematics\\
Princeton University\\
rahulp@math.princeton.edu

\vspace{+8 pt}
\noindent
Department of Mathematics \\
Imperial College \\
richard.thomas@imperial.ac.uk

\vspace{+8 pt}
\noindent
Department of Mathematics\\
Princeton University\\
apixton@math.princeton.edu


\begin{thebibliography}{99}


\bibitem{beu} A. Beauville, {\em Counting rational curves on $K3$ surfaces},
Duke Math. J. {\bf 97}, 99--108, 1999. alg-geom/9701019.



\bibitem{bdt}
K.~Behrend,
\newblock {\em {D}onaldson-{T}homas invariants via microlocal geometry,}
Ann. Math. \textbf{170}, 1307–-1338, 2009.
\newblock math.AG/0507523.






\bibitem{BF2}
K.~Behrend and B.~Fantechi.
\newblock {\em Symmetric obstruction theories and
Hilbert schemes of points on threefolds}, Algebra and Number Theory \textbf{2}, 313--345,
2008. math.AG/0512556.






\bibitem{BF} K.~Behrend and B.~Fantechi,
{\em The intrinsic normal cone}, Invent.\ Math.\
{\bf 128}, 45--88, 1997. alg-geom/9601010.





\bibitem{bghz} J. Bruinier, G. van der Geer, G. Harder, and D. Zagier,
{\em The 1-2-3 of modular forms}, Springer Verlag: Berlin, 2008.



\bibitem{brl} J. Bryan and C. Leung, {\em The enumerative geometry of
K3 surfaces and modular forms}, J. AMS {\bf 13}, 371--410, 2000. alg-geom/9711031.



\bibitem{BryP} J. Bryan and R. Pandharipande, {\em Local Gromov-Witten
theory of curves}, J. AMS {\bf 21}, 101--136, 2008. math.AG/0405204.


\bibitem{BuF}
R.-O.\ Buchweitz and H.\ Flenner.
\emph{A semiregularity map for modules and applications to deformations},
Compositio Math. {\bf 137}, 135--210, 2003. math.AG/9912245.

 
\bibitem{dolga} I. Dolgachev and S. Kondo, {\em Moduli of $K3$
surfaces and complex ball quotients}, Lectures in Istambul, math.AG/0511051.

\bibitem{faber1}
C.~Faber, \emph{Chow rings of moduli spaces of curves. {I}. {T}he {C}how ring
  of {$\overline{\mathcal{M}}_3$}}, Ann. of Math. \textbf{132}, 331--419, 1990.

\bibitem{faber2}
C.~Faber, \emph{Intersection-theoretical computations on {$\overline{\mathcal{
  M}}_g$}}, Parameter spaces (Warsaw, 1994), Banach Center Publ., vol.~36,
  Polish Acad. Sci., Warsaw, 1996, 71--81. alg-geom/9504005.

\bibitem{FP}  C. Faber and R. Pandharipande,
\emph{Hodge integrals and Gromov-Witten theory}, Invent.\ Math.\
{\bf 139}, 173--199, 2000. math.AG/9810173.

\bibitem{fpm} C. Faber and R. Pandharipande,
{\em Relative maps and tautological classes}, J. EMS {\bf 7}, 13--49, 2005.
math.AG/0304485.


\bibitem{GP} T. Graber and R. Pandharipande, {\em Localization
of virtual classes}, Invent. Math. {\bf 135}, 487--518, 1999. alg-geom/9708001.



\bibitem{HT}
D.~Huybrechts and R.~P.~Thomas.
\newblock {\em Deformation-obstruction
theory for complexes via Atiyah and Kodaira--Spencer classes},
Math. Ann. \textbf{346}, 2009.
arXiv:0805.3527.


\bibitem{joys} D. Joyce and Y. Song, {\em A theory of
generalized Donaldson-Thomas invariants}, arXiv:0810.5645.


\bibitem{kkv} S. Katz, A. Klemm, C Vafa, {\em M-theory, topological
strings, and spinning black holes}, Adv. Theor. Math. Phys. {\bf 3},
1445--1537, 1999. hep-th/9910181.


\bibitem{ky} T. Kawai and K Yoshioka, {\em String partition functions
and infinite products}, Adv. Theor. Math. Phys. {\bf 4}, 397--485, 2000.
hep-th/0002169.

\bibitem{KL} Y.-H. Kiem and J. Li, {\em Gromov-Witten
invariants of varieties with holomorphic 2-forms},
arXiv:0707.2986.

\bibitem{KL2} Y.-H. Kiem and J. Li, {\em Localizing virtual cycles by 
cosections}, arXiv:1007.3085.

\bibitem{gwyz} A. Klemm, D. Maulik, R. Pandharipande, and E. Scheidegger,
{\em Noether-Lefschetz theory and the Yau-Zaslow conjecture}, arXiv:0807.2477.


\bibitem{LeP}
J.~Le~Potier.
\newblock {\em Syst\`emes coh\'erents et structures de niveau}, Ast\'erisque
  {\bf 214}, 1993.


\bibitem{jleec1} J. Lee and C. Leung, {\em Yau-Zaslow formula
on $K3$ surfaces for non-primitive
classes}, Geom. \& Top. {\bf{9}}, 1977--2012, 2005. math.SG/0404537.


\bibitem{jleec2} J. Lee and C. Leung, {\em Counting elliptic curves
in K3 surfaces}, 
J. Alg. Geom.  \textbf{15}, 591--601, 2006. math.SG/0405041.


\bibitem{junli1} J. Li, {\em
Stable morphisms to singular schemes and relative
stable morphisms}, J. Diff. Geom. {\bf 57}, 509--578, 2001. 
math.AG/0009097.

\bibitem{junli2} J. Li, {\em
A degeneration formula for Gromov-Witten invariants}
 J. Diff. Geom. {\bf 60}, 199--293, 2002. math.AG/0110113.



\bibitem{LiTian} J.~Li and G.~Tian, 
{\em Virtual moduli cycles and Gromov-Witten invariants
of algebraic varieties}, 
J. AMS {\bf 11}, 119--174, 1998. alg-geom/9602007.

\bibitem{LiWu}
J.~Li and B.~Wu.
\newblock {\em Degeneration of {D}onaldson-{T}homas invariants}, 
preprint 2009.

\bibitem{Man}
M. Manetti,
\emph{Lie description of higher obstructions to deforming submanifolds}, Ann. Sc. Norm. Super. Pisa Cl. Sci. (5) \textbf{6}, 631--659, 2007. math.AG/0507278.

\bibitem{dmt}
D.~Maulik, {\em Gromov-{W}itten theory 
of $A_n$-resolutions}, arXiv:0802.2681.

\bibitem{mnop2}
D.~Maulik, N.~Nekrasov, A.~Okounkov, and R.~Pandharipande,
\newblock {\em Gromov-{W}itten theory and {D}onaldson-{T}homas theory. {II}},
  Compos. Math. {\bf 142}, 1286--1304, 2006. math.AG/0406092.


\bibitem{mo1}
D.~Maulik and A.~Oblomkov, {\em Quantum cohomology of
the Hilbert scheme of points of
$A_n$-resolutions}, J. AMS \textbf{22}, 1055--1091, 2009. arXiv:0802.2737.


\bibitem{mo2}
D.~Maulik and A.~Oblomkov, {\em Donaldson-Thomas theory of
$A_n\times {\mathbb{P}}^1$}, Compos. Math. \textbf{145}, 1249--1276, 2009. arXiv:0802.2739.



\bibitem{moop}
D.~Maulik, A.~Oblomkov, A.~Okounkov, and R.~Pandharipande, {\em
Gromov-Witten/Donaldson-Thomas correspondence for toric 3-folds},
arXiv:0809.3976.



\bibitem{mptop}  D. Maulik and R. Pandharipande, {\em A topological
view of Gromov-Witten theory}, Topology {\bf 45}, 887--918, 2006. math.AG/0412503.



\bibitem{gwnl} D. Maulik and R. Pandharipande, {\em Gromov-Witten theory and
Noether-Lefschetz theory}, arXiv:0705.1653.

\bibitem{mpt2} D. Maulik, 
R. Pandharipande, and R. P. Thomas, {\em in preparation}.


\bibitem{mumford}
D.~Mumford, \emph{Towards an enumerative geometry of the moduli space of
  curves}, Arithmetic and geometry, Vol. II, Progr. Math., vol.~36,
  Birkh\"auser Boston, Boston, MA, 1983, 271--328.


\bibitem{gwhurwitz} A. Okounkov and R. Pandharipande, 
{\em Gromov-Witten theory,
Hurwitz numbers, and completed cycles}, Ann. of Math {\bf{163}},
517--560, 2006. math.AG/0204305.

\bibitem{okpanp1} A. Okounkov and R. Pandharipande, {\em The equivariant
Gromov-Witten theory of ${\mathbf P}^1$}, Ann. of Math {\bf{163}},
561--605, 2006. math.AG/0207233.



\bibitem{okpanvir} A. Okounkov and R. Pandharipande, {\em Virasoro constraints
for target curves}, Invent. Math. {\bf 163}, 47--108, 2006. math.AG/0308097.


\bibitem{OP6}
A.~Okounkov and R.~Pandharipande,
{\em Quantum cohomology of the Hilbert scheme
of points in the plane}, Invent. Math. {\bf 179}, 523--557, 2010.
math.AG/0411210.


\bibitem{GWDT}
A.~Okounkov and R.~Pandharipande.
{\em The local Donaldson-Thomas theory of curves},
 math.AG/0512573.



\bibitem{icm} R. Pandharipande, {\em Three questions in Gromov-Witten theory},
Proceedings of the ICM (Beijing 2002), Vol. II, 503--512. math.AG/0302077.



\bibitem{clay} R. Pandharipande, {\em Maps, sheaves, and
$K3$ surfaces}, arXiv:0808.0253.


\bibitem{pt1}
R.~Pandharipande and R.~P. Thomas,
\newblock {\em Curve counting via stable pairs in the derived category},
Invent. Math. \textbf{178}, 407--447, 2009.
\newblock arXiv:0707.2348.

\bibitem{pt2} 
R.~Pandharipande and R.~P. Thomas,
\newblock{\em The 3-fold vertex via stable pairs},
Geom. \& Top. \textbf{13}, 1835--1876, 2009.
\newblock arXiv:0709.3823.

\bibitem{pt3} 
R.~Pandharipande and R.~P. Thomas,
\newblock{\em Stable pairs and BPS invariants},
J. AMS \textbf{23}, 267--297, 2010. \newblock arXiv:0711.3899.

\bibitem{1form}
R.~Pandharipande and R.~P.~Thomas,
\newblock{\em Almost closed 1-forms}, arXiv:1204.3958.

\bibitem{pix} A. Pixton, {\em Senior thesis}, Princeton 2008.

\bibitem{Ran}
Z. Ran, \emph{Semiregularity, obstructions and deformations of Hodge classes}, Ann.
Scuola Norm. Sup. Pisa Cl. Sci. (4) \textbf{28}, 809-–820, 1999.

\bibitem{Sa}
K.~Saito.
\newblock {\em Quasihomogene isolierte {S}ingularit\"aten von {H}yperfl\"achen},
Invent. Math. {\bf 14}, 123--142, 1971.

\bibitem{yauz} S.-T. Yau and E. Zaslow, {\em BPS states, string duality, and
nodal curves on $K3$}, Nucl. Phys. {\bf B457}, 484--512, 1995. hep-th/9512121.



\end{thebibliography}
\end{document}